\documentclass[11pt,a4paper]{article}

\usepackage{amsmath, latexsym, amsfonts, mathrsfs, ulem, amsthm}
\usepackage{amscd}
\usepackage{amssymb}
\usepackage{indentfirst}
\usepackage{multicol}
\usepackage{color}
\usepackage{todonotes}
\usepackage[toc,page]{appendix}
\usepackage{comment}

\usepackage{hyperref}
\hypersetup{
    colorlinks=true,
    linkcolor=blue,
    filecolor=magenta,
    urlcolor=cyan,
}

\usepackage[mathscr]{euscript}
\usepackage{dsfont}

\newcommand{\1}{\mathds{1}}

\allowdisplaybreaks
\numberwithin{equation}{section}

\newcommand{\bv}{\mathbf{v}}
\newcommand{\bm}{\mathbf{m}}
\newcommand{\tOmega}{\widetilde\Omega}
\newcommand{\bpsi}{\boldsymbol{\psi}}

\newcommand{\E}{\mathbb{E}}
\newcommand{\FT}{\mathcal{F}}
\newcommand{\tF}{\widetilde{\mathcal{F}}}
\newcommand{\tP}{\widetilde{\mathbb{P}}}

\newcommand{\dom}{\mathcal{O}}
\newcommand{\vf}{{\vc{f}}}
\newcommand{\vw}{{\vc{w}}}
\newcommand{\vr}{\varrho}
\newcommand{\vrn}{\vr_n}
\newcommand{\vun}{\vu_n}
\newcommand{\vt}{\vartheta}
\newcommand{\vu}{\vc{u}}
\newcommand{\vm}{\vc{m}}
\newcommand{\vq}{\vc{q}}
\newcommand{\vrr}{\vc{r}}
\newcommand{\vc}[1]{{\bf #1}}
\newcommand{\eps}{\varepsilon}
\newcommand{\Law}{{\mathcal L}}
\newcommand{\N}{\mathbb{N}}
\newcommand{\Prob}{\mathbb{P}}
\newcommand{\A}{{\mathbb{A}}}
\newcommand{\ep}{\varepsilon}
\newcommand{\embed}{\hookrightarrow}
\newcommand{\dd}{\delta}
\font\F=msbm10 scaled 1200
\newcommand{\R}{\mbox{\F R}}
\newcommand{\T}{\mbox{\F T}}
\newcommand{\wt}{\widetilde}

\newcommand{\Xr}{\mathcal{X}_{\vr}}
\newcommand{\Xsr}{\mathcal{X}_{\sqrt{\vr}}}
\newcommand{\Xru}{\mathcal{X}_{\vr\vu}}
\newcommand{\Xrdu}{\mathcal{X}_{\vr^{\frac{1}{2+\delta}}\vu}}

\newcommand{\lap}{\Delta}
\newcommand{\Div}{\operatorname{div}}
\newcommand{\Grad}{\nabla}
\newcommand{\pt}{\partial_{t}}
\newcommand{\Dt}{\frac{{\rm d}} {{\rm dt}}}
\newcommand{\dx}{{\rm d} {x}}
\newcommand{\de}{{\rm d}}
\newcommand{\ds}{{\rm d} s }
\newcommand{\dt}{{\rm d} t }
\newcommand{\dxds}{\dx \ \ds}
\newcommand{\dxdt}{\dx \ \dt}

\newcommand{\intO}[1]{\int_{\dom} #1 \ \dx}
\newcommand{\lr}[1]{\left( #1 \right)}
\newcommand{\intOB}[1]{\int_{\dom} \left( #1 \right) \ \dx}
\newcommand{\intTO}[1]{\int_0^T\!\!\!\! \int_{\dom} #1 \ \dxdt}
\newcommand{\inttO}[1]{\int_0^t\!\!\!\int_{\dom} #1 \ \dxds}

\newtheorem{theorem}{Theorem}[section]
\newtheorem{lemma}[theorem]{Lemma}
\newtheorem{proposition}[theorem]{Proposition}
\newtheorem{definition}[theorem]{Definition}
\newtheorem{remark}[theorem]{Remark}
\newtheorem{corollary}[theorem]{Corollary}
\newtheorem{assumption}[theorem]{Assumption}

\newcommand{\eq}[1]{\begin{equation}
\begin{split}
#1
\end{split}
\end{equation}}
\newcommand{\eqh}[1]{\begin{equation*}
\begin{split}
#1
\end{split}
\end{equation*}}

\newcommand\dela[1]{}

\definecolor{darkblue}{rgb}{0.1,0.1,0.9}

\begin{document}
\title{Sequential stability of weak martingale solutions \\
to stochastic compressible Navier-Stokes equations \\
with viscosity vanishing on vacuum}

\author{Zdzis\l aw Brze\'zniak$^*\;$,
Gaurav Dhariwal$^\dagger\;$,
Ewelina Zatorska$^\ddagger\;$
}

\date{\today}
\maketitle

{
\footnotesize
\centerline{$^*\;$Department of Mathematics, University of York}
\centerline{Heslington, York, YO105DD, UK}
\centerline{\small \texttt{zdzislaw.brzezniak@york.ac.uk}}

\bigbreak
\centerline{$^\dagger\;$Lloyds Banking Group}
\centerline{10 Gresham Street, EC2V 7JD, London, UK}
\centerline{\small \texttt{dhariwalgaurav90@gmail.com}}

\bigbreak
\centerline{$^\ddagger\;$Department of Mathematics, Imperial College London}
\centerline{South Kensington Campus -- SW7 2AZ, London, UK}
\centerline{\small \texttt{e.zatorska@imperial.ac.uk}}

}

\bigbreak

\begin{abstract}

In this paper, we investigate  the compressible Navier-Stokes equations with degenerate, density-dependent, viscosity coefficient driven by multiplicative stochastic noise. We consider  three-dimensional periodic domain and prove that the family of weak martingale solutions is sequentially compact.

\end{abstract}

\section{Introduction}
\noindent In this article we investigate the system of equations describing the flow of  a compressible fluid in the 3-dimensional domain with the periodic boundary condition.
The dynamics of such fluid is  characterised by the total mass density $\vr=\vr(t,x)$ and  the velocity vector field $\vu=\vu(t,x)$. The following equations express the physical laws of conservation of mass and momentum, respectively
 \begin{equation}\label{eqn-1.1}
 \left\{
\begin{array}{rl}
\vspace{0.2cm}
\pt\vr &+\Div (\vr \vu) = 0,\\
\vspace{0.2cm}
\de \lr{\vr\vu} &+\left[\Div (\vr \vu \otimes \vu) - \Div (\vr \Grad\vu) + \Grad p\right]\de t =\vr\vf \de W\end{array}\right.
{\rm in}\quad(0,T)\times\dom,
\end{equation}
Here by $\dom = \T^3$ we denote the $3$-dimensional torus,
  by $\vr\vf \de W$  special multiplicative noise  -- a stochastic external force, and by $p$  the pressure  depending on  the density $\vr$ via the following formula:
\begin{equation}\label{eqn:p}
p=p(\vr)=\vr^\gamma, \;\; \gamma \in (1,3).
\end{equation}
\dela{for some fixed parameter $\gamma$ satisfying the following condition
 \begin{equation}\label{eqn:gamma}
\gamma \in (1,3).
\end{equation}}

The aim of this paper is to prove the stability  of weak martingale solutions to system \eqref{eqn-1.1}.
We refer to this system as the stochastic compressible \textit{degenerate} Navier-Stokes equations, in contrast to the \textit{classical} stochastic  compressible Navier-Stokes equations
in which the stress tensor is independent  of the density.
The first result concerning the existence of global in time solutions to the stochastically perturbed compressible, viscous  multidimensional system,  was actually given for the degenerate case,
see Tornatore \cite{To00}. The author of that paper  considered a specific, rather un-physical, choice of the stress tensor, studied earlier in the deterministic setting by Vaigant and Kazhikhov \cite{VK95}.  More than decade later, Feireisl, Maslowski and Novotn\'y  in \cite{Fei+MaNo13}   investigated the stochastic classical $3$-d system.
Since then the study of stochastically driven compressible Navier-Stokes equations took off for even more general multiplicative noise $\Phi(\vr,\vr\vu)\,\de W$. This research was initiated by Breit and Hofmanova in \cite{Breit+Ho16}, and continued in collaboration with Feireisl in a series of papers  \cite{Breit+Feir+Ho18art,Breit+Feir+Ho17,Breit+Feir+Ho16}.
We  refer to their book \cite{Breit+Feir+Ho18}, for a complete account of the mathematical literature on that system, including  some singular limits results, and to  paper \cite{Breit+Feir+HoMa19} devoted to the existence of stationary solutions.
The underlying theory for all of these contributions is  the one developed for the classical deterministic system. Here the literature is much  broader and so we refer only to the first major contributions, i.e.  the pioneering work of by Lions \cite{Lio98} for compressible barotropic Navier-Stokes with  $\gamma>\frac{9}{5}$  and the proof of the existence of solutions covering more physical cases $\gamma\geq\frac53$ by Feireisl \cite{Fei_2001}. The overview of these methods can be found in the monograph \cite{NoSt04} by Novotny and Str\v{a}skraba.

\medskip
Let us also mention that there is an abundant amount of literature corresponding to the stochastically driven  incompressible fluids. An interested reader is referred to the original paper  \cite{Flandoli+Gat_1995} by Flandoli and G{\c a}tarek, as well as three more recent publications \cite{Bianchi+Flandoli_2020,Brz+Ferr_2019,Brz+Mot+Ondr_2017} and the references therein). The main idea of the above cited  papers is to find  some suitable a priori estimates for solutions to  appropriate  approximating problems,  then  to use appropriate  compactness results, more classical in \cite{Flandoli+Gat_1995} and  less classical, based on the weak topologies, in \cite{Brz+Ferr_2019,Brz+Mot+Ondr_2017}, to pass to the limit and  finally to identify the limit as a solution to the investigated problem.  The  paper \cite{Brz+Mot+Ondr_2017} relies on a novel approach to proving the passage to the limit based on  the Skorokhod-Jakubowski Theorem \cite{Jak97} which is a generalization to a large class of non-metric spaces of the classical Skorokhod Theorem from \cite{Skorokhod_1965}.

\medskip

For the compressible, degenerate  Navier-Stokes system the literature is much more sparse, and for the stochastic version, apart from the mentioned result of Tornatore \cite{To00}, none. For the deterministic version of system  \eqref{eqn-1.1}, i.e. when $\de W=1$, the sequential stability of weak solutions was proven by Mellet and  Vasseur \cite{Me+Va07}. The complete existence result was achieved nearly ten years later by Vasseur and Yu \cite{VaYu16, VaYu16b}, after couple of other attempts including approximation by cold pressure, drag terms or quantum force \cite{MuPoZa15, BrDe06}.
The main difficulty concerning systems with the viscosity coefficients vanishing when density equals $0$, is that the velocity vector field is  no longer  defined. This degeneracy leads to some further problems in the analysis, when compared to the constant-viscosity case of Lions \cite{Lio98} and Feireisl \cite{Fei_2001}. However, for certain forms of density-dependent viscosities, this degeneracy proves to be beneficial. Namely, it provides  a particular mathematical structure, called the BD-entropy inequality, that yields a global in time integrability of  $\Grad\varphi(\vr)$ for an appropriate increasing function $\varphi$. This mathematical structure was introduced for the first time by Bresch, Desjardins $\&$ Lin \cite{BrDeLi03} for the Korteweg equations. In the following work  \cite{BrDe03} by Bresch and Desjardins, the same concept was applied to 2-dimensional viscous shallow water model, without capillarity but with additional drag force. The contribution of Mellet and Vasseur \cite{Me+Va07} goes one step further and solves a problem of passing to the limit in the convective term without any further regularisations.
These authors  combined the BD-entropy method with an additional energy-type estimate   $\vr|\vu|^2$ in the space $L^\infty(0,T; L\log L(\dom))$.

\medskip

The present paper is concerned with the sequential stability  theory for the stochastic compressible Navier-Stokes Equations with density-dependent stress tensor
$\mathbb{S}(\vr,\nabla\vu)=\vr \Grad\vu$. To the best of our knowledge, this work is the first attempt to extend the deterministic results to the stochastic setting.
Our proof of the sequential stability of the weak martingale solutions  provides an additional essential insight into the the compactness argument.
The proof of the existence of those solutions, presumably lengthy and technical, will be considered in the future.

In simple terms, our main result says that given a sequence of martingale solutions satisfying given uniform bounds, there exists  a weak limit, possibly defined on a new probability space, which is also a martingale solution. The set of uniform bounds is a stochastic generalization of the deterministic ones from \cite{Me+Va07}, with the exception of the estimate of $\vr|\vu|^2\log(1+|\vu|^2)$. To make use of Mellet and Vasseur's idea we had to replace the symmetric part of the velocity gradient in the stress tensor by the full gradient.
Let us now shortly discuss the main difficulties and the content of the paper.

\medskip

The principle difficulty in proving convergence of solutions to problem \eqref{eqn-1.1} stems from only non-negativity of the density process $\vr$. Had this process been bounded from below by some constant $\eps>0$ one would be able to use the second equation in  \eqref{eqn-1.1} as the equation for the unknown $\vu$.
Because  $\vr$ can be equal to $0$ on a set of positive space-time Lebesgue measure, and because the velocity field $\vu$ is always preceded by the density function $\rho$, passage to the limit in nonlinear terms involving $\vu$ becomes particularly difficult. The corresponding   deterministic problem   has been investigated in a number of papers, starting from \cite{Me+Va05} by Mellet and Vasseur.

The present article will also use the approach from the deterministic variant of the problem to some extent. To be precise we first use the  following classical  functions associated to a solution $(\vr,\vu)$.
\begin{align}
\label{eqn-auxilary function-energy}
\intOB{{\frac12}\vr(t,x)|\vu(t,x)|^2+\frac{1}{\gamma - 1} \vr^\gamma(t,x)} & \hspace{1truecm}\mbox{ the energy}
\\
\label{eqn-auxilary function-B-D enstrophy}
\intOB{{\frac12}\vr(t,x)|\vu(t,x)+\Grad\log\vr(t,x)|^2+\frac{1}{\gamma - 1} \vr^\gamma(t,x)}
&\hspace{1truecm}\mbox{ the B-D enstrophy}
\\
\label{eqn-auxilary function-M-V energy}
 \intO{\frac{1}{2+\delta} \vr(t,x)|\vu(t,x)|^{2+\delta}}
& \hspace{1truecm}\mbox{ the M-V energy}
\end{align}
for some $\delta \in (0,1)$, as a source of the a'priori estimates on the solutions. For the convenience of the reader, we present all details of these a-priori estimates in the stochastic setting, assuming that the equations have regular enough solutions to justify all the passages, see Lemmas \ref{lem-energy estimate stoch}, \ref{lem-entropy stoch Bresch-Desjardins}, \ref{lem-estimates} and
and \ref{lem-Mellet+Vasseur-stoch}.  However, in the next steps we cannot use these functions directly and simply apply the stochastic counterparts of the classical deterministic compactness theorems.
Instead we introduce the following  auxiliary functions
\begin{align}\label{eqn-auxilary functions-five}
 \vt_n:=\sqrt{\vr_n}, \hspace{0.3truecm} \vm_n:=\vr_n \vu_n, \hspace{0.3truecm}  \vq_n:=\sqrt{\vr_n} \vu_n,  \hspace{0.3truecm}  \vrr_n:=\vr_n^{\frac{1}{2+\delta}} \vu_n
\end{align}
and
\begin{align}\label{eqn-mu_n}
\mu_n&=\Bigl(\vr_n,   \vt_n, \vm_n , \vq_n , \vrr_n,  W\Bigr),
\end{align}
and using the just estimates we prove that the laws of the  processes $\mu_n$ are tight on some appropriately chosen functional space
$\mathcal{X}_T$ defined in  \eqref{eqn-spaces-X_T},
see Lemmata \ref{lem-tightness of rho_n}, \ref{lem-tightness of sqrt of rho_n}, \ref{lem-tightness of  rho_n u_n}, \ref{lem-tightness of  rho_n u_n in C_w} and Corollary~\ref{cor-tight mu_n}. Let us point out here that, contrary to the deterministic case, the tightness deduced from these a priori estimates is insufficient to conclude the proof of our result.
Fortunately 
we are able to employ the Jakukowski-Skorokhod Theorem
and deduce that there exists a subsequence $\bigl(\mu_{k_n}\bigr)_{n=1}$,
a new  probability space $(\tOmega,\tF,\tP)$,  a sequence $\bigl(\tilde{\mu}_n\bigr)_{n=1}^\infty$ of $\mathcal{X}_T$-valued random variables, as well as a $\mathcal{X}_T$-valued random variable $\tilde{\mu}=\Bigl(\tilde{\vr},   \tilde{\vt},\tilde{\vm },\tilde{\vq },\tilde{\vrr},\tilde{ W}\Bigr)$, all defined on the new  probability space,
and such that
(i)  $\mathcal{L}(\tilde{\mu}_n)= \mathcal{L}(\mu_{k_n})$, $n \in \N$ and (ii)   $\tilde{\mu}_n$ converges $\tP$-almost surely in $\mathcal{X}_T$ to $\tilde{\mu}$.
However, in the present setting,   the limiting random variable $\tilde{\vm}$ does not define a candidate for the velocity field $\tilde{\vu}$.  In fact,  even the approximate random variables $\tilde{\mu}_n$ do not keep track of their original structure. We solve this problem
by
proving
 that
 $\wt\Prob$-a.s.
\begin{align*}
(\widetilde\vt_n, \widetilde\vm_n,\widetilde{\vq}_n)&= (\sqrt{\widetilde\vr_n},  \widetilde\vr_n^{\frac{1+\delta}{2+\delta}}\widetilde \vrr_n,\widetilde\vr_n^{\frac{\delta}{2(2+\delta)}}\widetilde \vrr_n), \\
(\widetilde\vt, \widetilde\vm,\widetilde{\vq})&= (\sqrt{\widetilde\vr},  \widetilde\vr^{\frac{1+\delta}{2+\delta}}\widetilde \vrr,\widetilde\vr^{\frac{\delta}{2(2+\delta)}}\widetilde \vrr).
\end{align*}
The proof of  these two identities requires  measurability of the corresponding nonlinear maps. These two identities allow us  to define vector-valued random variables
$\wt\vr_n$ and $\widetilde{\vu}_n$  on our new probability space  such that $\Prob$-a.s.
 \begin{align*}
(\widetilde{\vq}_n,\widetilde{\vm}_n)&=(\sqrt{\wt\vr_n}\widetilde{\vu}_n,\wt\vr_n\widetilde{\vu}_n)
\\
\wt\vr_n\widetilde{\vu}_n\otimes\widetilde{\vu}_n\to \wt\vr\widetilde{\vu}\otimes\wt\vu.
 \end{align*}

Let us conclude the short introduction with an overview of the paper. The rigorous formulation of the assumptions and the main result are presented below, in Section \ref{sec-main}. Then, in Section \ref{sec-apriori} we recall deterministic a-priori estimates and prove their stochastic equivalents for smooth enough solutions. The actual proof of our the main result is the content of Section \ref{sec-existence}. In Section \ref{sec-martingale} we show that the limiting process is a martingale solution to Problem \eqref{eqn-1.1}.
Finally in the Appendices we recall some known facts about the measurability of functions, simple consequences of the H\"older inequality, and a handy collection of facts from stochastic analysis.

\section{The main result}\label{sec-main}
\noindent We consider system (\ref{eqn-1.1}) with the initial conditions
\begin{equation}\label{initial}
\vr(0,x)=\vr^{0}(x),\quad \vr\vu(0,x)=\vc{m}^{0}(x),
\end{equation}
satisfying the following assumption.
\begin{assumption} \label{ass-initial}
Assume that $\vr^0, \vc{m}^0$ are measurable functions from $\dom$ taking values in $\R$ or {$\R^3$} such that $\vr_0\ge 0$ and $\vm^0=\vc{0}$ on the set $\{\vr^0=0\}=\{x\in \dom: \vr^0(x)=0\}$.\\
 In addition, denoting $\vu^0=\frac{\vc{m}^0}{\vr^0}\1_{\{\vr^0>0\}}$ we assume that
 \begin{align}
\label{est2_ini}
&\E \left(\int_\dom \vr^0|\vu^0|^2 \,\dx\right)^p
+\E \left(\intO{\vr^0|\Grad\log\vr^0|^2} \right)^p
+\E \left(  \int_\dom \lr{\vr^0}^\gamma \,\dx\right)^p \le C \\
\label{est4-}
& \E\left( \frac{1}{2+\delta}\intO{\vr^0|\vu^0|^{2+\delta}} \right) \le C
\end{align}
for  $\delta \in (0,1)$,  $p\geq1$ and
some positive constant $C$.
\end{assumption}
Concerning the external force, we assume what follows.
\begin{assumption}\label{ass-force}
For every $p \ge 1$, {$\vf \in L^p(\Omega; L^{3}(\dom))$}  is $\mathcal{F}_0$-measurable.
\end{assumption}

Throughout the paper we will use the following definition of the solution to \eqref{eqn-1.1}.

\begin{definition}[Martingale solution]
\label{def-mart_sol} Let $T>0$ be arbitrary, let $\vr^0$, $\vc{m}^0$ satisfy Assumption \eqref{ass-initial} and $\vc{f}$ satisfy Assumption \eqref{ass-force}.\\
We say that the system $(U, W, {\vr}, \vu)$ is a martingale solution to the problem \eqref{eqn-1.1}, with the initial data $\vr^0$ and $\vm^0$ as described in \eqref{initial}, if and only if  $ U := ({\Omega},{\mathcal{F}}, {\mathbb{P}},{\mathbb{F}})$
is a stochastic basis with filtration ${\mathbb{F}}=({\mathcal{F}}_t)_{t\in[0,T]}$,
$W$ is a $\R$-valued Wiener process  on $U$, and
$\left({\vr},{\vu}\right)$
are {processes} such that
\begin{trivlist}
\item[(i)] ${\vr}$ is a continuous  ${\mathbb{F}}$-progressively measurable $L^\gamma(\dom)$-valued process,
\item[(ii)] there exists a weakly continuous ${\mathbb{F}}$-progressively measurable $L^{3/2}(\dom)$-valued process $\vc{m}$, such that  $\vc{m}=\vr\,\vu$ a.e. in $\dom \times (0,T)$ $\Prob$-a.s.,
\item[(iii)]  for all  $\phi \in W^{1,\infty}(\dom)$, $\bpsi\in W^{2,\infty}(\dom)$, $t \in [0,T]$, ${\mathbb{P}}$-a.s.
\begin{align}
    \label{eqn:weak_sense1}
     & \int_\dom \vr (t)\phi \,\dx  = \int_\dom \vr^0 \phi\,\dx + \inttO{ {\vr}{\vu} \cdot \nabla \phi },  \\
    \label{eqn:weak_sense2}
    & \int_\dom \vc{m}(t)\cdot \bpsi \,\dx  = \int_\dom \vc{m}^0\cdot \bpsi \,\dx + \inttO{ {\vr}{\vu} \otimes {\vu}\colon \nabla \bpsi } \nonumber \\
    &\hspace{3truecm} + \inttO{ \vr \vu\cdot \lap \bpsi} + \inttO{ \Grad\vr \otimes\vu:\Grad\bpsi} \nonumber \\
    & \hspace{3truecm}+ \inttO{ {\vr}^\gamma\Div \bpsi}
     + \int_0^t \!\!\!\int_\dom{\vr}\vf \cdot \bpsi \,\dx\, \mathrm{d} {W}.
\end{align}
\end{trivlist}
\end{definition}

Then, the aim of this work is to prove the sequential stability of the martingale  solutions of (\ref{eqn-1.1}). More precisely, we want to show that a sequence $\left(U_n, W_n, \vrn,\vun\right)_{n\in \N}$ of martingale solutions to \eqref{eqn-1.1}, satisfying Definition~\ref{def-mart_sol} and some uniform bounds specified below, with initial data $\vrn(0)=\vrn^{0}$ and $\vc{m}_n(0) = \vc{m}_{n}^{0}$, converges to a martingale solution $\left(\widetilde{U}, \widetilde{W}, \widetilde{\vr}, \widetilde{\vu}\right)$ to \eqref{eqn-1.1} in an appropriate sense.

We will also be using the following assumption on the boundedness of solution in the energy class to establish our main theorem.

\begin{assumption} \label{ass-energy}
There exists a positive constant $C$ such that for
any $\delta\in(0,1)$ and a corresponding $c_\delta > 0$, and for any $p \ge 1$
 the following bounds are satisfied uniformly in $n \in \N$:
\begin{align}
\label{est1}
&\E \left(\sup_{t\in[0,T]} \int_\dom \vr_n|\vu_n|^2 \,\dx\right)^p
+\E\lr{\sup_{t \in [0,T]} \int_\dom \vr_n|\Grad\log\vr_n|^2\,\dx}^p
+\E \left(\sup_{t \in[0,T]} \int_\dom \vr_n^\gamma \,\dx\right)^p \le C \\
\label{est2}
& \E \left(\intTO{\vr_n|\Grad\vu_n|^2}\right)^p+  \E \left(\intTO{|\Grad\vr_n^{\frac{\gamma}{2}}|^2}\right)^p\le C,\\
\label{est4}
    & \E\left( \frac{1}{2+\delta} \sup_{t \in [0,T]} \intO{\vr_n|\vu_n|^{2+\delta}} + c_\delta \intTO{\vr_n|\vu_n|^\delta|\Grad\vu_n|^2} \right) \le C.
\end{align}
\end{assumption}

\begin{theorem}\label{thm-existence}
 Let
 $\left(U_n, W_n, \vrn,\vun\right)_{n\in \N}$ be a sequence of martingale solutions to problem \eqref{eqn-1.1} with the initial data $\vrn^{0}$ and $\vc{m}_{n}^{0}$ satisfying  \textbf{Assumption}~\ref{ass-initial} uniformly w.r.t. $n$. Let $\vc{f}$ satisfy \textbf{Assumption}~  \ref{ass-force} and let the uniform estimates from \textbf{Assumption}~\ref{ass-energy} be satisfied.
Furthermore, we assume that there exist random variables $\vr^0 \in L^\gamma(\dom)$ and $\vm^0 \in L^{3/2}(\dom)$ such that
\begin{equation}
\label{eqn:rho_n^0 to rho^0}
\begin{array}{c}
\vspace{0.2cm}
\vr^{0}\geq0,\quad \vrn^{0}\rightarrow\vr^{0} \mbox{ in } L^{\gamma}(\dom)\quad \mbox{and}\quad \vc{m}_n^0\rightarrow\vm^0 \mbox{ weakly in}\ L^{3/2}(\dom).
\end{array}
\end{equation}
Then, there exists a stochastic basis $\widetilde{U} = \left(\tOmega, \widetilde{\mathcal{F}}, \tP, \widetilde{\mathbb{F}}\right)$, $\R$-valued Wiener process $\wt W$ on $\wt U$ and $\wt{\mathcal{F}}$-measurable processes $\wt\vr$ and $\wt\vu$ such that the sequence $\{(\wt\vr_n, \wt\vu_n)\}_{n\in\N}$ converges to $(\wt\vr, \wt\vu)$.  Moreover, $\left(\widetilde{U}, \widetilde{W}, \widetilde\vr, \widetilde\vu\right)$ is a martingale solution to \eqref{eqn-1.1} such that $\mathcal{L}(\wt\vr(0)) =\mathcal{L}(\vr^0)$ and $\mathcal{L}(\wt\vr\wt\vu(0)) = \mathcal{L}(\vm^0)$.
\end{theorem}

\begin{remark}\label{rem-weak convergence}
The sequence $\{(\wt\vr_n, \wt\vu_n)\}_{n\in\N}$ converges to $(\wt\vr, \wt\vu)$ weakly. The precise meaning of it is explained in the proof.
\end{remark}

\begin{remark}
\label{rem-main_thm}
Each martingale solution is a tuple consisting of the stochastic basis, Wiener process and the measurable processes. Since, we consider a sequence of martingale solutions we are dealing with a sequence of stochastic bases as well as a sequence of Wiener processes. However, Jakubowski in \cite{Jak97}, showed that one can consider a common stochastic basis, see also \cite{Breit+Ho16},  $U = \left(\Omega, \mathcal{F}, \mathbb{P}, \mathbb{F}\right)$ and can assume without loss of generality one Wiener process $W$ and we adopt this philosophy. This is the reason that in \eqref{est1}--\eqref{est4} we can consider a common probability measure $\mathbb{P}$ incorporated in the definition of the expectation $\E$.
\end{remark}

\begin{remark}
\label{rem-gen}
The result of Theorem~\ref{thm-existence} can be generalised in the following two ways:
\begin{itemize}
\item[(i)] by considering time dependent stochastic diffusion, i.e. $\vf : \Omega \times [0,T] \times \dom \to \R^3$ and
\item[(ii)] by taking infinite dimensional noise, for e.g. a Hilbert space valued Wiener process $W$.
\end{itemize}
\end{remark}


\section{A priori estimates }
\label{sec-apriori}
In this section we state the a priori estimates, being derived under the hypothesis that all quantities in question are smooth enough to justify our manipulations. We derive the estimates for both, deterministic (in the absence of Wiener process, but in the presence of deterministic external force ``$\vf{\rm{d}}t$'') and stochastic system (in the presence of Wiener process ``$\vf\,{\rm{d}}W$''). The indexes $n$ are dropped, but the reader should notice that all these estimates are uniform in $n$.


\subsection{Conservation of mass}
We start with the conservation of mass.  Integrating the continuity equation \eqref{eqn-1.1}$_1$ over $\dom$ we deduce that
\begin{equation*}
\Dt \intO{\vr}=0,
\end{equation*}
i.e. knowing that $\intO{\vr^{0}}=M$ we have $\intO{\vr(t,x)}=M$ for any $t\in[0,T]$. The conservation of mass identity can be interpreted as  a following uniform bound
\begin{equation} \begin{split}
\label{eqn:cons_mass}
\|\vr\|_{L^\infty(0,T; L^1(\dom))} \le C,
\end{split}
\end{equation}
for some $C > 0$. The above bound holds for every $\omega \in \Omega$ and the constant $C$ is independent of $\omega$ and only depends on $M$.


\subsection{The energy equality}
In Theorem~\ref{thm-existence} we assume that the sequence $\left(\vr_n, \vu_n\right)_{n \in \N}$ satisfies certain energy inequality and a priori estimates. However, in this section we showcase the steps and techniques used to establish these estimates. One crucial step is the application of the It\^o Lemma which is not completely justified as we are working on infinite dimensional Banach spaces. In principal, one could obtain these estimates in a finite dimensional setting (by studying a finite dimensional approximation of \eqref{eqn-1.1} with the help of Faedo-Galerkin type approximation, where application of the It\^o Lemma is justified) and extend those to infinite dimensions through limiting procedure.

In what follows we derive the usual energy estimate, for deterministic as well as stochastic case.

For the whole section we choose and fix
 $T>0$, random variables  $\vr^0$, $\vc{m}^0$ satisfying Assumption \eqref{ass-initial} and $\vc{f}$ satisfying Assumption \eqref{ass-force} and a
system $(U, W, {\vr}, \vu)$ which is a martingale solution to the problem \eqref{eqn-1.1} in the sense of Definition \ref{def-mart_sol} with the initial data $\vr^0$ and $\vm^0$  such that
 Assumption~\ref{ass-energy}  is satisfied.  Here $ U := ({\Omega},{\mathcal{F}}, {\mathbb{P}},{\mathbb{F}})$
is a stochastic basis with filtration ${\mathbb{F}}=({\mathcal{F}}_t)_{t\in[0,T]}$ satisfying the so called usual  assumptions and $W$ is a $\R$-valued Wiener process  on $U$.
 We will prove a series of additional a'priori energy estimates under a vague assumption that the process $({\vr}, \vu)$ is sufficiently smooth.

\begin{lemma}\label{lem-energy estimate det}
\dela{For any  solution of the deterministic version of problem  (\ref{eqn-1.1}),} In the deterministic case,  i.e. with  $\mathrm{d}W=\mathrm{d}t$,  \dela{in the sense of Definition \ref{def-mart_sol}}  we have
\begin{equation}\label{first}
\Dt \intOB{{\frac12}\vr(t,x)|\vu(t,x)|^2+\frac{1}{\gamma-1}\vr^{\gamma}(t,x)}+\intO{\vr|\Grad\vu(t,x)|^2}=\intO{\vr(t,x)\vf(t,x)\cdot\vu(t,x)}.
\end{equation}

\end{lemma}
\begin{proof}[Proof of Lemma \ref{lem-energy estimate det}] Multiplying the momentum equation of system \eqref{eqn-1.1} by $\vu$ then integrating by parts and using the following relation
\begin{equation} \begin{split}
\intO{\Grad\vr^\gamma\cdot\vu}=\intO{\frac{\gamma}{\gamma-1}\vr\Grad\vr^{\gamma-1}\cdot\vu}
&=-\intO{\frac{\gamma}{\gamma-1}\vr^{\gamma-1}\Div(\vr\vu)}\\
&=\frac{1}{\gamma-1}\Dt\intO{\vr^\gamma}
\end{split}
\end{equation}
 along with the following identity (due to the boundary conditions)
 \begin{equation} \begin{split}
 \intO{\Div(\vr\Grad\vu)\cdot\vu}=-\intO{\vr|\Grad\vu|^2},
 \end{split}
 \end{equation}
 we infer the energy identity \eqref{first}.
\end{proof}

\begin{lemma}
\label{lem-energy estimate stoch}
The following equality
is satisfied for every $t \in [0,T]$, $\mathbb{P}$-a.s.,
\begin{align}
\label{eqn:stoc_est_1}
&\intOB{\frac12 \vr(t,x)|\vu(t,x)|^2 + \frac{1}{\gamma - 1} \vr^\gamma(t,x)} + \inttO{\vr(s,x)|\Grad\vu(s,x)|^2}\\
&\;= \intOB{\frac12 \frac{\vc{m}_0^2(x)}{\vr_0(x)} + \frac{1}{\gamma - 1} \vr_0^\gamma(x)} + \frac12 \inttO{\vr(s,x)|\vf(s,x)|^2}
\nonumber\\
&+ \int_0^t\!\!\!\int_{\dom} \vu(s,x)\cdot \vf(s,x)\vr(s,x)\,\mathrm{d}x\,\mathrm{d}W(s). \nonumber
\end{align}
\end{lemma}

\begin{proof}[Proof of Lemma \ref{lem-energy estimate stoch}]
We use the It\^o Lemma to obtain these estimates. We will repeatedly use the following notation $\bm = \vr \vu$. Let us introduce a function $\varphi$ defined by
\begin{align}\label{eqn:varhi-1}
\varphi(\vr, \bm)  &= \intO{\frac12 \vr(x)|\vu(x)|^2} = \intO{\frac12 \frac{|\bm(x)|^2}{\vr(x)}}.
 \end{align}
For the $\varphi$ defined above we have following

\begin{align}
\label{eqn:varhi-12}
\begin{split}
\frac{\partial \varphi(\vr, \bm)}{\partial \vr}(y) &= - \int_\dom \frac12 y(x)\frac{|\bm(x)|^2}{\vr^2(x)}\,dx, \qquad \frac{\partial \varphi}{\partial \bm}(\vc{z}) = \int_\dom \frac{1}{\vr} \bm(x) \vc{z}(x)\,dx,\\ \frac{\partial^2\varphi}{\partial \bm^2}(\vc{z}_1,\vc{z}_2) &= \int_\dom \frac{1}{\vr(x)}\vc{z}_1(x)\vc{z}_2(x)\,dx.
\end{split}
\end{align}
Therefore, by the application of It\^o Lemma to the function $\varphi$ and the processes $\vr$ and $\bm$ whose differentials are given by \eqref{eqn-1.1}, we get
\begin{align*}
\intO{\frac12 \vr|\vu|^2} & = \intO{\frac12 \frac{|\vc{m}_0|^2}{\vr_0}} - \frac12\inttO{\mathrm{d}\vr |\vu|^2} + \inttO{\vu\cdot \mathrm{d}(\vr\vu)} + \frac12 \inttO{\vr|\vf|^2}.
\end{align*}
Using \eqref{eqn-1.1}, we obtain
\begin{align}
\label{eqn:energy1}
\intO{\frac12 \vr|\vu|^2} & = \intO{\frac12 \frac{|\vc{m}_0|^2}{\vr_0}} + \frac12\inttO{\Div\left(\vr \vu\right) |\vu|^2} \nonumber \\
&\; - \inttO{\vu\cdot \Div\left(\vr \vu \otimes \vu\right)}+ \inttO{\vu\cdot \Div\left(\vr \Grad\vu\right)} \nonumber \\
&\;- \inttO{\vu\cdot \nabla \vr^\gamma} + \inttO{\vr\vu\cdot\vf}\,dW + \frac12 \inttO{\vr|\vf|^2}\nonumber\\
&\;:= \sum_{i=1}^7 J_i.
\end{align}
Now we investigate each of these terms, individually. For $J_3$, the integration by parts along with the boundary conditions, gives us
\begin{align}
\label{eqn:estJ2}
J_3 & = - \inttO{\vu\cdot \Div\left(\vr \vu \otimes \vu\right)}  = - \frac12 \inttO{\Div\left(\vr \vu\right) |\vu|^2},
\end{align}
and this term thus cancels out $J_2$.\\
For $J_4$ again, the integration by parts results in
\begin{align}
\label{eqn:estJ3}
J_4 & = \inttO{\vu\cdot\Div(\vr\Grad\vu))} = -\inttO{\vr|\Grad\vu|^2}.
\end{align}
Finally, for $J_5$ we write
\begin{align}
\label{eqn:estJ4}
J_5 & = -  \inttO{\Grad\vr^\gamma\cdot\vu}  = - \inttO{\frac{\gamma}{\gamma-1}\vr\Grad\vr^{\gamma-1}\cdot\vu} \nonumber\\
&\;
= \inttO{\frac{\gamma}{\gamma-1}\vr^{\gamma-1}\Div(\vr\vu)} \nonumber \\
& \;= - \frac{1}{\gamma-1}\intO{\vr^\gamma} + \frac{1}{\gamma-1}\intO{\vr_0^\gamma}.
\end{align}
Using \eqref{eqn:estJ2}--\eqref{eqn:estJ4} in \eqref{eqn:energy1} and on rearranging, we obtain \eqref{eqn:stoc_est_1}.
\end{proof}


\subsection{The Bresch-Desjardins entropy}
To proceed we need to find some better estimate of the norm of density than $L^{\infty}(0,T;L^{\gamma}(\dom))$. It will be a consequence of integrability of gradient of $\vr$ obtained by a modification of entropy inequality proved for the first time by Bresch $\&$ Desjardins \cite{BrDe03}. We will roughly recall most important steps from the original proof.
\begin{lemma}\label{lem-entropy Bresch-Desjardins}
In the deterministic case, i.e. with $\mathrm{d}W=\mathrm{d}t$, the following equality is  satisfied
\begin{align}\label{both}
&\Dt \intOB{{\frac12}\vr(t,x)|\vu+\Grad\log\vr(t,x)|^2+\frac{1}{\gamma-1}\vr^{\gamma}(t,x)}\\
&+\intO{\Grad\log\vr(t,x)\cdot\Grad p(\vr(t,x) )}+2\intO{\vr|\A\vu(t,x)|^2}
\nonumber
\\
&=\intO{\vr(t,x)\vf(t,x)(\vu(t,x)+\Grad\log\vr(t,x))},
\nonumber
\end{align}
where the anti-symmetric gradient $\A\vu$ is given by
\begin{equation}
\label{eqn:anti-symm}
\A\vu=\frac{1}{2}\left(\Grad\vu-\Grad^{\bot}\vu\right).
\end{equation}
\end{lemma}
\begin{proof}[Proof of Lemma \ref{lem-entropy Bresch-Desjardins}]
We will expand the LHS of equality \eqref{both}, i.e. we will compute
\begin{equation} \begin{split}
&\Dt \intOB{{\frac12}\vr|\vu+\Grad\log(\vr)|^2+\frac{1}{\gamma-1}\vr^{\gamma}}\\
&=\Dt\intOB{{\frac12}\vr|\vu|^2+\frac{1}{\gamma-1}\vr^{\gamma}}+\Dt \intO{{\frac12}\vr|\Grad\log(\vr)|^2}+
\Dt\intO{\vr\vu\cdot\Grad\log\vr}\\
&=I_1+I_2+I_3
\end{split}
\end{equation}
The first integral  $I_1$ is clearly the same as in the classical energy estimate so we drop it, for the second one we write
\begin{align}
\label{I2}
&I_2=\Dt \intO{{\frac12}\vr|\Grad\log(\vr)|^2} \nonumber\\
&= \intO{{\frac12}\vr\pt|\Grad\log(\vr)|^2}-\intO{{\frac12}\Div(\vr\vu)|\Grad\log(\vr)|^2} \nonumber\\
&= \intO{\vr\Grad\log\vr\cdot\Grad\lr{\frac{\pt\vr}{\vr}}}-\intO{{\frac12}\Div(\vr\vu)|\Grad\log(\vr)|^2} \nonumber\\
&=- \intO{\vr\Grad\log\vr\cdot\Grad\lr{\frac{\vu\cdot\Grad\vr}{\vr}+\Div\vu}}-\intO{{\frac12}\Div(\vr\vu)|\Grad\log(\vr)|^2}\nonumber\\
&=- \intO{\vr\Grad\log\vr\cdot\Grad\Div\vu}
-\intO{\vr\Grad\log\vr\otimes\vu:\Grad \log\vr}
\nonumber\\
&\qquad-\intO{\vr\Grad\vu:\Grad\log\vr\otimes\Grad\log\vr}-\intO{{\frac12}\Div(\vr\vu)|\Grad\log(\vr)|^2}\nonumber\\
&=\intO{\vr\lap\log\vr\Div\vu}+\intO{\vr|\Grad\log\vr|^2\Div\vu}\nonumber\\
&\qquad -\intO{\Div(\vr\Grad\log\vr\Grad\log\vr\otimes\vu)}+\intO{\vr\vu\lap\log\vr\cdot\Grad\log\vr}+\intO{|\Grad\log\vr|^2\Div(\vr\vu)}\nonumber\\
&\qquad-\intO{\vr\Grad\vu:\Grad\log\vr\otimes\Grad\log\vr}-\intO{{\frac12}\Div(\vr\vu)|\Grad\log(\vr)|^2} \nonumber \\
&=\intO{\vr\lap\log\vr\Div\vu}+\intO{\vr|\Grad\log\vr|^2\Div\vu}+\intO{\vr\vu\lap\log\vr\cdot\Grad\log\vr}\nonumber\\
&\qquad-\intO{\vr\Grad\vu:\Grad\log\vr\otimes\Grad\log\vr}+\intO{{\frac12}\Div(\vr\vu)|\Grad\log(\vr)|^2}\nonumber\\
&=\intO{\vr\lap\log\vr\Div\vu}+\intO{\vr|\Grad\log\vr|^2\Div\vu} \nonumber\\
&\qquad+{\frac12}\intO{\Div(\vr\vu\cdot\Grad\log\vr \Grad\log\vr)}-\intO{\vr\Grad\vu:\Grad\log\vr\otimes\Grad\log\vr} \nonumber \\
&=\intO{\vr\lap\log\vr\Div\vu}+\intO{\vr|\Grad\log\vr|^2\Div\vu}-\intO{\vr\Grad\vu:\Grad\log\vr\otimes\Grad\log\vr}.
\end{align}
Next, for $I_3$ we first write
\begin{equation} \begin{split}\label{I3}
I_3=\Dt\intO{\vr\vu\cdot\Grad\log\vr}=\intO{\pt(\vr\vu)\cdot\Grad\log\vr}+\intO{(\Div(\vr\vu))^2\frac{1}{\vr}}
\end{split}
\end{equation}
and we focus on the first part of this expression. Thanks to the momentum equation of system \eqref{eqn-1.1} we have
\eqh{
&\intO{\pt(\vr\vu)\cdot\Grad\log\vr}\\
&=-\intO{\Div(\vr\vu\otimes\vu)\cdot\Grad\vr}-\intO{\Grad\vr^\gamma\cdot\Grad\log\vr}+\intO{\Div(\vr\mathbb{D}\vu)\cdot\Grad\log\vr}+\intO{\vr\vf\cdot\Grad\log\vr}.
}
In particular, we can rearrange the third term as
\eqh{
&\intO{\Div(\vr\Grad\vu)\cdot\Grad\log\vr}=-\intO{\vr\Grad\vu:\Grad^2\log\vr}\\
&=-\intO{\Grad\vu:\Grad^2\vr}+\intO{\Grad\vu:\Grad\log\vr\otimes\Grad\vr}\\
&=-\intO{\partial_i\vu\cdot\partial_i\Grad\vr}+\intO{\Grad\vu:\Grad\log\vr\otimes\Grad\vr}\\
&=\intO{\Grad\Div\vu\cdot\Grad\vr}+\intO{\Grad\vu:\Grad\log\vr\otimes\Grad\vr}\\
&=-\intO{\Div\vu\lap\vr}+\intO{\Grad\vu:\Grad\log\vr\otimes\Grad\vr}\\
&=-\intO{\Div\vu(\vr\lap\log\vr+\vr|\Grad\log\vr|^2)}+\intO{\Grad\vu:\Grad\log\vr\otimes\Grad\vr}.
}
Therefore, coming back to \eqref{I3} we obtain
\begin{equation} \begin{split}
&\Dt\intO{\vr\vu\cdot\Grad\log\vr}+\intO{\Grad\vr^\gamma\cdot\Grad\log\vr}\\
&=-\intO{\Div(\vr\vu\otimes\vu)\cdot\Grad\vr}+\intO{(\Div(\vr\vu))^2\frac{1}{\vr}}+\intO{\vr\vf\cdot\Grad\log\vr}\\
&\qquad-\intO{\Div\vu(\vr\lap\log\vr+\vr|\Grad\log\vr|^2)}+\intO{\Grad\vu:\Grad\log\vr\otimes\Grad\vr}.
\end{split}
\end{equation}
Summing up this expression with \eqref{I2} we note that a lot of terms cancel and we are left with
\begin{equation} \begin{split}
&\Dt \intO{{\frac12}\vr|\Grad\log\vr|^2}+\Dt\intO{\vr\vu\cdot\Grad\log\vr}+\intO{\Grad\vr^\gamma\cdot\Grad\log\vr}\\
&=-\intO{\Div(\vr\vu\otimes\vu)\cdot\Grad\vr}+\intO{(\Div(\vr\vu))^2\frac{1}{\vr}}+\intO{\vr\vf\cdot\Grad\log\vr}.
\end{split}\end{equation}
Moreover, using the continuity equation we have
\eqh{
-\intO{\Div(\vr\vu\otimes\vu)\cdot\Grad\vr}&=-\intO{\Div(\vr\vu)\vu\cdot\Grad\vr\frac{1}{\vr}}-\intO{\vr(\vu\cdot\Grad)\vu\cdot\Grad\log\vr}\\
&=-\intO{(\Div(\vr\vu))^2\frac{1}{\vr}}+\intO{\Div(\vr\vu)\Div\vu}-\intO{(\vu\cdot\Grad)\vu\cdot\Grad\vr}.
}
Therefore we infer that
\begin{equation}\label{almost_there} \begin{split}
&\Dt \intO{{\frac12}\vr|\Grad\log\vr|^2}+\Dt\intO{\vr\vu\cdot\Grad\log\vr}+\intO{\Grad\vr^\gamma\cdot\Grad\log\vr}\\
&=\intO{\Div(\vr\vu)\Div\vu}-\intO{(\vu\cdot\Grad)\vu\cdot\Grad\vr}+\intO{\vr\vf\cdot\Grad\log\vr}\\
&=\intO{\vr(\Div\vu)^2}+\intO{\vu\cdot\Grad\vr\Div\vu}-\intO{(\vu\cdot\Grad)\vu\cdot\Grad\vr}+\intO{\vr\vf\cdot\Grad\log\vr}\\
&=-\intO{\vr\vu\cdot\Grad\Div\vu}+\intO{\Div((\vu\cdot\Grad)\vu)\vr}+\intO{\vr\vf\cdot\Grad\log\vr}\\
&=\intO{\vr\Grad\vu:\Grad\vu^T}+\intO{\vr\vf\cdot\Grad\log\vr}.
\end{split}\end{equation}
Finally noticing that
\eqh{
|\Grad\vu|^2-\Grad\vu:\Grad\vu^T=2|\A\vu|^2
}
we can sum up \eqref{almost_there} with \eqref{first} to get precisely \eqref{both}.
\end{proof}

Now we are ready to present the main contribution of the present subsection, i.e. a generalization of Lemma \ref{lem-entropy Bresch-Desjardins} to the stochastic case.
\begin{lemma}
\label{lem-entropy stoch Bresch-Desjardins}
The following equality
is satisfied for every $t \in [0,T]$, $\mathbb{P}$-a.s.,
\begin{align}\label{eqn:both_stoch_1}
& \intOB{{\frac12}\vr|\vu+\Grad\log\vr|^2+\frac{1}{\gamma-1}\vr^{\gamma}} \nonumber\\
&\qquad +\inttO{\Grad\log(\vr)\cdot\Grad p(\vr)} + 2\inttO{\vr|\A\vu|^2}\nonumber \\
&\; = \intOB{{\frac12}\vr_0\left|\frac{\vc{m}_0}{\vr_0}+\Grad\log\vr_0\right|^2+\frac{1}{\gamma-1}\vr_0^{\gamma}} + \frac12 \inttO{\vr|\vf|^2} \nonumber \\
  &\qquad + \int_0^t \intO{\vr\vf\cdot\left(\vu+\Grad\log\vr\right)}\,{\mathrm{d}}W,
\end{align}
where the anti-symmetric gradient $\A\vu$ is defined in \eqref{eqn:anti-symm}.
\end{lemma}

\begin{proof}[Proof of Lemma \ref{lem-entropy stoch Bresch-Desjardins}]
We will expand the LHS of equality \eqref{eqn:both_stoch_1} in the differential form, i.e. we will compute
\begin{equation}
\label{eqn:diff_est}
\begin{split}
&\mathrm{d}\left( \intOB{{\frac12}\vr|\vu+\Grad\log\vr|^2+\frac{1}{\gamma-1}\vr^{\gamma}}\right)\\
&=\mathrm{d}\left(\intOB{{\frac12}\vr|\vu|^2+\frac{1}{\gamma-1}\vr^{\gamma}}\right)+ \mathrm{d}\left( \intO{{\frac12}\vr|\Grad\log\vr|^2}\right)+
\mathrm{d}\left(\intO{\vr\vu\cdot\Grad\log\vr}\right)\\
&=I_1+I_2+I_3
\end{split}
\end{equation}
$I_1$ is already computed in \eqref{eqn:stoc_est_1}. Since, $I_2$ contains only $\vr$, a deterministic function it can be dealt as in Lemma~\ref{lem-entropy Bresch-Desjardins}. Now for $I_3$, we apply the It\^o product rule for the function
\begin{align}\label{eqn:varhi-2}
\varphi(\vr, \bm) &:= \intO{ \bm \cdot \nabla \log\vr}= \intO{\vr \vu \cdot \nabla \log\vr}
\end{align}
and the processes $\bm$ and $\vr$.
Note that the function $\varphi$ defined above in  \eqref{eqn:varhi-2} differed from the function defined earlier in \eqref{eqn:varhi-1}. Moreover,  the correction term will not appear since $\vr$ is a deterministic function.

Thus, we have
\[
\mathrm{d}\left(\intO{\bm \cdot \nabla \log\vr}\right) = \intO{ \mathrm{d}\bm \cdot \nabla \log\vr} + \intO{ \bv \cdot \nabla (\mathrm{d} \vr) \frac{1}{\vr}}.
\]
Integrating the above with respect to time, the integration by parts and the continuity equation \eqref{eqn-1.1}$_1$ results in
\begin{equation}
\label{eqn:estI3}
\begin{split}
\intO{\vr\vu\cdot \nabla \log\vr } & = \intO{\vc{m}_0\cdot \nabla \log\vr_0} + \inttO{ d(\vr\vu)\cdot \nabla \log\vr}\\
&\quad + \inttO{ \frac{1}{\vr}\Div(\vr\vu)^2}.
\end{split}
\end{equation}
Now, using the calculations from Lemma~\ref{lem-entropy Bresch-Desjardins} (for $I_3$ and the following computations), on integrating \eqref{eqn:diff_est} w.r.t. $t \in [0,T]$ and summing up the expressions for $I_1$, $I_2$ and $I_3$, we obtain \eqref{eqn:both_stoch_1}.
\end{proof}


\subsection{Uniform estimates}
The following uniform estimates are consequences of  Lemmata \ref{lem-energy estimate stoch} and \ref{lem-entropy stoch Bresch-Desjardins} and the Burkholder-Davis-Gundy inequality, see \cite{Pardoux_1976}.

\begin{lemma}\label{lem-estimates}
Assume that  $ p \ge 1$. Then there exists
a constant $C$ dpending only on the initial data and the external force such that
the following estimates are satisfied.
\dela{uniform with respect to $n$.}
\begin{align}
  \label{eqn:both_stoch_est_1}
& \sup_{t \in [0,T]}\E \left( \intOB{\frac12 \vr|\vu|^2+{\frac12}\vr|\vu+\Grad\log\vr|^2+\frac{1}{\gamma-1}\vr^{\gamma}} \right) \le C,\\
\label{eqn:both_stoch_2}
&\E \left(\sup_{t\in[0,T]} \int_\dom \vr|\vu|^2 \,\dx\right)^p
+\E\lr{\sup_{t \in [0,T]} \intOB{\vr|\Grad\log\vr|^2}}^p
+\E \left(\sup_{t \in[0,T]} \int_\dom \vr^\gamma \,\dx\right)^p \le C \\
\label{eqn:both_stoch_3}
& \E \left(\intTO{\vr|\Grad\vu|^2}\right)^p+  \E \left(\intTO{|\Grad\vr^{\frac{\gamma}{2}}|^2}\right)^p\le C,
\end{align}
where the constant $C$ depends on initial data, time and $\vf$.
\end{lemma}
\begin{proof}[Proof of Lemma \ref{lem-estimates}]

We first sum up inequalities \eqref{eqn:stoc_est_1} with \eqref{eqn:both_stoch_1}. Noticing that $\int_0^t \int_\dom \vr|\A\vu|^2\,\dx\ds\geq 0$ and can be disregarded we obtain the {inequality}

\begin{align}\label{eqn:sum_sum}
& \intOB{{\frac12}\vr|\vu|^2+{\frac12}\vr|\vu+\Grad\log\vr|^2+\frac{2}{\gamma-1}\vr^{\gamma}} +\inttO{\Grad\log\vr\cdot\Grad p(\vr)} +  \inttO{\vr|\Grad\vu|^2}\nonumber \\
&\;\leq \intOB{{\frac12}\vr_0\left|\frac{\vc{m}_0}{\vr_0}\right|^2+{\frac12}\vr_0\left|\frac{\vc{m}_0}{\vr_0}+\Grad\log\vr_0\right|^2+\frac{2}{\gamma-1}\vr_0^{\gamma}} \nonumber \\
&\qquad+\inttO{\vr|\vf|^2} + \int_0^t\!\!\!\int_{\dom} \vr\vu\cdot \vf\,\mathrm{d}x\,\mathrm{d}W+ \int_0^t\intO{\vr\left(\vu+\Grad\log\vr\right)\cdot \vf}\,{\mathrm{d}}W
\end{align}
Note that the integrand in  second term on the LHS of \eqref{eqn:sum_sum} equals
\begin{equation}\label{rs}
\Grad\log\vr\cdot\Grad p(\vr)=\gamma\vr^{\gamma-2}|\Grad\vr|^2= C_\gamma|\Grad\vr^{\frac{\gamma}{2}}|^2.
\end{equation}
Using the Sobolev embedding, we can therefore estimate
\begin{equation}\label{est_from_p} \begin{split}
C\int_0^t\|\vr^\frac{\gamma}{2}\|^2_{L^6(\dom)}\,\ds\leq C_\gamma \int_0^t\|\Grad\vr^\frac{\gamma}{2}\|^2_{L^2(\dom)}\,\ds=\inttO{\Grad\log\vr\cdot\Grad p(\vr)},
\end{split}\end{equation}
where the constant on the LHS comes from the Sobolev embedding. Inequality \eqref{eqn:sum_sum} thus gives
\begin{align}\label{eqn:sum_sum2}
& \intOB{{\frac12}\vr|\vu|^2+{\frac12}\vr|\vu+\Grad\log\vr|^2+\frac{2}{\gamma-1}\vr^{\gamma}} +C\int_0^t\|\vr^\frac{\gamma}{2}\|^2_{L^6(\dom)}\,\ds +  \inttO{\vr|\Grad\vu|^2}\nonumber \\
&\;\leq \intOB{{\frac12}\vr_0\left|\frac{\vc{m}_0}{\vr_0}\right|^2+{\frac12}\vr_0\left|\frac{\vc{m}_0}{\vr_0}+\Grad\log\vr_0\right|^2+\frac{2}{\gamma-1}\vr_0^{\gamma}} \nonumber \\
&\qquad+\inttO{\vr|\vf|^2} + \int_0^t\!\!\!\int_{\dom} \vr\vu\cdot \vf\,\mathrm{d}x\,\mathrm{d}W+ \int_0^t\intO{\vr\left(\vu+\Grad\log\vr\right)\cdot \vf}\,{\mathrm{d}}W\\
&\qquad:=\sum_{i=1}^4 J_i.\nonumber
\end{align}

In order to obtain the first estimate \eqref{eqn:both_stoch_est_1} we first take mathematical expectation and then supremum over time in \eqref{eqn:sum_sum2}. Note that  $\E(J_3+J_4) = 0$, and thus \eqref{eqn:sum_sum2} reduces to
\begin{align}\label{eqn:sum_sum2b}
& \sup_{t \in [0,T]} \E \left(\intOB{{\frac12}\vr|\vu|^2+{\frac12}\vr|\vu+\Grad\log\vr|^2+\frac{2}{\gamma-1}\vr^{\gamma}}\right) \nonumber \\
& \qquad + \sup_{t \in [0,T]}\E \left(C\int_0^t\|\vr^\frac{\gamma}{2}\|^2_{L^6(\dom)}\,\ds +  \inttO{\vr|\Grad\vu|^2}\right)\nonumber \\
&\;\leq \E \left(\intOB{{\frac12}\vr_0\left|\frac{\vc{m}_0}{\vr_0}\right|^2+{\frac12}\vr_0\left|\frac{\vc{m}_0}{\vr_0}+\Grad\log\vr_0\right|^2+\frac{2}{\gamma-1}\vr_0^{\gamma}}\right) \nonumber \\
& \qquad + \sup_{t \in [0,T]} \E \left(\inttO{\vr|\vf|^2}\right).
\end{align}
The first term on the r.h.s. of \eqref{eqn:sum_sum2b} is controlled by the Assumption \ref{ass-initial}.
To estimate the second term we use H\"older and Young inequalities respectively
\eq{
\label{eqn:estJ6}
 J_2&=\inttO{\vr|\vf|^2} \le C\int_0^t{\|\vr\|_{L^{3\gamma}(\dom)}\|\vf^2\|_{L^{\frac{3\gamma}{3\gamma-1}}(\dom)}}\,\ds
 =C\int_0^t{\|\vr^{\frac{\gamma}{2}}\|_{L^{6}(\dom)}^{\frac{2}{\gamma}}\|\vf\|^2_{L^{\frac{6\gamma}{3\gamma-1}}(\dom)}}\,\ds\\
 &\leq \ep \int_0^t{\|\vr^{\frac{\gamma}{2}}\|_{L^{6}(\dom)}^{2}\,\ds+t C_\ep\|\vf\|^{\frac{2\gamma}{\gamma-1}}_{L^{\frac{6\gamma}{3\gamma-1}}(\dom)}},
}
and so,
\eq{
\label{eqn:estJ6E}
\sup_{t\in[0,T]}\E \lr{\inttO{\vr|\vf|^2} }
 \leq \ep \sup_{t\in[0,T]}\E\lr{\int_0^t\|\vr^{\frac{\gamma}{2}}\|_{L^{6}(\dom)}^{2}\,\ds}+T C_\ep\E\|\vf\|^{\frac{2\gamma}{\gamma-1}}_{L^{\frac{6\gamma}{3\gamma-1}}(\dom)}.
}
Choosing $\ep$ sufficiently small the first term can be absorbed by the LHS of \eqref{eqn:sum_sum2b}. For the second term, note that $\frac{6\gamma}{3\gamma-1}\leq 3$ for $\gamma\in(1,3)$ , and so it is bounded provided $\vf$ is bounded in $L^3(\dom)$.

To prove the rest of the estimates in Lemmma \ref{lem-estimates} we need to first apply the supremum over time to both sides of \eqref{eqn:sum_sum2}, and then take the mathematical expectation of the $p$-th power, for $p \ge 1$. This results in
\begin{align}\label{eqn:sum_sum3}
& \E\lr{\sup_{t\in[0,T]}\intO{{\frac12}\vr|\vu|^2}}^p
+ \E\lr{\sup_{t\in[0,T]}\intO{{\frac12}\vr|\vu}}^p
+ \E\lr{\sup_{t\in[0,T]}\intO{\frac{2}{\gamma-1}\vr^{\gamma}}}^p \\
&+C\E\lr{\int_0^T\|\vr^\frac{\gamma}{2}\|^2_{L^6(\dom)}}^p +  \E\lr{\intTO{\vr|\Grad\vu|^2}}^p\nonumber \\
&\;\leq \E\lr{\intOB{{\frac12}\vr_0\left|\frac{\vc{m}_0}{\vr_0}\right|^2+{\frac12}\vr_0\left|\frac{\vc{m}_0}{\vr_0}+\Grad\log\vr_0\right|^2+\frac{2}{\gamma-1}\vr_0^{\gamma}} }^p\nonumber \\
&\qquad+\E\lr{\sup_{t\in[0,T]}\inttO{\vr|\vf|^2} }^p+\E\lr{\sup_{t\in[0,T]} \int_0^t\!\!\!\int_{\dom} \vr\vu\cdot \vf\,\mathrm{d}x\,\mathrm{d}W}^p\nonumber\\
&\qquad+ \E\lr{\sup_{t\in[0,T]}\int_0^t\!\!\!\intO{\vr\left(\vu+\Grad\log\vr\right)\cdot \vf}\,{\mathrm{d}}W}^p:=\sum_{i=1}^4 I_i,\nonumber
\end{align}
where we have used the algebraic inequality
\[
a^p + b^p \le (a + b)^p \le C(p) \left(a^p + b^p\right) \quad a > 0, b > 0 \; \mbox{and } p \ge 1.
\]
Estimate of $I_2$ uses \eqref{eqn:estJ6}, from which we deduce
\eq{
\label{eqn:estJ6Es}
I_2=\E\lr{\sup_{t\in[0,T]}\inttO{\vr|\vf|^2} }^p
 \leq \ep \E\lr{\int_0^T\|\vr^{\frac{\gamma}{2}}\|_{L^{6}(\dom)}^{2}\,\ds}+T C_\ep\E\|\vf\|^{\frac{2\gamma}{\gamma-1}}_{L^{\frac{6\gamma}{3\gamma-1}}(\dom)},
}
which again can be absorbed by the LHS of \eqref{eqn:sum_sum3} and assumptions on $\vf$. To estimate $I_3$ and $I_4$ we use the Burkholder-Davis-Gundy inequality, the H\"older
and the Young inequalites
\begin{align}
\label{eqn:estJ5}
& I_3 = \E\left(\sup_{t \in [0,T]} \left|\int_0^t\!\!\intO{\vr\vu\cdot\vf}\,{\rm{d}}W\right|\right)^p \nonumber \\
& \quad \le \E \left(\int_0^T\left|\intO{\vr\vu\cdot\vf}\right|^2\,\dt\right)^{\frac p2} \nonumber \\
& \quad \le \E \left(\int_0^T\|\sqrt{\vr}\|^2_{L^{6\gamma}}\|\sqrt{\vr}\vu\|^2_{L^2}\|\vf\|^2_{L^{\frac{6\gamma}{3\gamma-1}}}\,\dt\right)^{\frac p2} \nonumber \\
& \quad \le \E \left[\left(\int_0^T\|\vr^{\frac{\gamma}{2}}\|^{\frac{2}{\gamma}}_{L^{6}}\,\ds\right)^{\frac p2} \lr{\sup_{t\in[0.T]}\|\sqrt{\vr}\vu\|^2_{L^2}}^{\frac{p}{2}}\|\vf\|^p_{L^{\frac{6\gamma}{3\gamma-1}}}\right]\nonumber \\
& \quad \le \E \left[\left(\int_0^T\|\vr^{\frac{\gamma}{2}}\|^{2}_{L^{6}}\,\ds\right)^{\frac p{2\gamma}} \lr{\sup_{t\in[0.T]}\|\sqrt{\vr}\vu\|^2_{L^2}}^{\frac{p}{2}}T^{\frac{p}{2\gamma'}}\|\vf\|^p_{L^{\frac{6\gamma}{3\gamma-1}}}\right]\nonumber \\
&\quad \le \ep\E \left(\int_0^T\|\vr^{\frac{\gamma}{2}}\|^{2}_{L^{6}}\,\ds\right)^p+\ep\E\lr{\sup_{t\in[0.T]}\|\sqrt{\vr}\vu\|^2_{L^2}}^p +C(\ep)T^p\E|\vf\|^{2p\gamma'}_{L^{\frac{6\gamma}{3\gamma-1}}},
\end{align}
where $\ep$ is sufficiently small for the first two terms of \eqref{eqn:estJ5} to be absorbed by the LHS of \eqref{eqn:sum_sum3}. Note that $\frac{6\gamma}{3\gamma-1}\leq 3$ for $\gamma\in(1,3)$ and so, the last term is bounded provided $\vf$ is bounded in $L^3(\dom)$.
Exactly the same argument can be repeated for $I_4$ replacing $\vu$ by $\vu+\Grad\log\vr$. This finishes the proof.
\end{proof}

To summarise, the estimates from Lemma~\ref{lem-estimates} give us the following uniform bounds
\begin{equation}
\label{eqn:energy_uni}
\begin{split}
\|\sqrt{\vr}\vu\|_{L^p(\Omega; L^\infty(0,T; L^2(\dom)))} \le C,\\
\|\vr^\gamma\|_{L^p(\Omega; L^\infty(0,T; L^1(\dom)))} \le C, \\
\|\sqrt{\vr}\Grad\vu\|_{L^p(\Omega; L^2(0,T; L^2(\dom)))} \le C.
\end{split}
\end{equation}
as well as
\begin{equation}
\label{eqn-BD_unif}
\begin{split}
\|\nabla \sqrt{\vr}\|_{L^p(\Omega; L^\infty(0,T; L^2(\dom)))} \le C, \\
\|\nabla \vr^{\gamma/2}\|_{L^p(\Omega; L^2(0,T; L^2(\dom)))} \le C.
\end{split}
\end{equation}

From \eqref{eqn-BD_unif}, we have $\vr^\gamma$ is uniformly bounded in $L^p(\Omega; L^1(0,T; L^3(\dom)))$ and using the interpolation, one can infer that for $p \ge 1$ there exists a constant $C>0$ such that
\begin{equation}\label{eqn-rho_n^gamma}
    \|\vr^\gamma\|_{L^{p}(\Omega; L^{5/3}((0,T)\times\dom))} \le C \|\vr^\gamma\|^{2/5}_{L^{p}(\Omega; L^\infty(0,T; L^1(\dom)))} \|\vr^\gamma\|^{3/5}_{L^{p}(\Omega; L^{1}(0,T; L^3(\dom)))} \le C.
\end{equation}


\subsection{The Mellet-Vasseur estimate}
Our ultimate goal before the limit passage is to improve the integrability of the convective term.
\begin{lemma}\label{lem-Mellet+Vasseur-stoch}
Let $\dd\in(0,1)$. Then there exist constants $c_\delta>0$ and $C_\delta>0$
depending only on the initial data such that the following inequality is satisfied
\eq{\label{delta_stoch}
&\frac{1}{2+\delta} \intO{\vr(t,x)|\vu(t,x)|^{2+\delta}}+c_\delta\inttO{\vr(s,x)|\vu(s,x)|^\delta|\Grad\vu(s,x)|^2}\\
&\leq CT+ C_\delta\int_0^t\left(\intO{\left(\frac{p(\vr(s,x))^{2}\vr(s,x)^{-\frac{\dd}{2}}}{\vr(s,x)}\right)^{\frac{2}{2-\dd}}}\right)^{\frac{2-\dd}{2}}
\,\ds\\
&\qquad+  \frac{1+\delta}{2} \inttO{\vr(s,x)|\vf(s,x)|^2|\vu|^\delta} + \int_0^t\intO{\vr(s,x)\vf(s,x)\cdot\vu(s,x)|\vu(s,x)|^\delta}\,\mathrm{d}W(s).
}
\end{lemma}
\begin{remark}\label{rem-Mellet+Vasseur-stoch}
Note that because $\gamma<3$, the above  result implies the  equi-integrability of the acceleration term $\vr\vu\otimes\vu$.
\end{remark}

\begin{proof}[Proof of Lemma \ref{lem-Mellet+Vasseur-stoch}] When  $\mathrm{d}W=\mathrm{d}t$ this follows the same strategy as in the work of Mellet $\&$ Vasseur \cite[Lemma~3.2]{Me+Va07}.
We basically test the momentum equation of \eqref{eqn-1.1} by $\vu|\vu|^\delta$, and repeat the energy estimate. We can show that the r.h.s. can be controlled using the left hand side and estimates from the previous lemmas and the Gronwall inequality.

In the stochastic case  $\mathrm{d}W\neq\mathrm{d}t$ we  again denote
\[
\varphi(\vr, \bv) := \frac{1}{2+\delta} \intO{\vr\frac{|\bv|^{2+\delta}}{\vr^{2+\delta}}}=\frac{1}{2+\delta}\intO{\vr|\vu|^{2+\delta}}.
\]
For the $\varphi$ defined above we have following
\eq{\label{expand}
\frac{\partial \varphi(\vr, \bv)}{\partial \vr}(y) &= -\frac{1+\delta}{2+\delta}\intO{\frac{|\bv|^{2+\delta}}{\vr^{2+\delta}}\,y}=-\frac{1+\delta}{2+\delta}\intO{|\vu|^{2+\delta}\,y}, \\
\frac{\partial \varphi}{\partial \bv}(\vc{z}) &= \intO{\frac{\bv|\bv|^\delta}{\vr^{1+\delta}}\, \vc{z}}=\intO{\vu|\vu|^\delta\, \vc{z}},\\
 \frac{\partial^2\varphi}{\partial \bv^2}(\vc{z}_1,\vc{z}_2) & = (1+\delta)\intO{\frac{1}{\vr^{1+\delta}}|\bv|^\delta \, \vc{z}_1\vc{z}_2} = (1+\delta) \intO{ \frac{\,|\vu|^\delta}{\vr}\, \vc{z}_1\vc{z}_2}.
 }
Therefore, by the application of the It\^o Lemma for the function $\varphi$ and the processes $\vr$ and $\bv$ whose differentials are given by \eqref{eqn-1.1}, we obtain
\begin{align}
\label{delta_stoch_1}
&\frac{1}{2+\delta} \intO{\vr(t)|\vu(t)|^{2+\delta}}=\frac{1}{2+\delta}\intO{\frac{|\vc{m}_0|^{2+\delta}}{\vr_0^{1+\delta}}}
-\frac{1+\delta}{2+\delta}\int_0^t\!\!\intO{{\rm d}\vr\,|\vu|^{2+\delta}} \nonumber\\
& \qquad \qquad \qquad +\int_0^t\!\!\intO{{\rm d}(\vr\vu)\cdot\vu|\vu|^\delta}+ \frac12 \inttO{\vr|\vf|^2|\vu|^\delta} \nonumber\\
& = \frac{1}{2+\delta} \intO{\frac{|\vc{m}_0|^{2+\delta}}{\vr_0^{1+\delta}}}
+\frac{1+\delta}{2+\delta}\inttO{\Div\lr{\vr\vu}|\vu|^{2+\delta}} \\
&\quad -\inttO{\Div(\vr\vu\otimes\vu)\cdot\vu|\vu|^\delta}+\inttO{\Div(\vr\Grad\vu)\cdot\vu|\vu|^\delta}\nonumber \\
&\quad -\inttO{\Grad p(\vr)\cdot\vu|\vu|^\delta}+\int_0^t\!\!\intO{\vr\vf\cdot\vu|\vu|^\delta}\,\mathrm{d}W \nonumber \\
&\quad + \frac{1+\delta}{2} \inttO{\vr|\vf|^2|\vu|^\delta}=\sum_{i=1}^7 I_i. \nonumber
\end{align}
We observe that the extra term on the r.h.s. of \eqref{delta_stoch} in comparison with the deterministic case, see \cite[Lemma~3.2]{Me+Va07} for details, is
\[
 \frac{1+\delta}{2} \inttO{\vr|\vf|^2|\vu|^\delta}.
\]
Integrating by parts in terms $I_2$ and $I_3$ and using the boundary conditions we get cancellation. Integrating by parts in  $I_4$ we obtain two terms
\begin{equation} \begin{split}
I_4 & = \inttO{\Div(\vr\Grad\vu)\cdot\vu|\vu|^\delta}\\
& = -\inttO{\left(|\vu|^\delta\vr|\Grad\vu|^2 - \delta \vr|\vu|^{\delta-2}u_iu_k\partial_ju_i\partial_j u_k\right)}
 \\
& \leq -(1-\delta)\inttO{|\vu|^\delta\vr|\Grad\vu|^2},
\end{split}\end{equation}
that can be moved to the l.h.s of \eqref{expand} as it has the right sign provided $\delta<1$. Finally, for $I_5$ we have
\eqh{
I_5 & = -\inttO{\Grad p(\vr)\cdot\vu|\vu|^\delta} \\
& =\inttO{p(\vr)\Div\vu|\vu|^\delta}+\delta \inttO{ p(\vr)|\vu|^{\delta-2}\vu (\vu\cdot\Grad)\vu}\\
&\leq \sqrt{3+\delta}\inttO{p(\vr)|\vu|^\delta |\Grad\vu|}\\
& \leq \frac{1-\delta}{2}\inttO{\vr|\vu|^\delta|\Grad\vu|^2}+C_\delta\inttO{\frac{p(\vr)^{2}}{\vr}|\vu|^\delta}\\
&\leq\frac{1-\delta}{2}\inttO{\vr|\vu|^\delta|\Grad\vu|^2} \\
&\qquad + C_\delta\int_0^t\left(\intO{\left(\frac{p(\vr)^{2}\vr^{-\frac{\dd}{2}}}{\vr}\right)^{\frac{2}{2-\dd}}}\right)^{\frac{2-\dd}{2}}\lr{\intO{\vr|\vu|^2}}^{\frac{\delta}{2}}\ds,
}
and so the first term can be controlled by the contribution coming from $I_4$, while for the second term we use that $\vr|\vu|^2$ is uniformly bounded in $L^p(\Omega; L^\infty(0,T; L^1(\dom)))$ to conclude the proof of \eqref{delta_stoch}.
\end{proof}

Now, as from \eqref{eqn-rho_n^gamma} we know that for $p\geq 1$, $\vr^\gamma\in L^{p}(\Omega; L^{5/3}((0,T)\times\dom))$, the pressure term on the r.h.s. of  \eqref{delta_stoch} is bounded provided $\frac{p^2(\vr)}{\vr}=\vr^{2\gamma-1}$ can be controlled by {$\vr^{5\gamma/3}$}, which holds for $\gamma<3$.

In the following we show how to estimate the rest of the terms on the r.h.s. of \eqref{delta_stoch}. We start with the third term:
\begin{align}
\label{eqn:logest1}
& \E\left(\inttO{\vr|\vf|^2 |\vu|^\delta}\right) \nonumber \\
&\leq \E \left( \int_0^t \left[ \left(\intO{ \left(\vr^{1 - \tfrac{\delta}{2}}|\vf|^2\right)^{\tfrac{2}{2-\delta}}}\right)^{\tfrac{2-\delta}{2} }\left(\intO{\vr|\vu|^2}\right)^{\frac{\delta}{2}}\right]\,ds\right) \nonumber \\
&\le \frac{2-\delta}{2}\E \inttO{\vr |\vf|^{\tfrac{4}{2-\delta}}} + \frac{\delta}{2} \E \inttO{\vr|\vu|^2} \nonumber \\
& \; \le \frac{2-\delta}{6} \E \left( \|\vf^{\tfrac{4}{2-\delta}}\| _{L^{\frac{3\gamma}{3\gamma-1}}} \int_0^t \|\vr\|_{L^{3\gamma}}\ds\right) + \frac{\delta t}{3} \E \lr{\sup_{t\in[0,T]}\intO{\vr|\vu|^2}} \nonumber \\
& \; \le \frac{2-\delta}{6} \left( \E \|\vf\|^{8/(2-\delta)}_{L^{\frac{12\gamma}{(2-\delta)(3\gamma-1)}}}\right)^{1/2}\left(\E\|\vr\|^2_{L^1(0,T; L^{3\gamma}(\dom))}\right)^{1/2} + \frac{\delta t}{3} \E \lr{\sup_{t\in[0,T]}\intO{\vr|\vu|^2}}.
\end{align}
Since for every $\gamma \in (1,3)$ there exists a $\delta \in (0,1)$ such that $\frac{4\gamma}{(3\gamma - 1)(2 - \delta)} <1$, the assumptions on $\vf$, $\vf \in L^p(\Omega; L^3(\dom))$, and a priori estimates from Lemma~\ref{lem-energy estimate stoch} can be used to bound the r.h.s. of \eqref{eqn:logest1}.

Next we are required to bound the following term:
\[
\E\left(\sup_{t \in [0,T]} \int_0^t\intO{\vr\vf\cdot\vu|\vu|^\delta}\,\mathrm{d}W\right)
\]
By the Burkholder-Davis-Gundy inequality, H\"older and Young inequalities we obtain
\begin{align}
\label{eqn:logest4}
&\E \sup_{t \in [0,T]} \left(\int_0^t\intO{\vr(s)\vf\cdot\vu|\vu|^\delta}\,\mathrm{d}W\right) \nonumber \\
&\; \le \E \left[ \int_0^T\left(\intO {\vr(t)\vf\cdot\vu|\vu|^\delta} \right)^2 \dt\right]^{1/2} \nonumber \\
&\; \le \E \left[ \int_0^T  \left(\intO{\vr|\vu|^2}\right) \left(\intO{\vr |\vu|^{2+\delta}}\right)^{\tfrac{2\delta}{2+\delta}} \left(\intO {\vr|\vf|^{\tfrac{2(2+\delta)}{2-\delta}}}\right)^{\tfrac{2-\delta}{2+\delta}}\dt \right]^{1/2} \nonumber\\
&\; \le \E \left[ \left(\sup_{t\in[0,T]}\intO{\vr|\vu|^2}\right)^{1/2} \left(\int_0^T \int_\dom{\vr |\vu|^{2+\delta}}\dx\,\dt\right)^{\tfrac{\delta}{2+\delta}} \left(\int_0^T \int_\dom {\vr|\vf|^{\tfrac{2(2+\delta)}{2-\delta}}}\dx\,\dt\right)^{\tfrac{2-\delta}{2(2+\delta)}} \right] \nonumber\\
&\; \le \frac12 \E \left[ \sup_{t\in[0,T]}\intO{\vr|\vu|^2}\right] + \frac{2\delta}{2+\delta}\E \left[\sup_{t\in[0,T]} \int_\dom{\vr |\vu|^{2+\delta}}\dx\right] \nonumber \\
& \hspace{3truecm}+ \frac{2-\delta}{2(2+\delta)} T^{\tfrac{2\delta}{2-\delta}}\ \E \left[\int_0^T \int_\dom {\vr|\vf|^{\tfrac{2(2+\delta)}{2-\delta}}}\dx\,\dt\right]
\end{align}
The first term on the r.h.s. of \eqref{eqn:logest4} can be uniformly bounded by the a priori estimate \eqref{eqn:energy_uni} and the second term can be absorbed on the LHS of \eqref{delta_stoch} for $\delta < \frac12$. The last term on the r.h.s. of \eqref{eqn:logest4} is estimated as follows
\begin{align*}
    \E \left[\int_0^T \!\!\!\int_\dom {\vr|\vf|^{\tfrac{2(2+\delta)}{2-\delta}}}\dx\,\dt\right]
    &\; \le \frac12 \E \left[\|\vr\|^2_{L^1(0,T; L^{3\gamma}(\dom))}\right] + \frac12 \E\left[\||\vf|^{\tfrac{2(2+\delta)}{2-\delta}}\|^2_{L^{\frac{3\gamma}{3\gamma-1}}}\right]
\end{align*}
and which is bounded due to the a priori estimate \eqref{eqn-BD_unif} and the assumption on $\vf$, i.e. $\vf \in L^p(\Omega; L^3(\dom))$. Indeed, as for every $\gamma \in (1,3)$ there exists a $\delta \in (0,1)$ such that $\tfrac{2\gamma(2+\delta)}{(3\gamma-1)(2-\delta)} < 1$.


\section{Existence of solutions and the  proof of Theorem \ref{thm-existence}}
\label{sec-existence}

For the whole section we choose and fix
 $T>0$, sequences of random variables  $\vr^0_n$ and  $\vc{m}^0_n$ satisfying Assumption \eqref{ass-initial} and $\vc{f}$ satisfying uniformly Assumption \ref{ass-force} and a sequence
system $(U, W, {\vr_n}, \vu_n)$ of  martingale solutions to the problem \eqref{eqn-1.1} in the sense of Definition \ref{def-mart_sol} with the initial data $\vr^0_n$ and $\vm^0_n$  such that
 Assumption~\ref{ass-energy}  is satisfied.  Here for simplicity of the exposition we assume that  $ U := ({\Omega},{\mathcal{F}}, {\mathbb{P}},{\mathbb{F}})$
is a fixed stochastic basis with filtration ${\mathbb{F}}=({\mathcal{F}}_t)_{t\in[0,T]}$ satisfying the so called usual  assumptions and $W$ is a $\R$-valued Wiener process  on $U$ but we could have taken as well this object to be $n$-dependent.

 In the previous Section \ref{sec-apriori} we  proved a series of additional a'priori energy estimates for  the processes $({\vr_n}, \vu_n)$. Using these a priori estimates we deduce  that the  sequence $(U, W, {\vr_n}, \vu_n)$ of  martingale solutions   satisfy additionally  estimates \eqref{eqn:energy_uni}, \eqref{eqn-BD_unif} uniformly in $n$, i.e. for all  \[ p\in[1,\infty) \mbox{  and } \delta \in (0,1),\] there exists $C$ (depending on $p$ and $\delta$), such that for every $n \in \mathbb{N}$ we have
\begin{align}
\label{eqn-4.1}
&\|\vr_n\|_{L^p(\Omega; L^\infty(0,T; L^\gamma(\dom)))} \le C, \\
\label{eqn-4.2}
&\|\sqrt{\vr_n}\vu_n\|_{L^{p}(\Omega; L^\infty(0,T; L^2(\dom)))} \le C,\\
\label{eqn:4.3}
&\|\sqrt{\vr_n}\Grad\vu_n\|_{L^{p}(\Omega; L^2(0,T; L^2(\dom)))} \le C,\\
\label{eqn:4.4}&\| \sqrt{\vr_n}\|_{L^p(\Omega; L^\infty(0,T; H^1(\dom)))} \le C, \\
\label{eqn:4.5}&\|\nabla \vr_n^{\gamma/2}\|_{L^p(\Omega; L^2(0,T; L^2(\dom)))} \le C, \\
\label{eqn:4.6}&\|\vr_n^{\frac{1}{2+\delta}}|\vu_n|\|_{L^{2+\delta}(\Omega; L^\infty(0,T; L^{2+\delta}(\dom)))} \le C.
\end{align}
As observed after the proof of  Lemma~\ref{lem-estimates}  inequalities \eqref{eqn-4.2}, \eqref{eqn-4.1}, \eqref{eqn:4.3}, \eqref{eqn:4.4} and \eqref{eqn:4.5}  are a consequence of that Lemma.
Let us emphasize that Lemma~\ref{lem-estimates} is a consequence of our first new a'priori estimates on the Bresch-Desjardins entropy from our Lemma \ref{lem-entropy stoch Bresch-Desjardins}. Moreover,
inequalities \eqref{eqn:4.6} is a consequence of  the first of our new a'priori estimates on of the  Mellet-Vasseur
from Lemma \ref{lem-Mellet+Vasseur-stoch}. Let us point out that  we  do not use the second of the a'priori estimates following from Lemma \ref{lem-Mellet+Vasseur-stoch}. Interestingly, this estimate is neither used by the authors of the deterministic paper \cite{Me+Va07}.

We recall that by Remark~\ref{rem-main_thm}, it is possible to choose a common stochastic basis on which this sequence of martingale solutions is defined.


\subsection{Tightness of density and momentum}
With these a priori estimates at hand we are able to pass to the weak limit for the subsequence of $\vr_n$. Our next purpose is to show tightness of the laws  $\Law(\vr_n)$, $\Law(\vr_n\vu_n)$ according to the Definition~\ref{def-tightness}. To that purpose we will use  some intuitions coming from  the deterministic case and the compact embedding result as stated in Lemma~\ref{lem-compactness Simon-1}, see also \cite[Corollary 4]{Sim87}.

Our first result is as follows.

\begin{lemma}\label{lem-tightness of rho_n}
Assume that  $1\leq q<3$. Then the family of measures  $\{\Law({\vr_n}):n\in\N\}$ is tight on $C([0,T];L^{q}(\dom))$.
\end{lemma}

\begin{proof}[Proof of Lemma \ref{lem-tightness of rho_n}]  Let us fix $q \in [1,3)$. Firstly, we observe that from the uniform bound \eqref{eqn:4.4}, we have $\vr_n \in L^1(\Omega; L^\infty(0,T; W^{1,\frac32}(\dom)))$. Indeed,
 \[
 \nabla \vr_n = 2 \sqrt{\vr_n}\, \nabla\left( \sqrt{\vr_n}\right) \mbox{ in a weak sense}
 \]
and since  $\sqrt{\vr_n}$ is uniformly bounded in $L^{2}(\Omega; L^\infty(0,T; L^6(\dom)))$, by the Sobolev embedding theorem (as $d=3$) and by the a priori estimates \eqref{eqn:4.4}, and $\nabla \sqrt{\vr_n}$ is uniformly bounded in $L^2(\Omega; L^\infty(0,T; L^2(\dom)))$, again by  the a priori estimates \eqref{eqn:4.4},  we infer that
\begin{equation}\label{ineq-01}
\E\|\nabla \vr_n\|_{L^\infty(0,T; L^{\frac32}(\dom))} \le C.
\end{equation}
By a very similar argument we can prove, using a priori estimates \eqref{eqn-4.2} and \eqref{eqn:4.4}, together with the Sobolev embedding theorem, that
\begin{equation}\label{ineq-02}
\E\| \vr_n\vu_n  \|_{L^\infty(0,T; L^{\frac32}(\dom))} \le C.
\end{equation}
Therefore, by  continuity equation we infer
\begin{equation}\label{ineq-03}
 \E\|\pt\vr_n\|^{p}_{L^\infty(0,T;W^{-1,\frac32}(\dom))} \le C.
\end{equation}
Summing up, we proved that
\begin{equation}\label{ineq-03b}
\E\|{\vr_n}\|_{L^\infty(0,T; W^{1,\frac32}) \cap W^{1,\infty}(0,T; W^{-1,\frac32})}	\leq C.
\end{equation}
Since $q<3$, the embedding $W^{1,\frac32}(\dom) \embed  L^q(\dom)$ is compact.
Thus, by  Lemma~\ref{lem-compactness Simon-1} with  $X=W^{1,\frac32}(\dom)$, $B=L^q(\dom)$, and $Y= W^{-1,\frac32}(\dom)$, we infer that the embedding
\begin{equation}\label{ineq-04}
L^\infty(0,T; W^{1,\frac32}(\dom)) \cap W^{1,\infty}(0,T; W^{-1,\frac32}(\dom)) \embed  C([0,T]; L^q(\dom)) \mbox{ is compact}.
\end{equation}
Let us choose and fix $\ep>0$. Put  $\eta= \cfrac{C}{\ep} $, where $C$ is from estimate \eqref{ineq-03b}.  Set
$$
B_{\eta} :=\{ f \in  L^\infty(0,T; W^{1,\frac32}(\dom)) \cap W^{1,\infty}(0,T; W^{-1,\frac32}(\dom)) : \|f\|_{L^\infty(0,T; W^{1,\frac32}) \cap W^{1,\infty}(0,T; W^{-1,\frac32})} \le \eta\}.
$$
Then by the Chebyshev inequality and uniform estimates \eqref{ineq-03b}  obtained above, we deduce that
\begin{equation} \begin{split}
&\Prob\big\{\vr_n \in B_{\frac{1}{\ep}}^c \big\}
\le\frac{1}{\eta} \E\|{\vr_n}\|_{L^\infty(0,T; W^{1,\frac32}) \cap W^{1,\infty}(0,T; W^{-1,\frac32})}	\leq C\frac{1}{\eta}=\ep.
\end{split}\end{equation}
 Since  by \eqref{ineq-04}, set  $B_{\frac{1}{\ep}}$ is a precompact set in $C([0,T]; L^q(\dom))$,  tightness of laws of $\vr_n$ on $C([0,T];L^q(\dom))$ follows directly from the Definition~\ref{def-tightness}.
 \end{proof}

In a similar manner we prove the tightness of laws for $\sqrt{\vr_n}$. In what follows, if  $X$ be a topological vector space,  then by $(X,w)$  we mean that the set $X$ is equipped with weak topology.
If additionally, $X$ is isomorphic to a dual of another topological vector space, then by $(X, w^\ast)$, we denote the set $X$ equipped with the weak-star topology.

\begin{lemma}\label{lem-tightness of sqrt of rho_n}
Assume that  $1\leq r<6$.
Then the  family of measures  $\{\Law(\sqrt{\vr_n}):n\in\N\}$ is tight on $C([0,T];L^{r}(\dom))\cap \left(L^\infty(0,T; H^1(\dom)), w^\ast\right)$.
\end{lemma}

\begin{proof}[Proof of Lemma \ref{lem-tightness of sqrt of rho_n}] Let us fix $r \in [1,6)$. To verify tightness in the space $C([0,T];L^{r}(\dom))$ we proceed as previously, {using the compactness of the embedding $H^1(\dom) \embed L^r(\dom)$ and }noticing that $\sqrt{\vr_n}$ satisfies the equation
\begin{equation} \begin{split}
\pt\sqrt{\vr_n}+\Div(\sqrt{\vr_n}\vu_n)-\frac12\sqrt{\vr_n}\Div\vu_n=0.
\end{split}\end{equation}
Since, by the Banach-Alaoglu theorem a closed ball in $L^\infty(0,T; H^1)\cong \bigl( L^1(0,T; H^{-1})\big)^\ast$ is compact in the weak-star topology, the tightness of $\{\Law(\sqrt{\vr_n}):n\in\N\}$ on the space $\left(L^\infty(0,T; H^1(\dom)), w^\ast\right)$ follows straight from the Chebyshev inequality and a priori estimate \eqref{eqn:4.4}.
\end{proof}

\begin{lemma}\label{lem-tightness of  rho_n u_n} Let $r \in [1, \infty)$ and $ \alpha \in [1, \frac32)$.
Then the  family of measures $\{\Law(\vr_n\vu_n):n\in\N\}$ is tight on $L^r(0,T; L^\alpha(\dom))$.
\end{lemma}

\begin{proof}[Proof of Lemma \ref{lem-tightness of  rho_n u_n}]
Let us fix $r \in [1, \infty)$ and $\alpha \in [1,\frac32)$. Let us also choose and fix auxiliary exponents $p \in [1, \infty)$  and $p^\ast \in [1, 2)$ such that $p^\ast = \frac{2p}{p+2}$, i.e.
\[
\frac1p+\frac12=\frac{1}{p^\ast}.
\]
From the uniform estimates \eqref{eqn-4.2} and \eqref{eqn:4.4} and the Sobolev embedding $H^1(\dom) \embed L^6(\dom)$, we infer that $\sqrt{\vr_n}\vu_n \in L^p(\Omega;L^\infty(0,T; L^2(\dom)))$ and $\sqrt{\vr_n} \in L^2(\Omega; L^\infty(0,T;L^6(\dom)))$ and are uniformly bounded in the respective spaces. Therefore, by applying the H\"older inequality, see inequality \eqref{eqn:Holder-2},  we  infer that
\begin{equation}\label{eqn:4.14}
\vr_n \vu_n \in L^{p^\ast}(\Omega;L^\infty(0,T; L^{3/2}(\dom))) \mbox{ and is uniformly bounded}.
\end{equation}
We can also prove that $\nabla\left(\vr_n \vu_n\right) \in L^{p^\ast}(\Omega;L^2(0,T;L^1(\dom)))$ uniformly.  Indeed, since
\[
\nabla\left(\vr_n\vu_n\right) = 2 (\nabla \sqrt{\vr_n})(\sqrt{\vr_n}\vu_n) + \sqrt{\vr_n}(\sqrt{\vr_n}\nabla \vu_n)
\]
it is sufficient  to prove that the first  term on the r.h.s. above belongs to $L^{p^\ast}(\Omega;L^\infty(0,T; L^1(\dom)))$ and is uniformly bounded and
the second term belongs to $L^{p^\ast}(\Omega;L^2(0,T; L^{3/2}(\dom)))$ and is uniformly bounded.

For the first one, in view of the  H\"older inequality, see \eqref{eqn:Holder-4},  it is sufficient to observe that by  \eqref{eqn-4.2} and \eqref{eqn:4.4}, $\sqrt{\vr_n}\vu_n \in L^p(\Omega;L^\infty(0,T;L^2(\dom)))$ and $\nabla \sqrt{\vr_n} \in L^2(\Omega;L^\infty(0,T;L^2(\dom)))$ (uniformly).
For the second one, we observe that as proved above
$\sqrt{\vr_n} \in L^2(\Omega; L^\infty(0,T;L^6(\dom)))$ (uniformly)  and,  by \eqref{eqn:4.3},  $\sqrt{\vr_n}\Grad\vu_n \in L^p(\Omega;L^2(0,T;L^2(\dom)))$ (also uniformly), the claim follows by applying \eqref{eqn:Holder-3}.

Hence we can conclude that
\begin{equation}\label{eqn:4.15}
\vr_n\vu_n \in L^{p^\ast}(\Omega;L^2(0,T;W^{1,1}(\dom)))  \mbox{ and is uniformly bounded}.
\end{equation}
Next we will verify that the sequence of random variables $\vr_n\vu_n$ satisfies the Aldous condition $\textbf{[A]}$, see Definition~\ref{def-Aldous criteria}, {in a Banach  space  $Y$,  which will be chosen later in the proof}.
For this aim we choose and fix  $\bpsi \in C^\infty(\dom)$. Then  the momentum equation of system of \eqref{eqn-1.1} gives us
\begin{equation}
\label{eqn:mom_test}
\begin{split}
& \int_\dom \vr_n(t,x)\vu_n(t,x) \bpsi(x)\,\dx - \int_\dom {\vc{m}_n^0}\cdot \bpsi(x)\,\dx \\
& \; = \int_0^t \int_\dom \left(\vr_n(s,x)\vu_n(s,x)\otimes \vu_n(s,x)\right)\cdot \nabla \bpsi(x)\,\dx\,\ds - \int_0^t \int_\dom \vr_n(s,x) \Grad\vu_n(s,x) \cdot \nabla \bpsi(x)\,\dx\,\ds \\
&\quad + \int_0^t \int_\dom \vr_n^\gamma(s,x)\Div \bpsi(x)\,\dx\,\ds + \int_0^t \int_\dom \vr_n(s,x)\vf(x) \bpsi(x)\,\dx\,dW(s), \\
&:= J^{\bpsi}_1(t) + J^{\bpsi}_2(t) + J^{\bpsi}_3(t) + J^{\bpsi}_4(t), \;\; \mbox{for all $t \in [0,T]$}
\end{split}
\end{equation}
 and therefore
\[\begin{split}
    & \left|\int_\dom \vr_n((\tau_n+h)\wedge T,x)\vu_n((\tau_n+h)\wedge T,x) \bpsi(x)\,\dx - \int_\dom \vr_n(\tau_n,x)\vu_n(\tau_n,x) \bpsi(x)\,\dx\right|\\
    &\qquad\le \sum_{k=1}^4\left|J^{\bpsi}_k((\tau_n+h)\wedge T) - J^{\bpsi}_k(\tau_n)\right|
\end{split}\]
for any sequence $(\tau_n)_{n\in\N}$ of $[0,T]$-valued stopping times and $h > 0$. Hence, for every $\eta > 0$,
\eq{\label{eqn:rut1}
    & \Prob \left\{\|\vr_n((\tau_n+h)\wedge T)\vu_n((\tau_n+h)\wedge T)- \vr_n(\tau_n)\vu_n(\tau_n)\|_{Y} \ge \eta\right\} \\
    & \quad \le \sum_{k=1}^4\Prob\left\{\sup_{\bpsi\in Y^\prime : \|\bpsi\|_{Y^\prime} = 1} \left|J^{\bpsi}_k((\tau_n+h)\wedge T) - J^{\bpsi}_k(\tau_n)\right| \ge \frac{\eta}{4}\right\},
}
  where  $Y^\prime$ is the functional dual of $Y$. We aim to apply the Chebyshev inequality and estimate the expected value of each term in the sum.

In order to continue with our proof we choose and fix: $p \in [1, \infty)$ and  $p^\ast \in [1, 2)$. We claim that due to the uniform estimates \eqref{eqn-4.2}--\eqref{eqn:4.4}, the sequences below are uniformly bounded in the corresponding spaces:
\begin{trivlist}
\item[(i)] $\vr_n \vu_n \otimes \vu_n$ in $L^{p}(\Omega; L^\infty(0,T; L^{1}(\dom)))$;
\item[(ii)] $\vr_n^\gamma$ in  $L^p(\Omega; L^{\infty}(0,T; L^{1}(\dom)))$.
\item[(iii)] $\vr_n \Grad\vu_n$ in $L^{p^\ast}(\Omega; L^2(0,T; L^{\frac32}(\dom)))$.
\end{trivlist}
With all of these uniform bounds we will analyse each term in \eqref{eqn:rut1} individually.
\eq{ \label{eqn:rut2}
& \E \left|J^{\bpsi}_1((\tau_n+h)\wedge T) - J^{\bpsi}_1(\tau_n)\right| =  \E\left|\int_{\tau_n}^{(\tau_n+h)\wedge T} \int_\dom \left(\vr_n\vu_n\otimes \vu_n\right): \nabla \bpsi\,\dx\,\ds \right| \\
& \; \le h \E \left(\sup_{s \in [0, T]}\|\vr_n \vu_n \otimes \vu_n\|_{L^{1}(\dom)}\right) \|\nabla \bpsi\|_{L^\infty(\dom)} \le C_T \|\bpsi\|_{W^{1,\infty}(\dom)} h.
}
For the second term $J^{\bpsi}_2$ we have
\begin{align} \label{eqn:rut3}
&\E \left|J^{\bpsi}_2((\tau_n+h)\wedge T) - J^{\bpsi}_2(\tau_n)\right| =  \E \left|\int_{\tau_n}^{(\tau_n+h)\wedge T} \int_\dom \vr_n \Grad\vu_n(s,x) \cdot \nabla \bpsi\,\dx\,\ds\right| \nonumber\\
& \quad \le \E \left|\int_{\tau_n}^{(\tau_n+h)\wedge T} \|\vr_n \Grad\vu_n\|_{L^{\frac32}}\|\nabla \bpsi\|_{L^3}\,ds\right| \nonumber\\
&\quad \le h^{1/2} \E \left(\int_{\tau_n}^{(\tau_n+h)\wedge T}\|\vr_n\Grad\vu_n\|_{L^{\frac32}}^2ds\right)^{1/2} \|\nabla \bpsi\|_{L^3} \le C_T \|\bpsi\|_{W^{1,3}(\dom)}h^{1/2}.
\end{align}
For the pressure term, we get
\eq{\label{eqn:rut4}
&\E \left|J^{\bpsi}_3((\tau_n+h)\wedge T) - J^{\bpsi}_3(\tau_n)\right| = \E\left|\int_{\tau_n}^{(\tau_n+h)\wedge T} \int_\dom \vr_n^\gamma\Div \bpsi\,\dx\,\ds\right| \\
&\quad \le \E \left|\int_{\tau_n}^{(\tau_n+h)\wedge T}\|\vr_n^\gamma\|_{L^1}\|\Div\bpsi\|_{L^\infty}ds\right| \\
&\quad\le h \E\left(\sup_{s\in[0,T]} \|\vr_n^\gamma\|_{L^1}\right)\|\Div \bpsi\|_{L^\infty} \le C_T \| \bpsi\|_{W^{1,\infty}(\dom)}h.
}
Note in particular that so far the most restrictive assumption on $\bpsi$ has been made during the estimates on $J^{\bpsi}_1$, $J^{\bpsi}_3$ and for those we have that
\eqh{
W^{3,2}(\dom)\hookrightarrow W^{1,\infty}(\dom).
}
Finally, for the stochastic term, using the BDG inequality, assumptions on $\vf$ and uniform estiamte \eqref{eqn:4.4}, we obtain
\begin{align}\label{eqn:rut5}
&\E \left|J^{\bpsi}_4((\tau_n+h)\wedge T) - J^{\bpsi}_4(\tau_n)\right| = \E\left|\int_{\tau_n}^{(\tau_n+h)\wedge T} \int_\dom \vr_n\vf \cdot\bpsi\,\dx\,dW\right| \nonumber\\
&\; \le  \E \left(\int_{\tau_n}^{(\tau_n+h)\wedge T} \left[\int_\dom \vr_n\vf \cdot\bpsi\,dx\right]^2 ds\right)^{1/2} \nonumber\\
& \le \E \left(\int_{\tau_n}^{(\tau_n+h)\wedge T} \| \vr_n\|^2_{L^3}\|\vf\|_{L^{3/2}}^2\|\bpsi\|^2_{L^\infty}ds\right)^{1/2} \nonumber\\
&\le h^{1/2} \E \left( \left(\sup_{s\in [0,T]} \|\vr_n\|^2_{L^3}\right)^{1/2} \|\vf\|_{L^{3/2}}\right)\|\bpsi\|_{L^\infty} \nonumber\\
&\le h^{1/2} \left(\E \left( \sup_{s\in [0,T]} \|\vr_n\|^2_{L^3}\right)\right)^{1/2} \left(\E \|\vf\|^2_{L^{3/2}}\right)^{1/2}\|\bpsi\|_{L^\infty}
\le C_{T,\vf}\|\bpsi\|_{L^\infty}h^{1/2}.
\end{align}
By the Chebyshev inequality, we obtain for given $\eta > 0$
\begin{align*}
&\Prob\left\{\sup_{\bpsi\in Y^\prime : \|\bpsi\|_{Y^\prime} = 1} \left|J^{\bpsi}_k((\tau_n+h)\wedge T) - J^{\bpsi}_k(\tau_n)\right| \ge \frac{\eta}{4}\right\} \\
&\hspace{2truecm}\leq \frac{4}{\eta} \left(\sup_{\bpsi\in Y^\prime : \|\bpsi\|_{Y^\prime} = 1} \E \left|J^{\bpsi}_k((\tau_n+h)\wedge T) - J^{\bpsi}_k(\tau_n)\right|\right), \;\; k = 1,\cdots, 4.
\end{align*}
Therefore, for $Y = W^{-3,2}(\dom)$, using estimates \eqref{eqn:rut2}--\eqref{eqn:rut5} in \eqref{eqn:rut1}, we have shown that
\[
\Prob \left\{\|\vr_n((\tau_n+h)\wedge T)\vu_n((\tau_n+h)\wedge T)- \vr_n(\tau_n)\vu_n(\tau_n)\|_{Y} \ge \eta\right\} \le \frac{8}{\eta}Ch^{1/2},
\]
where $C$ is a constant independent of $h$. Let us fix $\eps > 0$. With the above estimate we can choose $\theta > 0$ such that
\[
\Prob \left\{\|\vr_n((\tau_n+h)\wedge T)\vu_n((\tau_n+h)\wedge T)- \vr_n(\tau_n)\vu_n(\tau_n)\|_{Y} \ge \eta\right\} \le \eps
\]
for all $n \in \N$ and $0 < h \le\theta$ and therefore, the Aldous condition $\textbf{[A]}$ holds in $Y = W^{-3,2}(\dom)$. Therefore by Lemma~\ref{lem-Aldous condition-2} there exists a measurable subset $A_\eps \subset L^1(0,T;Y)$ such that
\[
\mathbb{P}^{X_n}(A_\eps) \ge 1 - \frac{\eps}{2}, \qquad\quad \lim_{h \to 0}\sup_{u \in A_\eps}\|\tau_h u-u\|_{L^1(0,T;Y)} = 0.
\]
Thus, we can apply Lemma~\ref{lem-compactness Simon-2} with $X = W^{1,1}(\dom)$, $B = L^{3/2 -}(\dom)$ and $Y = W^{-3,2}(\dom)$,  and set $F = \{\vr_n(\omega) \vu_n(\omega)\}_{n \in \N}$ to show that the set $F$ is relatively compact in $L^r(0,T; L^\alpha(\dom))$.

Let $\eta > 0$ and define the following closed ball $B_\eta$
\[
B_\eta := \{\mathrm{v} \in F : \|\mathrm{v}\|_{ L^\infty(0,T; L^{3/2}) \cap L^1(0,T; W^{1,1})} \le \eta \}.
\]
Now, by the Chebyshev inequality and uniform bounds from above, we have
\begin{equation} \begin{split}
&\Prob(\{\omega\in\Omega:\vr_n(\omega) \vu_n (\omega) \in B_\eta^c\}) \le \frac{1}{\eta} \E \|\vr_n \vu_n\|_{ L^\infty(0,T; L^{3/2}) \cap L^1(0,T; W^{1,1})} \le \frac{C}{\eta}.
\end{split}\end{equation}
Taking $\ep = \frac{C}{\eta}$, and choosing $K_\ep = B_{\frac{1}{\ep}}$ we obtain from the first part that $K_\ep$ is a relatively compact set in $L^r(0,T; L^\alpha(\dom))$ and hence we establish tightness of laws of $\vr_n \vu_n$ on $L^r(0,T;L^\alpha(\dom))$ directly from the Definition~\ref{def-tightness}.
\end{proof}

In order to prove the tightness of laws induced by $\vr_n\vu_n$ we will take the approach of Breit et al., see \cite[Proposition~4.3.8]{Breit+Feir+Ho18}, namely we will prove the time regularity of $\vr_n\vu_n$, which holds uniformly in $n$ and we will do so by proving the time regularity for the deterministic part and the stochastic part of \eqref{eqn:mom_test} separately.


\begin{proposition}
\label{prop.det_holder}
Let $T > 0$. Then there exists a $C > 0$ and $\kappa  \in (0, \frac12)$ such that for every $n\in\mathbb{N}$,  the process $v$ defined by the deterministic part of \eqref{eqn:mom_test}, i.e.  by
\begin{equation}
    \label{eqn:mom_det}
    \begin{split}
\langle Y_n(t), \bpsi\rangle      & := \int_\dom\vc{m}_n^0 \cdot\bpsi\,\dx - \inttO{\left(\vr_n\vu_n\otimes \vu_n\right)\cdot \nabla \bpsi} \\
    &\quad  + \inttO{\vr_n\Grad\vu_n \cdot\nabla \bpsi} - \inttO{\vr_n^\gamma\Div \bpsi}, \;\;\; \bpsi \in {W^{3,2}}(\dom),
    \end{split}
\end{equation}
satisfies
\[
\E\|Y_n\|_{C^\kappa([0,T]; W^{-3,2}(\dom))} \le C.
\]
\end{proposition}
\begin{proof}[Proof of Proposition \ref{prop.det_holder}]
Proposition~\ref{prop.det_holder} can be proved by following the steps from the first half of the proof of Lemma~\ref{lem-tightness of  rho_n u_n}.
\end{proof}
Let us recall that if $\mathbb{U}$ is a Banach space, then by $C^\kappa([0,T];\mathbb{U})$ we denote the Banach space of all continuous functions $u:[0,T] \to \mathbb{U}$ such that
\begin{equation}\label{eqn:Holder space}
\Vert u \Vert_{C^\kappa([0,T];\mathbb{U})}:=\vert u(0)\vert_\mathbb{U}+ \sup \Bigl\{ \frac{\vert u(t)-u(s)\vert_\mathbb{U}}{\vert t -s \vert^\kappa}: s<t \in [0,T]\Bigr\}<\infty.
\end{equation}
The above nor is equivalent to another one, more commonly used, in which the term $\vert u(0)\vert_\mathbb{U}$ is replaced by $\sup\bigl\{ \vert u(t)\vert_\mathbb{U}: \, t \in [0,T]\bigr\}$.

\begin{proposition}
\label{prop.stoch_holder}
Let $T > 0$. Then there exists $C^\prime>0$ and $\kappa^\ast > 0$ such that {for every $n\in \mathbb{N}$,} the process $M_n$ defined by
\[
 M_n(t) := \int_0^t \!\!\int_\dom \vr_n\vf \,\dx\,\de W, \;\; t\in [0,T],
\]
satisfies, {$\mathbb{P}$-almost surely,} $M_n \in C^{\kappa^\ast}([0,T];W^{-2,2}(\dom))$ and
\begin{equation}\label{eqn:M_n-Holder}
\E\|M_n\|_{C^{\kappa^\ast}([0,T]; W^{-3,2}(\dom))} \le C^\prime.
  \end{equation}
\end{proposition}

\begin{proof}[Proof of Proposition \ref{prop.stoch_holder}]
Let us begin the proof with the following observation which is a consequence  of the H\"older inequality and  of the continuity the embedding $W^{-2,2}(\dom) \embed L^\infty(\dom)$.
\begin{align}\label{eqn:W^{-2,2}}
  \|\vr \vf\|_{W^{-2,2}(\dom)}&= \sup\bigl\{ \vert \int_{\dom} \vr(x)\vf(x)\bpsi(x)\,\dx \vert: \vert \bpsi \vert_{W^{2,2}(\dom)} \leq 1   \bigr\}\\
  &\leq \sup\bigl\{ \|\vr\|_{L^3}\|\vf\|_{L^{3/2}}\|\bpsi\|_{L^\infty} \,\dx: \vert \bpsi \vert_{W^{2,2}(\dom)} \leq 1   \bigr\} \leq  C \|\vr\|_{L^3}\|\vf\|_{L^{3/2}}.
  \nonumber
\end{align}

Let us fix $T > 0$ and $q > 2$. Then for $0 \le s < t \le T$.  Since $W^{-2,2}(\dom)$ is a Hilbert space, by applying the Doob inequality  we get
\begin{align*}
\E\left(\left| M_n(t) - M_n(s) \right|_{W^{-2,2}(\dom)}^q\right) & = \E\left|\int_s^t  \vr_n(r)\vf \,\de W(r)\right|_{W^{-2,2}(\dom)}^q  \le C_q  \E\left(\int_s^t\|\vr_n(r) \vf\|_{W^{-2,2}(\dom)} ^2\de r\right)^{q/2}
\\&  \leq
CC_q
 \E\left(\int_s^t\|\vr_n(r)\|^2_{L^3}\|\vf\|^2_{L^{3/2}}\de r\right)^{q/2} \\
&\leq CC_q  |t-s|^{q/2}  \E\left( \sup_{r \in [0,T]} \|\vr_n(r)\|^q_{L^3}\|\vf\|^q_{L^{3/2}} \right)
\\&\leq CC_q  |t-s|^{q/2} \left(\E \sup_{r \in [0,T]}\|\vr_n(r)\|^{q}_{L^3}\right)^{1/2} \left(\E\|\vf\|^{2q}_{L^{3/2}}\right)^{1/2}
\\
&\hspace{2truecm}  \le C^\prime |t-s|^{q/2}.
\end{align*}

Since the above estimate holds true for every $q > 2$ and hence we can conclude by the Kolmogorov continuity theorem, see~Theorem~\ref{thm-Kolmogorov cont}, that there exists $\kappa^\ast > 0$ such that $M_n \in C^{\kappa^\ast}([0,T]; W^{-2,2}(\dom))$ almost surely and \eqref{eqn:M_n-Holder} holds. The proof is now complete.
\end{proof}

\begin{lemma}\label{lem-tightness of  rho_n u_n in C_w}
The family of measures $\{\Law(\vr_n\vu_n):n\in\N\}$ is tight on $C_w([0,T]; L^{3/2}(\dom))$.
\end{lemma}

\begin{proof}[Proof of Lemma \ref{lem-tightness of  rho_n u_n in C_w}] By Proposition~\ref{prop.det_holder} and Proposition~\ref{prop.stoch_holder} we infer  that  for some $\kappa > 0$,
\[\vr_n\vu_n \in L^1(\Omega;C^\kappa([0,T];W^{-3,2}(\dom))).\]  Moreover, in the proof of Lemma~\ref{lem-tightness of  rho_n u_n} we showed that $\vr_n\vu_n \in L^1(\Omega; L^\infty(0,T; L^{3/2}(\dom)))$. Thus, proof of the lemma can be concluded by observing that the embedding
\[
L^\infty(0,T; L^{3/2}(\dom)) \cap C^\kappa([0,T]; W^{-3,2}(\dom)) \embed  C_w([0,T]; L^{3/2}(\dom))
\]
is compact; which follows directly from Lemma~\ref{lem-BFH weak cont compact}, and by application of the Chebyshev inequality.
\end{proof}

\begin{lemma}
\label{lem-tightness of sqrt of rho_n-L^2_w}
Assume that $\delta \in (0,1)$ is as in Lemma~\ref{lem-Mellet+Vasseur-stoch}.
The families of measures $\{\Law(\sqrt{\vr_n}\vu_n):n\in\N\}$, and resp. $\{\Law(\vr_n^{\frac{1}{2+\delta}}\vu_n):n\in\N\}$, are  tight on $L^\infty(0,T; L^2_{w^*}(\dom))$, resp.  $L^\infty(0,T; L^{2+\delta}_{w^*}(\dom))$.
\end{lemma}
 \begin{proof}[Proof of Lemma \ref{lem-tightness of  rho_n u_n in C_w}] Tightness of $\{\Law(\sqrt{\vr_n}\vu_n):n\in\N\}$ in $\left(L^\infty(0,T; L^2(\dom)),w^\ast\right)$ follows directly from the Banach-Alaoglu theorem, the Chebyshev inequality and a priori estimates \eqref{eqn-4.2}, as in Lemma~\ref{lem-tightness of sqrt of rho_n}. Similarly we  deduce tightness of  $\{\Law(\vr_n^{\frac{1}{2+\delta}}\vu_n):n\in\N\}$ using a priori estimates \eqref{eqn:4.6}.
 \end{proof}

\noindent For \[1 \le r < \infty, \;\;\delta \in (0,1),\;\; {q \in  [1, 3)}, \;\alpha \in [1,\tfrac32),\]
we introduce the following notation
\begin{align}\label{eq-spaces-4.24}
&\mathcal{X}_{\vr} := C([0,T]; L^q(\dom)),\\
\label{eq-spaces-4.25}
 &\mathcal{X}_{\sqrt{\vr}} := C([0,T]; L^{2q}(\dom))\cap \left(L^\infty(0,T; H^1(\dom)), w^\ast\right),\\
 \label{eq-spaces-4.26}
& \mathcal{X}_{\vr\vu} := L^r(0,T; L^\alpha(\dom)) \cap C_w([0,T]; L^{3/2}(\dom)),  \\
\label{eq-spaces-4.27}
 & \mathcal{X}_{\sqrt{\vr}\vu} := \bigl(L^\infty(0,T; L^2(\dom)),w^\ast \bigr),\\
 \label{eq-spaces-4.28}
  & \mathcal{X}_{\vr^{\frac{1}{2+\delta}}\vu} := \left(L^\infty(0,T; L^{2+\delta}(\dom)),w^\ast \right),\\
  \label{eq-spaces-4.29}
  & \mathcal{X}_W := C([0,T];\R)
  \end{align}
and
\begin{equation}
  \label{eqn-spaces-X_T}
\mathcal{X}_T := \mathcal{X}_{\vr}\times \mathcal{X}_{\sqrt{\vr}}\times\mathcal{X}_{\vr\vu}\times \mathcal{X}_{\sqrt{\vr}\vu}\times \mathcal{X}_{\vr^{\frac{1}{2+\delta}}\vu} \times \mathcal{X}_W .
  \end{equation}

Let us denote by $\mu_{\vr_n}$, $\mu_{\sqrt{\vr_n}}$, $\mu_{\vr_n\vu_n}$, $\mu_{\sqrt{\vr_n}\vu_n}$, $\mu_{\vr^{\frac{1}{2+\delta}}\vu}$ and $\mu_W$ the laws of $\vr_n$, $\sqrt{\vr_n}$, $\vr_n\vu_n$,  $\sqrt{\vr_n}\vu_n$, ${\vr^{\frac{1}{2+\delta}}\vu}$ and $W_n \equiv W$ respectively, on the corresponding path space. The joint law of all variables on $\mathcal{X}_T$ is denoted by $\mu_n$.

Using the above notation one could rewrite previous results on tightness in a  bit more compact way. For this let us recall definition \eqref{eqn-mu_n} of a sequence $(\mu_n)_{n=1}$.

From Lemmata~\ref{lem-tightness of rho_n}, \ref{lem-tightness of sqrt of rho_n}, \ref{lem-tightness of  rho_n u_n}, \ref{lem-tightness of  rho_n u_n in C_w} and \ref{lem-tightness of  rho_n u_n in C_w} we can deduce the following corollary.

\begin{corollary}
\label{cor-tight mu_n}
The family of measures  $\{\mathcal{L}(\mu_n) : n \in \N\}$ is tight on the space $\mathcal{X}_T$ defined above in \eqref{eqn-spaces-X_T}.
\end{corollary}

Having this we can apply the Skorokhod representation theorem in order to deduce the existence and convergence of the random variables on new probability spaces. Below we state the generalised version of the Skorokhod-Jakubowski  Theorem in the framework of  non-metric spaces, see \cite{Jak97}.

\begin{theorem}[Skorokhod-Jakubowski]\label{thm-Skor+Jak}
Let $(Z,\mathcal{Z})$ be a topological space such that there exists a sequence $(f_m)_{m\in\N}$
of continuous functions $f_m:Z\to\R$ that separate points of $Z$. Let $\mathcal{S}=\sigma((f_m)_{m\in\N})$ be the
$\sigma$-algebra generated by $(f_m)_{m\in\N}$. Then
\begin{enumerate}
\item Every compact subset of $Z$ is metrizable.
\item If $(\mu_m)_{m\in\N}$ is a tight sequence of probability measures on $(Z,\mathcal{S})$,
then there exists a subsequence $(\mu_{m_k})_{k\in\N}$, a probability space
$(\tOmega,\tF,\tP)$, and $Z$-valued Borel measurable
random variables $\mathbf{w}_k$ and $\mathbf{w}$ such that {\rm (i)} $\mu_{m_k}$ is the law of $\mathbf{w}_k$
and {\rm (ii)} $\mathbf{w}_k\to\mathbf{w}$ almost surely on $\tOmega$.
\end{enumerate}
\end{theorem}

\begin{remark}\label{rem-skor-jak} Let us observe that in the framework of Theorem \ref{thm-Skor+Jak},
\[ \mathcal{S} \subset \mathcal{B}(Z),\]
where $\mathcal{B}(Z):=\sigma(\mathcal{Z})$ is the Borel $\sigma$-field on $Z$.
It may happen, see e.g. Remark \ref{rem-L^infty spaces}, that the equality holds as well.
\end{remark}

In the remaining parts of this subsection we will show that Theorem \ref{thm-Skor+Jak} is applicable in the context of our Corollary \ref{cor-tight mu_n}.

\begin{remark}\label{rem-L^infty spaces}
Assume that $X$ is a separable Banach space with it's dual denoted by $X^\ast$. For example, $X=L^1(0,T; L^{r^\prime}(\dom))$ so that $X^\ast=L^\infty(0,T; L^r(\dom))$,
where we suppose that  $r \in (1,\infty)$ and $r^\prime \in (1,\infty)$ is the conjugate exponent, i.e.  $\frac1r+\frac1r^\prime=1$.

\medskip

Let us point out that on the space $X^\ast$ equipped with the $w^\ast$ topology, there exists a  sequence $(f_n)_{n=1}^\infty$ of continuous $\mathbb{R}$-valued functionals which separate points.
Indeed, it is sufficient to consider a  sequence  $(x_n)_{n=1}^\infty$ which is dense in the unit ball of $X$,   and define
\[f_n: X^\ast  \ni x^\ast \mapsto \langle x^\ast, x_n \rangle  \in \mathbb{R},
\]
  where $\langle \cdot, \cdot \rangle$ denotes the duality between $X^\ast$ and $X$.
  Note that  the sequence $(f_n)_{n=1}^\infty$ separates the points on $X^\ast$. It follows, that the space
  $\left(X^\ast,w^\ast \right)$ satisfies the assumptions of  the Skorokhod-Jakubowski Theorem from \cite{Jak97}, see Theorem \ref{thm-Skor+Jak}. \\
  By \cite[Theorem 3.28]{Bre11}  every ball in $X^\ast$ equipped with the weak$^\ast$ topology is  metrizable.  Moreover, by inspection of  the proof of \cite[Theorem 3.28]{Bre11}
  the weak$^\ast$ topology on ball in $X^\ast$ is generated by the functionals $(f_n)$. Therefore, the Borel $\sigma$-field on $X^\ast$ generated by the $w^\ast$ topology, is equal to the $\sigma$-field generated by the functionals  $(f_n)$.

\medskip

  Since a  ball in $X^\ast$ equipped with the weak$^\ast$ topology is also  compact by the Banach-Alaoglu Theorem, we infer that it is a Polish  compact space.\\
  Moreover,  if $Y$ is a topological space (with $\sigma(Y)$ denoting the smallest $\sigma$-field on $Y$ containing all open subsets) and function $F:X^\ast \to Y$ is sequentially continuous, then $F$ is $\mathcal{B}(X^\ast, w^\ast)/\sigma(Y)$ measurable. Indeed, for every $r>0$, the restriction $F_{|B_{X^\ast}(r)}$ of  $F$ to the closed ball $B_{X^\ast}(r)$ in $X^\ast$ of radius $r$ is continuous (as $B_{X^\ast}(r)$ is first countable) and thus is $\mathcal{B}(B_{X^\ast}(r), w^\ast)/\sigma(Y)$ measurable. Since, for any set $A\in \sigma(Y) $,
  \begin{equation}\label{eqn:??}
  F^{-1}(A)=\bigcup_{n=1}^\infty   F^{-1}(A) \cap B_{X^\ast}(n)= \bigcup_{n=1}^\infty   (F_{|B_{X^\ast}(n)})^{-1}(A),
  \end{equation}
  the claim follows. Note that the whole ${X^\ast}$ equipped with the weak$^\ast$ topology is not first countable.
\end{remark}

\begin{proposition}
\label{prop-jak+sk}
There exists a probability space $(\tOmega,\tF,\tP)$ with $\mathcal{X}_T$-valued Borel-measurable random variables $\tilde{\mu}_n=\left(\widetilde{\vr}_n, \widetilde{\vt}_n, \widetilde{\vm}_n, \widetilde{\vq}_n, \widetilde{\vrr}_n, \widetilde W_n\right)$, $n \in \N$ and $\tilde{\mu}=\left(\widetilde{\vr}, \widetilde{\vt}, \widetilde{\vm}, \widetilde{\vq},\widetilde{\vrr}, \widetilde W\right)$ and there exists a subsequence $\bigl(\mu_{k_n}\bigr)_{n=1}$,  such that the following hold:
\begin{itemize}
    \item[(1)] The law of $\tilde{\mu}_n\dela{\left(\widetilde{\vr}_n, \widetilde{\vt}_n, \widetilde{\vm}_n, \widetilde{\vq}_n,\widetilde{\vrr}_n, \widetilde W_n\right)}$ is equal to the law of $\mu_{k_n}$, $n \in \N$.
    \item[(2)] $\tilde{\mu}_n\dela{\left(\widetilde{\vr}_n, \widetilde{\vt}_n, \widetilde{\vm}_n, \widetilde{\vq}_n,\widetilde{\vrr}_n, \widetilde W_n\right)}$ converges $\tP$-almost surely in $\mathcal{X}_T$ to $\tilde{\mu}\dela{\left(\widetilde{\vr}, \widetilde{\vt}, \widetilde{\vm}, \widetilde{\vq},\widetilde{\vrr}, \widetilde W\right)}$.
\end{itemize}
\end{proposition}

\begin{proof}[Proof of Proposition \ref{prop-jak+sk}] This result is a direct consequence of Corollary~\ref{cor-tight mu_n} and Theorem~\ref{thm-Skor+Jak}\footnote{The assumptions of the theorem can be verified similarly as in \cite[Corollary~3.12]{Brz+Mot_2013}.}
\end{proof}

In view of the definition (\ref{eqn-spaces-X_T},\ref{eq-spaces-4.24},\ref{eq-spaces-4.25},\ref{eq-spaces-4.26}) of the space  $\mathcal{X}_T$, the property (2) from Proposition \ref{prop-jak+sk} implies that the following hold $\tP$-almost surely,
\begin{align}
\widetilde{\vr}_n(0) & \to \widetilde{\vr}(0) \mbox{ in } L^q(\dom),
\label{eqn:vr_n-convergence}\\
\widetilde{\vt}_n (0) & \to  \widetilde{\vt}(0) \mbox{ in } L^{2q}(\dom),
\label{eqn:vt_n-convergence}
\\
 \widetilde{\vm}_n(0) & \to \widetilde{\vm}(0) \mbox{ weakly in } L^{3/2}(\dom).
\label{eqn:vm_n-convergence}
\end{align}


\subsection{Identification of the limiting random variables}
\label{sec-identification}
We are now going to identify $\widetilde\vt_n, \widetilde\vm_n$, $\wt\vq_n$ and $\widetilde\vt, \widetilde\vm$, $\wt\vq$.

In the remaining parts of this subsection we choose and fix  $\delta \in (0,1)$ and $q \in [1, 3)$ such that
\begin{equation}\label{eqn-q-delta}
q(1+2\delta) \ge 3(1+\delta).
\end{equation}
To simplify the notation we are going to denote a subsequence $\mu_{k_n}$ by $\mu_n$, for $n\in \mathbb{N}$.

\begin{lemma}\label{lem-identification}
 { $\wt\Prob$-a.s. the following equalities hold for every $n \in \mathbb{N}$, a.e. in $\dom \times (0,T)$ }
\eq{\label{eqn-identification-n}
(\widetilde\vt_n, \widetilde\vm_n,\widetilde{\vq}_n)&= (\sqrt{\widetilde\vr_n},  \widetilde\vr_n^{\frac{1+\delta}{2+\delta}}\widetilde \vrr_n,\widetilde\vr_n^{\frac{\delta}{2(2+\delta)}}\widetilde \vrr_n).
}
\end{lemma}
\begin{proposition}\label{prop-identification} Under assumptions of Lemma \ref{lem-identification}, $\wt\Prob$-a.s.  the following identity holds  in $\dom \times (0,T)$
\eq{\label{eqn-identification}
(\widetilde\vt, \widetilde\vm,\widetilde{\vq})&= (\sqrt{\widetilde\vr},  \widetilde\vr^{\frac{1+\delta}{2+\delta}}\widetilde \vrr,\widetilde\vr^{\frac{\delta}{2(2+\delta)}}\widetilde \vrr).
}
\end{proposition}

\begin{proof}[Proof of Lemma \ref{lem-identification}] Let us choose and fix $n \in \mathbb{N}$. Let us denote, on $\mathcal{B}(\Xr \times \Xsr)$,  \begin{align*}
\Upsilon_n:= \mathcal{L}(\vr_n, \sqrt{\vr}_n), \;\;\; \wt\Upsilon_n:=\mathcal{L}(\wt\vr_n, \wt\vt_n).
\end{align*}
Since by  Proposition~\ref{prop-jak+sk}  \[\mathcal{L}\left(\widetilde{\vr}_n, \widetilde{\vt}_n, \widetilde{\vm}_n, \widetilde{\vq}_n,\widetilde{\vrr}_n\right) = \mathcal{L}\left({\vr}_n, {\sqrt{\vr}}_n, {\vr}_n\vu_n, {\sqrt{\vr}_n\vu_n},\vr_n^{\frac{1}{2+\delta}}\vu_n\right) \mbox{  on }\mathcal{B}(\mathcal{X}_T),\]
 we infer that   $\wt\Upsilon_n  = \Upsilon_n $ on $\mathcal{B}(\Xr \times \Xsr)$.
Since  $\vr_n = \left(\sqrt{\vr_n}\right)^2$, $\Prob$-a.s. we infer that  \[\E\|\vr_n - \left(\sqrt{\vr_n}\right)^2\|_{L^1(0,T;L^1(\dom))} = 0.\]
Let us  observe that the map
\begin{equation}\label{eqn-beta map}
\beta : \Xsr \ni x \mapsto x^2 \in \Xr \mbox{ is continuous }
\end{equation}
 and hence Borel measurable. Moreover, the map $\|\cdot\| : \Xr \to \R$ is continuous, and linear combination of measurable maps is measurable. Combining all this with equality $\Upsilon_n = \wt\Upsilon_n$ on $\mathcal{B}(\Xr \times \Xsr)$ we get following chain of identities
\begin{align*}
    0&= \E\|\vr_n - \beta(\sqrt{\vr_n})\|_{\Xr} = \int_\Omega \|\vr_n(\omega) - \beta(\sqrt{\vr_n})(\omega)\|_{\Xr} d\Prob(\omega)\\
    & = \int_{\Xr \times \Xsr} \|x - \beta(y)\|_{\Xr}\,d\Upsilon_n(x,y) = \int_{\Xr \times \Xsr} \|x - \beta(y)\|_{\Xr}\,d\wt\Upsilon_n(x,y) \\
    & = \int_{\widetilde{\Omega}} \|\wt\vr_n(\omega) - \beta(\wt\vt_n(\omega))\|_{\Xr} d\wt\Prob(\omega) = \wt\E\|\wt\vr_n - \beta(\wt\vt_n)\|_{\Xr}.
\end{align*}
Thus we infer that
\begin{equation}\label{eqn:4.33}
\wt\Prob\mbox{-a.s. }\wt\vt_n = \sqrt{\wt\vr_n} \mbox{ in }\Xr.
\end{equation}

We will use a similar argument to identify $\wt\vm_n$. We start by defining a map
\begin{equation}\label{eqn:Gamma-4.33}
  \Gamma \colon \Xr \times \mathcal{X}_{\vr^{\frac{1}{2+\delta}}\vu} \ni (\mathrm{v}, \vw) \mapsto \mathrm{v}^{\frac{1+\delta}{2+\delta}} \vw \in \big(L^\infty(0,T; L^{3/2}(\dom)), w^\ast\big).
\end{equation}

Firstly, we observe that $\Gamma$ is a well-defined function, see also Lemma~\ref{lem-Gamma}. Secondly by  Lemma~\ref{lem-Gamma}, we infer that
for every $\bpsi \in L^1(0,T; L^3(\dom))$, the function
\[
\Xr \times \mathcal{X}_{\vr^{\frac{1}{2+\delta}}\vu} \ni (\mathrm{v}, \vw) \mapsto \langle \Gamma(\rm{v}, \vw), \bpsi \rangle \in \R,
\]
  where $\langle \cdot, \cdot \rangle$ denotes the duality between $L^\infty(0,T; L^{3/2}(\dom))$ and $L^1(0,T; L^3(\dom))$,
  is sequentially  continuous.

Due to the sequential continuity of the map $\Gamma$, the map $g$
\begin{equation}
\label{eqn:iden_mom_2}
\begin{split}
    g \colon &\Xr \times \Xrdu \times \Xru \times L^1(0,T; L^3(\dom)) \longrightarrow \R \\
    (\rm{v}, \vw, \vc{z}, \bpsi) &\mapsto \int_0^T \int_\dom \left(\vc{z}(t,x) - \Gamma(\rm{v}(t,x), \vw(t,x))\right)\cdot \bpsi(t,x)\,\dx\,\dt
\end{split}
\end{equation}
is sequentially continuous too and hence, by Remark \ref{rem-L^infty spaces}, measurable. Since, for every $\omega \in \Omega$
\[
\vr_n\vu_n = \vr_n^{\frac{1+\delta}{2+\delta}}\vr_n^{\frac{1}{2+\delta}}\vu_n,\]
by the measurability of the map $g$ and of random variables $\vr_n, \vr_n^{\frac{1}{2+\delta}}\vu_n$ and $\vr_n\vu_n$ from $\Omega$ to $\Xr, \Xrdu$ and $\Xru$ respectively, we  infer that
for every $\psi \in L^1(0,T; L^3(\dom))$
\eq{
\label{eqn:iden_mom_1}
& \E g(\vr_n, \vr_n^{\frac{1}{2+\delta}}, \vr_n\vu_n, \bpsi) \\
& \quad = \E \int_0^T \int_\dom \left(\vr_n\vu_n(t,x) - \vr_n^{\frac{1+\delta}{2+\delta}}(t,x)\vr_n^{\frac{1}{2+\delta}}(t,x)\vu_n(t,x)\right)\cdot \bpsi(t,x)\,\dx\,\dt = 0.
}
Let us denote $\mathcal{L}\left(\vr_n, \vr_n^{\frac{1}{2+\delta}}\vu_n, \vr_n\vu_n\right)$ and $\mathcal{L}\left(\wt\vr_n, \wt\vrr_n, \wt\vm_n\right)$ on $\mathcal{B}(\Xr \times \Xrdu \times \Xru)$ by $\nu_n$ and $\wt\nu_n$ respectively. From the equality of joint laws of
\begin{equation} \begin{split}
\left({\vr}_n, \sqrt{{\vr}_n} , \vr_n\vu_n, \sqrt{\vr_n}\vu_n,\vr_n^{\frac{1}{2+\delta}}\vu_n\right) \quad \text{and} \quad \left(\widetilde{\vr}_n, \widetilde{\vt}_n, \widetilde{\vm}_n, \widetilde{\vq}_n,\widetilde{\vrr}_n\right)
\end{split}\end{equation}
on $\mathcal{X}_T$, which is a consequence of Proposition \ref{prop-jak+sk},  we have
\eq{
\label{eqn:iden_mom_3}
\nu_n = \wt\nu_n \; \mbox{ on } \mathcal{B}\left(\Xr\times\Xru\times\mathcal{X}_{\vr^{\frac{1}{2+\delta}}\vu}\right).
}
Thus, from identity \eqref{eqn:iden_mom_1} and equality of laws \eqref{eqn:iden_mom_3} we have for $\bpsi \in L^1(0,T; L^3(\dom))$
\begin{align*}
    0&= \E \int_0^T \int_\dom \left(\vr_n\vu_n(t,x) - \vr_n^{\frac{1+\delta}{2+\delta}}(t,x)\vr_n^{\frac{1}{2+\delta}}(t,x)\vu_n(t,x)\right)\cdot \bpsi(t,x)\,\dx\,\dt \\
    & = \int_\Omega \left(\int_0^T \int_\dom \left(\vr_n\vu_n(\omega) -\Gamma(\vr_n(\omega), \vr_n^{\frac{1}{2+\delta}}\vu_n(\omega))\right)\cdot\bpsi\,\dx\,\dt\right)\mathrm{d}\Prob(\omega)\\
    & = \int_{\Xr \times \Xrdu \times \Xru}\left(\int_0^T \int_\dom \left(\vc{z} - \Gamma(\rm{v},\vw)\right)\cdot\bpsi\,\dx\,\dt\right)\mathrm{d}\nu_n\left(\rm{v},\vw,\vc{z}\right)\\
    & = \int_{\Xr \times \Xrdu \times \Xru} g(\rm{v}, \vw, \vc{z}, \bpsi)\,\mathrm{d}\nu_n\left(\rm{v},\vw,\vc{z}\right)\\
    & = \int_{\Xr \times \Xrdu \times \Xru} g(\rm{v}, \vw, \vc{z}, \bpsi)\,\mathrm{d}\wt\nu_n\left(\rm{v},\vw,\vc{z}\right)\\
    & = \int_{\Xr \times \Xrdu \times \Xru}\left(\int_0^T \int_\dom \left(\vc{z} - \Gamma(\rm{v},\vw)\right)\cdot\bpsi\,\dx\,\dt\right)\mathrm{d}\wt\nu_n\left(\rm{v},\vw,\vc{z}\right)\\
    & = \int_{\tOmega}\left(\int_0^T \int_\dom\left( \wt\vm_n(\omega) - \Gamma(\wt\vr_n(\omega), \wt\vrr_n(\omega))\right)\cdot\bpsi\,\dx\,\dt\right) d\wt\Prob(\omega)\\
    & = \wt\E \int_0^T \int_\dom \left(\wt\vm_n(t,x) - \wt\vr_n^{\frac{1+\delta}{2+\delta}}(t,x)\wt\vrr_n(t,x)\right)\cdot \bpsi(t,x)\,\dx\,\dt,
\end{align*}
where we have used the measurability of $\wt\vm_n, \wt\vr_n$ and $\wt\vrr_n$ (which is due to Proposition~\ref{prop-jak+sk}) and the measurability  of the map $g$  defined in \eqref{eqn:iden_mom_2}. Therefore,
in view of separability of the space $ L^1(0,T; L^3(\dom))$ we deduce that
\begin{align}\label{eqn-identification of tilde{m}_n}
\wt\vm_n = \wt\vr_n^{\frac{1+\delta}{2+\delta}}\wt\vrr_n \mbox{ a.e. in } \dom \times (0,T), \;\; \wt\Prob\mbox{-a.s..}
\end{align}

We can similarly identify $\wt\vq_n$. For this purpose we consider another pair of maps $\widehat\Gamma$ and $\widehat g$ defined by
\begin{equation}\label{eqn:widehatGamma}
\widehat \Gamma \colon \Xr \times \mathcal{X}_{\vr^{\frac{1}{2+\delta}}\vu} \ni (\mathrm{v}, \vw) \mapsto \mathrm{v}^{\frac{\delta}{2(2+\delta)}} \vw \in \big(L^\infty(0,T; L^{2}(\dom)), w^\ast\big),
\end{equation}

\begin{equation}
\label{eqn:iden_mom_4}
\begin{split}
    \widehat g \colon &\Xr \times \mathcal{X}_{\vr^{\frac{1}{2+\delta}}\vu} \times \mathcal{X}_{\sqrt{\vr}\vu} \times L^1(0,T; L^2(\dom)) \longrightarrow \R \\
    (\rm{v}, \vw, \vc{z}, \bpsi) &\mapsto \int_0^T \int_\dom \left({\vc{z}(t,x) - \widehat\Gamma(\rm{v}(t,x), \vc{w}(t,x))}\right)\cdot \bpsi(t,x)\,\dx\,\dt
\end{split}
\end{equation}
$\widehat \Gamma$ is well-defined as well as sequentially continuous, see Lemma~\ref{lem:gamma2} for details, and so is $\widehat g$.

Arguing as above we can check that
\begin{align*}
    0&= \E \int_0^T \int_\dom \left(\sqrt{\vr_n}\vu_n(t,x) - \vr_n^{\frac{\delta}{2(2+\delta)}}(t,x)\vr_n^{\frac{1}{2+\delta}}(t,x)\vu_n(t,x)\right)\cdot \bpsi(t,x)\,\dx\,\dt \\
    & = \int_\Omega \left(\int_0^T \int_\dom \left(\sqrt{\vr_n}\vu_n(\omega) - \widehat \Gamma(\vr_n(\omega), \vr_n^{\frac{1}{2+\delta}}\vu_n(\omega))\right)\cdot\bpsi\,\dx\,\dt\right)\mathrm{d}\Prob(\omega)\\
    & = \int_{\Xr \times \Xrdu \times \mathcal{X}_{\sqrt{\vr}\vu}}\left(\int_0^T \int_\dom \left(\vc{z} - \widehat\Gamma(\rm{v},\vw)\right)\cdot\bpsi\,\dx\,\dt\right)\mathrm{d}\nu_n\left(\rm{v},\vw,\vc{z}\right)\\
    &= \int_{\Xr \times \Xrdu \times \mathcal{X}_{\sqrt{\vr}\vu}} \widehat g(\rm{v}, \vw, \vc{z}, \bpsi)\,\mathrm{d}\nu_n\left(\rm{v},\vw,\vc{z}\right)\\
    & = \int_{\Xr \times \Xrdu \times \mathcal{X}_{\sqrt{\vr}\vu}} \widehat g(\rm{v}, \vw, \vc{z}, \bpsi)\,\mathrm{d}\wt\nu_n\left(\rm{v},\vw,\vc{z}\right)\\
    & = \int_{\Xr \times \Xrdu \times \mathcal{X}_{\sqrt{\vr}\vu}}\left(\int_0^T \int_\dom \left(\vc{z} - \widehat\Gamma(\rm{v},\vw)\right)\cdot\bpsi\,\dx\,\dt\right)\mathrm{d}\wt\nu_n\left(\rm{v},\vw,\vc{z}\right)\\
    & = \int_{\tOmega}\left(\int_0^T \int_\dom\left( \wt\vq_n(\omega) - \widehat\Gamma(\wt\vr_n(\omega), \wt\vrr_n(\omega))\right)\cdot\bpsi\,\dx\,\dt\right) d\wt\Prob(\omega)\\
    & = \wt\E \int_0^T \int_\dom \left(\wt\vq_n(t,x) - \wt\vr_n^{\frac{\delta}{2(2+\delta)}}(t,x)\wt\vrr_n(t,x)\right)\cdot \bpsi(t,x)\,\dx\,\dt,
\end{align*}
holds for all test functions $\bpsi \in L^1(0,T; L^2(\dom))$. Here $\nu_n$ and $\wt\nu_n$ denote $\mathcal{L}(\vr_n, \vr_n^{\frac{1}{2+\delta}}, \sqrt{\vr_n}\vu)$ and $\mathcal{L}(\wt\vr_n, \wt\vrr_n, \wt\vq_n)$ respectively and $\nu_n = \wt \nu_n$ on $\mathcal{B}\left(\Xr \times \Xrdu \times \mathcal{X}_{\sqrt{\vr}\vu}\right)$. Therefore, we deduce that

\begin{align}\label{eqn-identification of tilde{q}_n}
\wt\vq_n = \wt\vr_n^{\frac{\delta}{2(2+\delta)}}\wt\vrr_n \mbox{ a.e. in }\dom \times [0,T], \;\; \wt\Prob\mbox{-a.s.}.
\end{align}
 This concludes the proof of Lemma \ref{lem-identification}.
\end{proof}

\begin{proof}[Proof of Proposition \ref{prop-identification}]
In order to identify the limit $\widetilde\vt$ we note that
\eq{
\wt\vt_n\wt\vt_n\to\lr{\wt\vt}^2 \quad\text{strongly\ in}\ L^1(0,T; L^1(\dom)), \quad\wt\Prob\mbox{-a.s.}
}
as a consequence of the strong convergence of $\wt\vt_n$ in $\mathcal{X}_{\sqrt{\vr}}$. But by \eqref{eqn:4.33} and second assertion of Proposition \ref{prop-jak+sk}  we have  that
\eq{
\wt\vt_n\wt\vt_n=\wt\vr_n\to\wt\vr \quad\text{strongly\ in}\ L^1(0,T; L^1(\dom)), \quad\wt\Prob\mbox{-a.s.}.
}
Therefore $\wt\vr =(\wt\vt)^2$, {a.e. in $\dom\times(0,T)$} $\wt\Prob$-a.s.. To identify $\widetilde\vm$ we first notice that thanks to strong convergence of $\wt\vr_n$ and weak convergence of $\wt\vrr_n$
\eq{
\wt\vr_n^\frac{1+\delta}{2+\delta}\wt\vrr_n\to \wt\vr^\frac{1+\delta}{2+\delta}\wt\vrr,\quad \text{weakly*\ in}\ L^\infty(0,T; L^\frac{3}{2}(\dom)), \quad\wt\Prob\mbox{-a.s.},}
and then because of previous identification and Proposition~\ref{prop-jak+sk}
\eq{
\wt\vr_n^\frac{1+\delta}{2+\delta}\wt\vrr_n=\wt\vm_n\to\wt\vm\quad \text{ in}\  C_{w}(0,T; L^\frac{3}{2}(\dom)), \quad\wt\Prob\mbox{-a.s.}.
}
Therefore $\wt\vm=\wt\vr^\frac{1+\delta}{2+\delta}\wt\vrr$ a.e. in $\dom\times(0,T)$ $\wt\Prob$-a.s.. Identification of $\wt\vq$ follows in a similar manner.
 This concludes the proof of  Proposition \ref{prop-identification}.
\end{proof}

\begin{lemma}
\label{lem-tilde u_n}
Let $n \in \N$. Then, there exists a random variable $\widetilde{\vu}_n$ defined on the new probability space $(\tOmega,\tF,\tP)$, obtained in Proposition~\ref{prop-jak+sk}, for which
\eq{
(\widetilde{\vq}_n,\widetilde{\vm}_n)&=(\sqrt{\wt\vr_n}\widetilde{\vu}_n,\wt\vr_n\widetilde{\vu}_n) \quad {a.e.\  in \ \dom\times(0,T)}\ \wt\Prob\mbox{-a.s.}.
}
\end{lemma}

\begin{proof}[Proof of Lemma \ref{lem-tilde u_n}] From the equality of joint laws and estimate \eqref{eqn-4.2}, we deduce that
\eq{\label{r_bound}
\left\|\widetilde{\vrr}_n\right\|_{L^{2+\delta}(\tOmega; L^\infty(0,T; L^{2+\delta}(\dom)))} \le C,
}
from this it follows that $\widetilde{\vrr}_n$ is bounded a.e. in $\wt\Omega\times(0,T)\times\dom$. Therefore, by identity  \eqref{eqn-identification-n} it follows that
\eq{\label{qmr}
\wt\vr_n=0\implies\wt\vq_n= \widetilde\vr_n^{\frac{\delta}{2(2+\delta)}}\widetilde \vrr_n=0\\
\wt\vr_n=0\implies\wt{\vm}_n= \widetilde\vr_n^{\frac{1+\delta}{2+\delta}}\widetilde \vrr_n=0
}
a.e. in $\wt\Omega\times(0,T)\times\dom$. Now we define
\eq{\label{def_tun}
 \wt\vu_n=\left\{
\begin{array}{lll}
  \wt\vrr_n \wt\vr_n^{-\frac{1}{2+\delta}} & {\rm if}  & \wt\vr_n\neq0,\\
   0  & {\rm if}  & \wt\vr_n=0.
\end{array}\right.
}
Therefore, from \eqref{qmr} it follows that
\eq{\label{def_tqn}
 \wt\vq_n=\left\{
\begin{array}{lll}
 \sqrt{ \wt\vr_n}\wt\vu_n    & {\rm for}  & \wt\vr_n\neq0,\\
   0  & {\rm for}  & \wt\vr_n=0,
\end{array}\right.
}
but due to definition of $\wt\vu_n$, \eqref{def_tun}, we have that
\eq{
 \wt\vq_n=  \sqrt{ \wt\vr_n}\wt\vu_n
}
a.e. in $\wt\Omega\times(0,T)\times\dom$. Using the same argument for $\wt\vm_n$ we get the second part of the  equality. \end{proof}


\subsection{Convergence of the convective term and of the stress tensor}
\label{sec-conv_den-mom}

\begin{lemma}
\label{lem-tilde u}
Let $\widetilde{\vu}_n$ be as defined in Lemma~\ref{lem-tilde u_n}, then there exists a random variable $\wt \vu$ such that
\begin{align}
\label{eqn-tilde u}
\wt\vr_n\widetilde{\vu}_n\otimes\widetilde{\vu}_n\to \wt\vr\widetilde{\vu}\otimes\wt\vu \quad  in \ L^1(0,T; L^1(\dom))\ \wt\Prob\mbox{-a.s.}.
\end{align}
\end{lemma}

\begin{proof}[Proof of Lemma \ref{lem-tilde u}] So far we know that $\wt\Prob$-a.s.
\eq{\label{mqr_c}
&\wt\vm_n\to \wt\vm \quad \text{in }\quad  L^r(0,T; L^\alpha(\dom)) \cap C_w([0,T]; L^{3/2}(\dom)),\qquad r\in[1,\infty),\; \alpha \in [1,\tfrac32)\\
&\wt\vq_n\to \wt\vq \quad \text{weakly* in }\quad  L^\infty(0,T; L^{{2}}(\dom))\\
&\wt\vrr_n\to\wt\vrr \quad \text{weakly* in }\quad  L^\infty(0,T; L^{2+\delta}(\dom)).
}
From Lemma \ref{lem-tilde u_n} we know that there exists a random variable $\wt\vu_n$ s.t. $\wt\Prob$-a.s.
\eq{\label{4.33}
&\wt\vr_n\widetilde{\vu}_n\to \wt\vm \quad \text{in }\quad  L^r(0,T; L^\alpha(\dom))\  \text{and\ weakly* in} \ {L^\infty([0,T]; L^{3/2}(\dom))}\\
&\sqrt{\wt\vr_n}\widetilde{\vu}_n \to \wt\vq \quad \text{weakly* in }\quad  L^\infty(0,T; L^{{2}}(\dom)).
}
Indeed, knowing  that $\wt\vm_n=\wt\vr_n\wt\vu_n$ a.e. in $\dom\times(0,T)$ and that $\wt\vm_n\to\wt\vm$ in $L^r(0,T; L^\alpha(\dom)) \cap C_w([0,T]; L^{3/2}(\dom))$ we deduce only the weaker notion of convergence for $\wt\vr_n\widetilde{\vu}_n$.
From  \eqref{mqr_c}$_3$ and due to the lower semi-continuity of norms we have from \eqref{r_bound}
\eq{\label{r_bound_lim}
\left\|\widetilde{\vrr}\right\|_{L^{2+\delta}(\tOmega; L^\infty(0,T; L^{2+\delta}(\dom)))} \le C.
}
Proceeding as in the proof of Lemma  \ref{lem-tilde u_n}  we can define a new random field $\wt\vu$
\eq{\label{eqn-u tilde-def}
 \wt\vu=\left\{
\begin{array}{lll}
 {\wt\vr^{\frac{-1}{2+\delta}}}\,{ \wt\vrr}    & {\rm for}  & \wt\vr\neq0,\\
   0  & {\rm for}  & \wt\vr=0.
\end{array}\right.
}
Let us point out that for fixed $(\omega,t)\in \tOmega\times [0,T]$,
This definition in conjunction with the 2nd and 3rd parts of identification \eqref{eqn-identification} in Proposition \ref{prop-identification} implies that $\wt\Prob$-a.s.
\eq{
(\widetilde{\vm},\widetilde{\vq})&=(\wt\vr\widetilde{\vu},\sqrt{\wt\vr}\widetilde{\vu}), \;\; \mbox{a.e. in $(0,T)\times\dom$}.
}
 In particular, by Proposition \ref{prop-jak+sk} and definitions (\ref{eq-spaces-4.26}-\ref{eq-spaces-4.27}) , $\wt\Prob$-a.s., $(\wt\vr\widetilde{\vu},\sqrt{\wt\vr}\widetilde{\vu})$ belongs to
$ \mathcal{X}_{\vr\vu} \times  \mathcal{X}_{\sqrt{\vr}\vu}$, i.e. to $ \Bigl( L^r(0,T; L^\alpha(\dom)) \cap C_w([0,T]; L^{3/2}(\dom))\Bigr) \times  \bigl(L^\infty(0,T; L^2(\dom)),w^\ast \bigr)$.

Therefore, to conclude we only need to improve the convergence result from \eqref{4.33} to show that
\eq{
&\sqrt{\wt\vr_n}\widetilde{\vu}_n \to \sqrt{\wt\vr}\widetilde{\vu} \quad \text{ in }\quad  L^2(0,T; L^2(\dom)) \quad \wt\Prob\mbox{-a.s.}.
}
To prove it, we proceed similarly to Mellet and Vasseur. As in \eqref{r_bound_lim}, by applying the Fatou Lemma, we deduce that
\eq{\label{mvest1}
\int_{\dom}\wt\vr|\wt\vu|^{2+\delta}\,\dx
\leq\int_{\dom}{\wt\vrr}^{2+\delta}\,\dx
= \int_{\dom}\liminf_{n\to\infty} {\wt\vrr_n}^{2+\delta}\,\dx \leq \liminf_{n\to\infty} \int_{\dom} {\wt\vrr_n}^{2+\delta}\,\dx <\infty.
}
Note that the left inequality follows from the definition of \eqref{eqn-u tilde-def} for $\wt\vr\neq0$,
when $\wt\vr=0$ this integral is equal to $0$.
We will now show that
\eq{\label{conv_sqru}
&\sqrt{\wt\vr_n}\widetilde{\vu}_n=\frac{\wt\vm_n}{\sqrt{\wt\vr_n}} \to  \frac{\wt\vm}{\sqrt{\wt\vr}}=\sqrt{\wt\vr}\widetilde{\vu}\quad \text{almost \ everywhere}.
}
Firstly, due to \eqref{eqn-u tilde-def} and strong convergence of $\wt\vm_n$ and $\wt\vr_n$, the convergence \eqref{conv_sqru} is true on the set $\{\wt\vr(t,x)\neq0\}$. Secondly, this is also true on the set where $\wt\vu_n$ is bounded and $\{\wt\vr(t,x)=0\}$. Indeed,  we have $\sqrt{\wt\vr_n}\wt\vu_n\chi_{\left|\wt\vu_n\right|\leq M}\leq M\sqrt{\wt\vr_n}\to 0$.

To conclude that
\eq{\label{convae}
\sqrt{\wt\vr_n}\wt\vu_n\chi_{\left|\wt\vu_n\right|\leq M} \to \sqrt{\wt\vr} \wt\vu \chi_{| \wt\vu|\leq M}\quad \text{almost \ everywhere}.
}
we use that
$$
\{\left|\wt\vu_n\right|\leq M\}=\left\{\left|\frac{\wt\vm_n}{\wt\vr_n}\right|\leq M\right\},
$$
and so for $\wt\vr\neq0$ using the strong convergence of $\wt\vr_n$, $\wt\vm_n$ we have
$$
\chi_{\left|\frac{\wt\vm_n}{\wt\vr_n}\right|\leq M}\to\chi_{\left|\frac{\wt\vm}{\wt\vr}\right|\leq M}=\chi_{| \wt\vu|\leq M} \quad \text{almost \ everywhere}.
$$
We are now ready to conclude, we first split
\eq{
\intTO{\left|\sqrt{\wt\vr_n}\widetilde{\vu}_n-\sqrt{\tilde\vr}\tilde \vu\right|^2}
\leq& \intTO {\left|\sqrt{\wt\vr_n}\widetilde{\vu}_n\chi_{\left|\wt\vu_n\right|\leq M}-\sqrt{\tilde\vr}\tilde \vu \chi_{|\tilde \vu|\leq M}\right|^2}\\
&+4\intTO{\left|\sqrt{\wt\vr_n}\widetilde{\vu}_n\chi_{\left|\wt\vu_n\right|\geq M}\right|^2}\\
&+4\intTO{\left|\sqrt{\tilde\vr}\tilde \vu \chi_{|\tilde \vu|\geq M}\right|^2}.
}
Due to \eqref{convae} and uniform boundedness of the integrand the first term converges to 0 when $n\to\infty$. For the other two terms we make use of estimate \eqref{r_bound} and inequality \eqref{mvest1}, we therefore have
\eq{
4\intTO{\left|\sqrt{\wt\vr_n}\widetilde{\vu}_n\chi_{\left|\wt\vu_n\right|\geq M}\right|^2}
\leq&
\frac{4}{M^\delta}\intTO{\wt\vr_n|\wt\vu_n|^{2+\delta}}\\
\leq & \frac{4}{M^\delta}\intTO{|\wt\vrr_n|^{2+\delta}}\leq  \frac{C}{M^\delta},
}
and
\eq{
4\intTO{\left|\sqrt{\wt\vr}\widetilde{\vu}\chi_{\left|\wt\vu\right|\geq M}\right|^2}
\leq&
\frac{4}{M^\delta}\intTO{\wt\vr|\wt\vu|^{2+\delta}}\\
\leq & \frac{4}{M^\delta}\intTO{|\wt\vrr|^{2+\delta}}\leq  \frac{C}{M^\delta},
}
and so both terms  converge to $0$ as $M\to\infty$ (recall $\delta>0$).
\end{proof}

In the next Lemma we will show that the random field $\vu$ is weakly differentiable in some specific sense. To explain it's meaning let us observe that if a locally integrable $\vu$ is weakly differentiable with $\nabla \vu$  being also locally integrable and $\vr$ is a locally bounded function,  then for every test function $\bpsi \in C^\infty(\dom)$ the following identity holds
\begin{align}
\label{eqn-IP}
\int_{\mathcal{O}} \Div( \vr\, \Grad \vu) \psi \, dx = \int_{\mathcal{O}} \vr \vu\cdot \lap \bpsi \, dx + \int_{\mathcal{O}} \Grad \vr \otimes\vu:\Grad\bpsi \, dx.
\end{align}

\begin{lemma}[Convergence of the diffusion terms]\label{lem-convergence diffusion}
Let $\wt\vu_n$, $\wt \vu$ be as defined in Lemma~\ref{lem-tilde u_n} and Lemma~\ref{lem-tilde u} respectively. Then $\wt\Prob$-a.s.
\begin{equation*}
\Div(\wt\vr_n \Grad \wt\vu_n) \to \Div( \wt\vr\, \Grad \wt\vu) \quad \mbox{in } \mathcal{D}^\prime,
\end{equation*}
in the following sense. For every $ \bpsi \in W^{2,\infty}(\dom)$ and every $t\in [0,T]$,  $\wt\Prob$-a.s., the following equality holds
\begin{equation}
    \label{eqn:stress_convg}
    \begin{split}
    & \lim_{n\to\infty}\left(\inttO{\sqrt{\wt\vr_n}\sqrt{\wt\vr_n}\wt\vu_n \lap \bpsi} + 2\inttO{\Grad\sqrt{\wt\vr_n}\otimes\sqrt{\wt\vr_n}\wt\vu_n:\Grad\bpsi}\right) \\
    &\quad = \inttO{\wt\vr \wt\vu\cdot \lap \bpsi} + \inttO{ \Grad \wt\vr \otimes\wt\vu:\Grad\bpsi}.
    \end{split}
\end{equation}
\end{lemma}

\begin{proof}[Proof of Lemma \ref{lem-convergence diffusion}] It is enough to prove the Lemma for $t=T$.
Let us choose and fix  a test function $\bpsi \in C^\infty(\dom)$ and define
\begin{align*}
    I_1^n &:= \intTO{\sqrt{\wt\vr_n}\sqrt{\wt\vr_n}\wt\vu_n \lap \bpsi} \\
    I_2^n &:= 2\intTO{\Grad\sqrt{\wt\vr_n}\otimes\sqrt{\wt\vr_n}\wt\vu_n:\Grad\bpsi}
\end{align*}

From Proposition~\ref{prop-jak+sk} and Lemma~\ref{lem-identification} we know that $\sqrt{\wt\vr_n} \to \sqrt{\wt\vr}$ strongly in $C([0,T]; L^{2q}(\dom))$, $q \in [1,3)$ and weakly$^\ast$ in $L^\infty(0,T; H^1(\dom))$ $\wt \Prob$-a.s.. Moreover, from Lemma~\ref{lem-tilde u} we have $\sqrt{\wt\vr_n}\wt\vu_n \to \sqrt{\wt\vr} \wt\vu$ strongly in $L^2(0,T; L^2(\dom))$ $\wt \Prob$-a.s.. Using these convergence results we infer that for $\bpsi \in C^\infty(\dom)$ or in particular for $\bpsi \in W^{2,\infty}(\dom)$
\begin{align*}
I_1^n &\to \intTO{\wt\vr \wt\vu\cdot \lap \bpsi} \qquad \wt\Prob\mbox{-a.s.},\\
I_2^n &\to \intTO{ \Grad \wt\vr \otimes\wt\vu:\Grad\bpsi} \qquad \wt\Prob\mbox{-a.s.},
\end{align*}
and hence we infer that
\[
I_1^n + I_2^n  \to  \intTO{\wt\vr \wt\vu\cdot \lap \bpsi} + \intTO{ \Grad \wt\vr \otimes\wt\vu:\Grad\bpsi}  \qquad \wt\Prob\mbox{-a.s.}.
\]
Once we have proved Lemma for $\bpsi \in C^\infty(\dom)$ we can can easily deduce it for an arbitrary $\bpsi \in W^{2,\infty}(\dom)$.
\end{proof}


\section{Existence of martingale solutions}
\label{sec-martingale}

Our aim is to show that the limiting system $\left(\widetilde{U}, \widetilde W,  \widetilde \vr, \widetilde \vu\right)$, where $\widetilde{U}=(\tOmega,\tF,\tP)$ ,  is a martingale solution to \eqref{eqn-1.1}, where $\widetilde \vu$ is as defined in the proof of  Lemma~\ref{lem-tilde u} and we use notation introduced in Proposition \ref{prop-jak+sk}. The very first step in this direction is to show that $\widetilde W_n$, $n \in \N$, and $\widetilde W$ are indeed $\R$-valued Brownian motion. This follows directly from the following lemmas taken from \cite[Lemma~5.2 and proof]{Brz+Gold_Jeg_2013}. The approach presented in this section is similar to the one given in paper \cite{BrDhGi20}.

\begin{lemma}
 \label{lem-Brownian Motion}
Suppose that a  process $\left(\widetilde W_n(t)\right)_{t \in [0,T]}$, defined on $\left(\widetilde \Omega, \widetilde{\mathcal{F}}, \tP\right)$, has the same law {on $C([0,T];\mathbb{R})$} as the $\R$-valued Brownian motion $W$, defined on $\left(\Omega, \mathcal{F}, \mathbb{P}\right)$. Then $\widetilde W_n$ is also an $\R$-valued Brownian motion on $\left(\widetilde \Omega, \widetilde{\mathcal{F}}, \tP\right)$.
\end{lemma}

\begin{lemma}
 \label{lem-Brownian Motion-2}
 The process $\left(\widetilde W(t)\right)_{t \in [0,T]}$ is an $\R$-valued Brownian motion on $\left(\widetilde \Omega, \widetilde{\mathcal{F}}, \tP\right)$. If $s\in [0,T)$, then the  increments $\widetilde W(t) - \widetilde W(s)$, $t\in [s,T]$,  are independent of the $\sigma$-algebra generated by $\widetilde \vu(r)$ and $\widetilde W(r)$ for $r \in [0,s]$.
\end{lemma}

Below we collect all the convergence results obtained in Section~\ref{sec-existence}, which will be used later to prove the existence of a martingale solution. Let $\widetilde\vu_n$ and $\widetilde\vu$ be the processes as specified in Lemmas~\ref{lem-tilde u_n}--\ref{lem-convergence diffusion} and $q \in [1,3)$. Then, the following assertions hold  $\tP$-a.s.
\begin{align}
\label{eqn:convg_coll-1}
        \widetilde\vr_n &\longrightarrow \widetilde \vr \quad \mbox{in } C([0,T]; L^q(\dom)),\\
\label{eqn:convg_coll-2}
        \sqrt{\widetilde\vr_n} &\longrightarrow \sqrt{\widetilde \vr} \quad \mbox{in } \left(L^\infty(0,T; H^1(\dom)), w^\ast\right),\\
\label{eqn:convg_coll-3}
        \sqrt{\widetilde\vr_n} \widetilde \vu_n &\longrightarrow \sqrt{\widetilde \vr}\widetilde\vu \quad \mbox{in } L^2(0,T; L^2(\dom)),\\
\label{eqn:convg_coll-4}
        \widetilde\vr_n \widetilde \vu_n \otimes \widetilde \vu_n &\longrightarrow \widetilde\vr \widetilde \vu \otimes \widetilde\vu \quad \mbox{in } L^1(0,T;L^1(\dom)),\\
\label{eqn:convg_coll-5}
        \Div(\wt\vr_n \Grad \wt\vu_n) &\longrightarrow \Div( \wt\vr\, \Grad \wt\vu) \quad \mbox{in } \mathcal{D}^\prime ( (0,T) \times \dom) .
\end{align}

\begin{lemma}
\label{lem:est-tilde}
Let $(\widetilde \vu_n)$ be the sequence from  Lemma~\ref{lem-tilde u_n}. Then for every  $p \in [1, \infty)$ the random variables $\widetilde \vr_n$ and $\widetilde \vu_n$ satisfy the following uniform estimates
\begin{align}
\label{tilde_un_n-1}
&\|\sqrt{\widetilde \vr_n}\widetilde \vu_n\|_{L^{p}(\Omega; L^\infty(0,T; L^2(\dom)))} \le C,\\
\label{tilde_un_n-2}
&\|\widetilde\vr_n\|_{L^p(\Omega; L^\infty(0,T; L^\gamma(\dom)))} \le C, \\
\label{tilde_un_n-3}
&\| \sqrt{\widetilde\vr_n}\|_{L^p(\Omega; L^\infty(0,T; H^1(\dom)))} \le C.
\end{align}
\end{lemma}

\begin{proof}[Proof of Lemma \ref{lem:est-tilde}] The estimates for the random variables $\widetilde \vr_n$ and $\widetilde \vu_n$ are the consequence of the fact that the law of $\left(\widetilde\vr_n, \widetilde\vartheta_n, \widetilde\vm_n, \widetilde \vq_n, \widetilde \vrr_n\right)$ on $\mathcal{X}_T$ (see \eqref{eqn-spaces-X_T}) is given by $\mu_n$, $n \in \N$ (see Proposition~\ref{prop-jak+sk} (1)) and identification of the random variables established in Lemma~\ref{lem-identification} and Lemma~\ref{lem-tilde u_n}.
\end{proof}

\begin{lemma} \label{lem-convgergence}
Let $\widetilde \vu_n$ and $\widetilde \vu$ be as defined in Lemma~\ref{lem-tilde u_n} and Lemma~\ref{lem-tilde u} respectively. Let $\phi : \dom \to \R$ and $\bpsi \colon \dom \to \R^3$ be test functions such that $\phi \in W^{1,\infty}(\dom)$ and $\bpsi \in W^{2,\infty}(\dom)$. Then
\begin{align}
    \label{eqn:convg_rho}
   & \lim_{n \to \infty} \widetilde\E \left|\intTO{\left(\widetilde \vr_n - \widetilde \vr\right) \phi}\right| = 0,\\
    \label{eqn:convg_mom}
    &\lim_{n \to \infty} \widetilde\E \int_0^T\left|\inttO{\left( \wt\vr_n\wt\vu_n - \wt\vr\wt\vu\right)\cdot \Grad \phi}\right|\dt = 0,\\
    \label{eqn:convg_mom2}
    &\lim_{n \to \infty} \widetilde\E \left|\intTO{\left(\widetilde \vr_n\wt\vu_n - \widetilde \vr \wt\vu\right)\cdot\bpsi}\right| = 0,\\
    \label{eqn:convg_non}
    &\lim_{n\to\infty}\wt\E\int_0^T\left|\inttO{\left(\wt\vr_n\wt\vu_n\otimes\wt\vu_n -\wt\vr\wt\vu\otimes\wt\vu\right)\colon \Grad \bpsi}\right|\dt = 0,\\
    \label{eqn:convg_stress}
    &\lim_{n\to\infty}\wt\E\int_0^T\left|\inttO{\left(\wt\vr_n\Grad\wt\vu_n - \wt\vr\Grad\wt\vu\right) \colon\Grad \bpsi}\right|\dt = 0,\\
    \label{eqn:convg_press}
    &\lim_{n\to\infty}\wt\E\int_0^T\left|\inttO{\left(\wt\vr_n^\gamma - \wt\vr^\gamma\right)\Div \bpsi}\right|\dt = 0,\\
    \label{eqn:convg_density}
    &\lim_{n\to\infty}\wt\E \left|\intO{\left(\widetilde \vr_n(0,\cdot) - \widetilde \vr(0,\cdot)\right) \phi}\right| = 0,\\
    \label{eqn:convg_momentum}
     &\lim_{n\to\infty}\wt\E \left|\intO{\left(\widetilde {\vc{m}_n}(0,\cdot) - \widetilde{\vc{m}}(0,\cdot)\right) \cdot\bpsi}\right| = 0.
\end{align}
\end{lemma}

\begin{proof}[Proof of Lemma \ref{lem-convgergence}] Let us fix $\phi \in W^{1,\infty}(\dom; \R)$ and $\bpsi \in W^{2,\infty}(\dom;\R^3)$.\\
\textbf{Step 1. Proof of \eqref{eqn:convg_rho}.}
From \eqref{eqn:convg_coll-1}, we know that $\wt\vr_n \to \wt\vr$ in $C([0,T]; L^2(\dom))$ $\wt\Prob$-a.s., what implies that
\begin{equation}
    \label{eqn:convg_1}
    \lim_{n\to\infty} \intTO{ \wt\vr_n \phi}=  \intTO{ \wt\vr \phi} \qquad \wt\Prob\mbox{-a.s.}.
\end{equation}
We have for some $r > 1$, by \eqref{tilde_un_n-2}
\[
\wt\E\left(\left|\intTO{ \wt\vr_n \phi}\right|^r\right) \le C \|\varphi\|_{L^2(\dom)}^r \wt\E\int_0^T\|\wt\vr_n(t)\|^{r}_{L^2(\dom)}\dt \le C.
\]
This bound provides the equi-integrability of
$\int_0^T\int_\dom \wt\vr_n(t,x) \phi(x)\,\dx\,\dt$.
Taking into account the convergence \eqref{eqn:convg_1}, the Vitali convergence theorem \ref{thm-Vitali} then shows that \eqref{eqn:convg_rho} holds.\\
\textbf{Step 2. Proof of \eqref{eqn:convg_mom}.}
From \eqref{eqn:convg_coll-1} and \eqref{eqn:convg_coll-3} , $\sqrt{\wt\vr_n} \to \sqrt{\wt\vr}$ strongly in $C([0,T]; L^2(\dom))$ $\wt\Prob$-a.s. and $\sqrt{\wt\vr_n}\wt\vu_n \to \sqrt{\wt\vr}\wt\vu$ strongly in $L^2(0,T; L^2(\dom))$ $\wt\Prob$-a.s. respectively. Therefore, for a.a. $t \in (0,T]$ $\wt\Prob$-a.s.

\begin{equation}
    \label{eqn:convg_2}
    \lim_{n \to \infty} \inttO{\wt\vr_n\wt\vu_n\cdot \Grad \phi} = \inttO{ \wt\vr\wt\vu\cdot \Grad \phi}.
\end{equation}
Furthermore, by estimates \eqref{tilde_un_n-1} and \eqref{tilde_un_n-2}, we find that
\begin{align*}
    &\wt\E\left(\left|\int_0^t\int_\dom \wt\vr_n\wt\vu_n\cdot\Grad \phi\,\dx \ds\right|^2\right) \\
    &\quad \le \|\Grad \phi\|^2_{L^\infty(\dom)}\wt\E\left(\left|\int_0^t\|\sqrt{\wt\vr_n}\sqrt{\wt\vr_n}\wt\vu_n\|_{L^1(\dom)}\ds\right|^2\right)\\
    &\quad\le \|\Grad \phi\|^2_{L^\infty(\dom)}\wt\E\left(\left|\int_0^t\|\sqrt{\wt\vr_n}\|_{L^2(\dom)}\|\sqrt{\wt\vr_n}\wt\vu_n\|_{L^2(\dom)}\ds\right|^2\right) \\
    &\quad \le \|\Grad \phi\|^2_{L^\infty(\dom)}T^2 \left(\wt\E\|\sqrt{\wt\vr_n}\|^4_{L^\infty(0,T;L^2(\dom))}\right)^{1/2} \left(\wt\E\|\sqrt{\wt\vr_n}\wt\vu_n\|^4_{L^\infty(0,T;L^2(\dom))}\right)^{1/2} \le C.
\end{align*}
This bound and the $\wt\Prob$-a.s. convergence \eqref{eqn:convg_2} allows us to apply again the Vitali convergence theorem \ref{thm-Vitali} to infer that \eqref{eqn:convg_mom} holds.\\
\textbf{Step 3.}  We can similarly establish the convergence \eqref{eqn:convg_mom2}.\\
\textbf{Step 4. Proof of \eqref{eqn:convg_non}.}  From the convergence \eqref{eqn:convg_coll-4}, we have for $t \in [0,T]$ $\wt\Prob$-a.s.
\begin{align}
    \label{eqn:convg_3}
    & \lim_{n \to \infty} \int_0^t \int_\dom\wt\vr_n\wt\vu_n\otimes\wt\vu_n \colon \Grad \bpsi\,\dx\ds \nonumber \\
    & \quad = \int_0^t \int_\dom\wt\vr\wt\vu\otimes\wt\vu \colon \Grad \bpsi\,\dx\ds.
\end{align}
Using the estimate \eqref{tilde_un_n-1}, we get
\begin{align*}
    &\wt\E\left(\left|\inttO{\wt\vr_n\wt\vu_n\otimes\wt\vu_n \colon \Grad\bpsi}\right|^2\right) \\
    & \quad \le \|\Grad \bpsi\|^2_{L^\infty(\dom)}\wt\E\|\sqrt{\wt\vr_n}\wt\vu_n\|^4_{L^2(0,T;L^2(\dom))} \le C.
\end{align*}
Using the bound obtained above, the $\wt\Prob$-a.s. convergence \eqref{eqn:convg_3} and the Vitali convergence theorem \ref{thm-Vitali} we conclude that \eqref{eqn:convg_non} holds.\\
\textbf{Step 5. Proof of \eqref{eqn:convg_stress}.} let us choose and fix  $t \in [0,T]$. The convergence $\Div\left(\wt\vr_n\Grad\wt\vu_n\right) \to \Div\left(\wt\vr\Grad\wt\vu\right)$ in $\mathcal{D}^\prime$ (as shown in Lemma~\ref{lem-convergence diffusion}), shows that $\wt\Prob$-a.s.
\begin{align}
    \label{eqn:convg_4}
    &\lim_{n\to\infty}\inttO{ \wt\vr_n\Grad\wt\vu_n\colon\Grad \bpsi} \\
    &\quad = \inttO{\wt\vr\wt\vu\cdot\Delta \bpsi} + \inttO{\Grad\wt\vr\otimes\wt\vu\colon\Grad\bpsi}.
\end{align}
Moreover, using the estimates \eqref{tilde_un_n-1} and \eqref{tilde_un_n-3}, we have
\begin{align*}
    &\wt\E\left|\inttO{\wt\vr_n\Grad\wt\vu_n\colon\Grad\bpsi}\right|^2 \\
    &\quad = \wt\E\left|\inttO{\wt\vr_n\wt\vu_n\cdot\Delta\bpsi} + \inttO{\Grad\wt\vr_n\otimes\wt\vu_n\colon\Grad\bpsi}\right|^2 \\
    &\quad \le C\left(\wt\E\left|\inttO{\sqrt{\wt\vr_n}\sqrt{\wt\vr_n}\wt\vu_n\colon\Delta\bpsi}\right|^2 + 4\wt\E\left|\inttO{\Grad\sqrt{\wt\vr_n}\otimes\sqrt{\wt\vr_n}\wt\vu_n\colon\Grad\bpsi}\right|^2\right) \\
    &\quad \le C\|\Delta\bpsi\|^2_{L^\infty(\dom)}\wt\E\left|\int_0^t\|\sqrt{\wt\vr_n}\|_{L^2(\dom)}\|\sqrt{\wt\vr_n}\wt\vu_n\|_{L^2(\dom)}\ds\right|^2 \\
    &\qquad \quad + C \|\Grad\bpsi\|^2_{L^\infty(\dom)}\wt\E\left|\int_0^t\|\Grad\sqrt{\wt\vr_n}\|_{L^2(\dom)}\|\sqrt{\wt\vr_n}\wt\vu_n\|_{L^2(\dom)}\ds\right|^2\\
    &\quad \le C\|\Delta\bpsi\|^2_{L^\infty(\dom)}T^2\left(\wt\E\|\sqrt{\wt\vr_n}\|^4_{L^\infty(0,T;L^2(\dom))}\right)^{1/2}\left(\wt\E\|\sqrt{\wt\vr_n}\wt\vu_n\|^4_{L^\infty(0,T;L^2(\dom))}\right)^{1/2}\\
    &\qquad \quad + C\|\Grad\bpsi\|^2_{L^\infty(\dom)}T^2\left(\wt\E\|\Grad\sqrt{\wt\vr_n}\|^4_{L^\infty(0,T;L^2(\dom))}\right)^{1/2}\left(\wt\E\|\sqrt{\wt\vr_n}\wt\vu_n\|^4_{L^\infty(0,T;L^2(\dom))}\right)^{1/2} \le C.
\end{align*}
The above bound, the $\wt\Prob$-a.s. convergence \eqref{eqn:convg_4} and the Vitali convergence theorem \ref{thm-Vitali} are enough to deduce \eqref{eqn:convg_stress}.\\
\textbf{Step 6. Proof of \eqref{eqn:convg_press}.} Let us begin with recalling  that $\gamma < 3$. We can prove the convergence result similarly as above by using the strong convergence \eqref{eqn:convg_coll-1} of $\wt\vr_n$, estimate \eqref{tilde_un_n-2} and the Vitali convergence theorem \ref{thm-Vitali}.\\
\textbf{Step 7. Proof of \eqref{eqn:convg_density} and \eqref{eqn:convg_momentum}}. These assertions follow from \eqref{eqn:vr_n-convergence} and \eqref{eqn:vm_n-convergence} respectively. \\
The proof of Lemma is complete.
\end{proof}

\begin{lemma}
\label{lem-convergence_stoch}
Let \text{Assumption}~\ref{ass-force} hold. Moreover, let $\wt W_n$ and $\wt W$ be $\R$-valued Brownian motions as in Lemma~\ref{lem-Brownian Motion} and Lemmma~\ref{lem-Brownian Motion-2} respectively, where $\wt W_n \to \wt W$ in $C([0,T];\R)$ as $n \to \infty$ $\wt\Prob$-almost surely. Then, for every {$\bpsi \in L^\infty(\dom)$}, the following holds
\begin{equation}
    \label{eqn:convg_stoch}
    \lim_{n\to\infty}\wt\E\int_0^T\left|\int_0^t\int_\dom\left( \wt\vr_n\vf\,\mathrm{d}\wt W_n - \wt\vr\vf\,\mathrm{d}\wt W\right)\cdot\bpsi\,\dx\right|^2\dt = 0.
\end{equation}
\end{lemma}

\begin{proof}[Proof of Lemma \ref{lem-convergence_stoch}] Since $\wt W_n \to \wt W$ in $C([0,T];\R)$, {due to results in \cite[Section 4.3.5]{Ben95}, \cite[Lemma 5.1]{GyKr96}, \cite[Lemma~2.1]{DeGlTe11} and \cite[Lemma 2.6.5 and Lemma 2.6.6]{Breit+Feir+Ho18}},  it is sufficient to show that
\[
\lim_{n\to \infty}\wt\E\int_0^T\left|\int_0^t\int_\dom  \left(\wt\vr_n\vf - \wt\vr\vf\right){\rm{d}}\wt W\cdot\bpsi\,\dx\right|^2\dt = 0.
\]
We estimate for $\bpsi \in L^\infty(\dom)$
\begin{align*}
    \int_0^t \left|\int_\dom \left(\wt\vr_n\vf - \wt\vr\vf\right)\cdot\bpsi\,\dx\right|^2\ds
    & \quad \le \int_0^t \|\wt\vr_n - \wt\vr\|^2_{L^q(\dom)}\|\vf\|^2_{L^{q^\prime}(\dom)}\|\bpsi\|^2_{L^{\infty}(\dom)}\\
    & \quad \le \|\wt\vr_n-\wt\vr\|^2_{L^2(0,T;L^q(\dom))}\|\vf\|^2_{L^{q^\prime}(\dom)}\|\bpsi\|^2_{L^{\infty}(\dom)},
\end{align*}
where $q^\prime=\frac{q}{q-1}$.  Recall that $\wt\vr_n \to \wt\vr$ in $C([0,T]; L^q(\dom))$ $\wt\Prob$-a.s. for $q < 3$, by \eqref{eqn:convg_coll-1}. In particular, $q^\prime<3$ for $q$ large enough. Therefore under the Assumption \ref{ass-force} on $\vf$, we infer that for {a.a} $t \in [0,T]$, $\omega \in \wt\Omega$
\begin{equation}
    \label{eqn:convg_5}
    \lim_{n\to\infty} \int_0^t \left|\int_\dom \left(\wt\vr_n\vf- \wt\vr\vf\right)\cdot\bpsi\,\dx\right|^2\ds = 0.
\end{equation}
We conclude from \eqref{tilde_un_n-3} and assumptions on $\vf$ that
\begin{align*}
    &\wt\E\left|\int_0^t \left|\int_\dom \left(\wt\vr_n\vf - \wt\vr\vf\right)\cdot\bpsi(x)\,\dx\right|^2\ds\right|^2 \\
    & \quad \le C \wt\E\left(\|\bpsi\|^4_{L^{\infty}(\dom)}\|\vf\|^4_{L^{3/2}(\dom)}\int_0^t\left(\|\wt\vr_n(s)\|^4_{L^3(\dom)}+\|\wt\vr(s)\|^4_{L^{3}(\dom)}\right)\ds\right) \\
    & \quad \le C\|\bpsi\|^4_{L^{\infty}(\dom)}T \left(\wt\E\|\vf\|^8_{L^{3/2}(\dom)}\right)^{1/2} \left(\wt\E\left(\|\wt\vr_n\|^8_{L^\infty(0,T; L^3(\dom))} + {\|\wt\vr\|^8_{L^\infty(0,T; L^3(\dom))}}\right)\right)^{1/2} \le C.
\end{align*}
With this bound, convergence \eqref{eqn:convg_5} and the Vitali convergence theorem \ref{thm-Vitali} we obtain for all $\bpsi \in L^{\infty}(\dom)$
\[
\lim_{n\to\infty}\wt\E\int_0^t \left|\int_\dom \left(\wt\vr_n\vf - \wt\vr\vf\right)\cdot\bpsi\,\dx\right|^2\ds = 0.
\]
Hence, by the It\^o isometry for $t \in [0,T]$ and $\bpsi \in L^{\infty}(\dom)$,
\begin{equation}
    \label{eqn:convg_6}
    \lim_{n\to\infty}\wt\E\left|\int_0^t \!\!\int_\dom  \left(\wt\vr_n\vf - \wt\vr\vf\right){\rm{d}}\wt W\cdot\bpsi\,\dx\right|^2 = 0.
\end{equation}
We use the It\^o isometry again and estimate \eqref{tilde_un_n-3} for $n \in \N$, {strong convergence \eqref{eqn:convg_coll-1} along with the lower semicontinuity of norms} to infer
\begin{align*}
    &\wt\E\left|\int_0^t \!\!\int_\dom \left(\wt\vr_n\vf - \wt\vr\vf\right){\rm{d}}\wt W\cdot\bpsi(x)\,\dx\right|^2 \\
    &\quad = \wt\E\left(\int_0^t \left|\int_\dom \left(\wt\vr_n\vf - \wt\vr\vf\right)\cdot\bpsi(x)\,\dx\right|^2\ds\right)\\
    & \quad \le C \wt\E\left(\|\bpsi\|^2_{L^{\infty}(\dom)}\|\vf\|^2_{L^{3/2}(\dom)}\int_0^t\left(\|\wt\vr_n(s)\|^2_{L^3(\dom)}+\|\wt\vr(s)\|^2_{L^{3}(\dom)}\right)ds\right) \\
    & \quad \le C\|\bpsi\|^2_{L^{\infty}(\dom)}T\left(\wt\E\|\vf\|^4_{L^{3/2}(\dom)}\right)^{1/2} \left(\wt\E\left(\|\wt\vr_n\|^4_{L^\infty(0,T; L^3(\dom))} + {\|\wt\vr\|^4_{L^\infty(0,T; L^3(\dom))}}\right)\right)^{1/2} \le C.
\end{align*}
This bound and the convergence \eqref{eqn:convg_6} allows us to apply the dominated convergence theorem to conclude that for all $\bpsi \in L^{\infty}(\dom)$,
\[
\lim_{n\to \infty}\wt\E\int_0^T\left|\int_0^t\!\! \int_\dom  \left(\wt\vr_n\vf - \wt\vr\vf\right){\rm{d}}\wt W\cdot\bpsi(x)\,\dx\right|^2\dt = 0.
\]
This finishes the proof of the lemma.
\end{proof}

Let {$\phi \in W^{1,\infty}(\dom; \R)$ and $\bpsi \in W^{2,\infty}(\dom; \R^3)$.} Moreover, let $\wt\vu_n$, $\wt\vu$ be the processes as defined in Lemma~\ref{lem-tilde u_n} and Lemma~\ref{lem-tilde u} respectively and define
\begin{align*}
    \Lambda^1_n(\wt\vr_n,\wt\vu_n,\phi)(t) &:= \intO{\wt\vr_n(0,x)\cdot\phi(x)} + \int_0^t\int_\dom \wt\vr_n(s)\wt\vu_n(s)\cdot\Grad \phi\,\dx\ds, \\
    \Lambda^1(\wt\vr,\wt\vu,\phi)(t) &:= \intO{\wt\vr(0,x)\cdot\phi(x)} +  \int_0^t\int_\dom \wt\vr(s)\wt\vu(s)\cdot \Grad \phi\,\dx\ds, \\
    \Lambda^2_n(\wt\vr_n,\wt\vu_n,\wt W_n, \bpsi)(t) &:= \int_\dom {\wt\vm_n(0)}\cdot \bpsi \,\dx + \int_0^t \int_\dom \widetilde{\vr}_n(s)\widetilde{\vu}_n(s) \otimes \widetilde{\vu}_n(s) \colon \nabla \bpsi \,\dx \ds \\
    &\;- \int_0^t \int_\dom \widetilde{\vr}_n(s)\Grad\widetilde{\vu}_n(s) \colon  \nabla \bpsi \,\dx \ds
 + \int_0^t \int_\dom \widetilde{\vr}_n^\gamma(s) \Div \bpsi \,\dx \ds \nonumber \\
 &\; + \int_0^t \int_\dom\widetilde{\vr}_n(s) \vf \cdot \bpsi \,\dx \rm{d}\wt W_n(s), \nonumber\\
    \Lambda^2(\wt\vr,\wt\vu,\wt W, \bpsi)(t) &:= \int_\dom {\wt\vm(0)} \cdot \bpsi \,\dx + \int_0^t \int_\dom \widetilde{\vr}(s)\widetilde{\vu}(s) \otimes \widetilde{\vu}(s)\colon \nabla \bpsi \,\dx \ds \nonumber \\
&\; - \int_0^t \int_\dom \widetilde{\vr}(s)\Grad\widetilde{\vu}(s)\colon \nabla \bpsi \,\dx \ds
    + \int_0^t \int_\dom \widetilde{\vr}^\gamma(s) \Div \bpsi \,\dx \ds \nonumber \\
    &\;+ \int_0^t \int_\dom\widetilde{\vr}(s) \vf \cdot \bpsi \,\dx \rm{d}\wt W(s),
\end{align*}
for $t \in [0,t]$, where the diffusion terms should be understood in the following sense
\begin{align*}
    -\int_0^t\int_\dom \wt\vr_n(s)\Grad\wt\vu_n(s)\colon \Grad \bpsi\,\dx\ds &=  \int_0^t\int_\dom\wt\vr_n(s) \wt\vu_n(s)\cdot \lap \bpsi\,\dx\ds\\
    &\quad  + \int_0^t\int_\dom \Grad \wt\vr_n(s) \otimes\wt\vu_n(s):\Grad\bpsi\,\dx\ds,\\
     -\int_0^t\int_\dom \wt\vr(s)\Grad\wt\vu(s)\colon \Grad \bpsi\,\dx\ds & =  \int_0^t\int_\dom\wt\vr(s) \wt\vu(s)\cdot \lap \bpsi\,\dx\ds\\
    &\quad  + \int_0^t\int_\dom \Grad \wt\vr(s) \otimes\wt\vu(s):\Grad\bpsi\,\dx\ds.
\end{align*}

Then, the following corollary is essentially a consequence of Lemmata~\ref{lem-convgergence} and~\ref{lem-convergence_stoch}.

\begin{corollary} \label{cor_lambda}
Let us assume that  $\phi \in W^{1,\infty}(\dom; \R)$ and $\bpsi \in W^{2,\infty}(\dom; \R^3)$. Then the following convergences hold$\colon$
\begin{align}
\label{eqn:cor_1}
    &\lim_{n\to\infty}\left\|\intO{\wt\vr_n(x)\cdot\phi(x)} - \intO{\wt\vr(x)\cdot\phi(x)}\right\|_{L^1(\wt\Omega\times[0,T])} = 0,\\
    \label{eqn:cor_2}
    &\lim_{n\to\infty}\|\Lambda^1_n(\wt\vr_n,\wt\vu_n,\phi) - \Lambda^1(\wt\vr,\wt\vu,\phi)\|_{L^1(\wt\Omega\times[0,T])} = 0,\\
    \label{eqn:cor_3}
     &\lim_{n\to\infty}\left\|\int_\dom\wt\vr_n(x)\wt\vu_n(x) \cdot \bpsi(x)\,\dx - \int_\dom\wt\vr(x)\wt\vu(x)\cdot \bpsi(x)\,\dx\right\|_{L^1(\wt\Omega\times[0,T])} = 0,\\
    \label{eqn:cor_4}
    &\lim_{n\to\infty}\|\Lambda^2_n(\wt\vr_n,\wt\vu_n,\wt W_n, \bpsi) - \Lambda^2(\wt\vr,\wt\vu, \wt W, \bpsi)\|_{L^1(\wt\Omega\times[0,T])} = 0.
\end{align}
\end{corollary}

\begin{proof}[Proof of Corollary \ref{cor_lambda}] The first and the third convergences  follow directly from the identities
\begin{align*}
&\left\|\intO{\wt\vr_n(x)\cdot\phi(x)} - \intO{\wt\vr(x)\cdot\phi(x)}\right\|_{L^1(\wt\Omega\times[0,T])}\\
& \qquad = \wt\E\left|\int_0^T\intO{(\wt\vr_n(t,x) - \wt\vr(t,x))\cdot\phi(x)}\,\dt\right| \\
&\left\|\int_\dom\wt\vr_n(x)\wt\vu_n(x) \cdot \bpsi(x)\,\dx - \int_\dom\wt\vr(x)\wt\vu(x)\cdot \bpsi(x)\,\dx\right\|_{L^1(\wt\Omega\times[0,T])} \\
&\qquad  = \wt\E\left|\int_0^T\int_\dom\wt\vr_n(t,x)\wt\vu_n(t,x)-\wt\vr(t,x)\wt\vu(t,x)\cdot\bpsi(x)\,\dx\,\dt\right|
\end{align*}
and convergences \eqref{eqn:convg_rho} and \eqref{eqn:convg_mom2} respectively. For  \eqref{eqn:cor_2},  by the Fubini's theorem we have
\begin{align*}
    \|\Lambda^1_n(\wt\vr_n,\wt\vu_n,\phi) - \Lambda^1(\wt\vr,\wt\vu,\phi)\|_{L^1(\wt\Omega\times[0,T])} = \int_0^T \wt\E\left|\Lambda^1_n(\wt\vr_n,\wt\vu_n,\phi) - \Lambda^1(\wt\vr,\wt\vu,\phi)\right|\dt.
\end{align*}
Convergence \eqref{eqn:convg_mom} shows that the term in the definition of $\Lambda^1_n(\wt\vr_n,\wt\vu_n,\phi)$ tends to the corresponding term in $\Lambda^1(\wt\vr,\wt\vu,\phi)$ at least in the space $L^1(\wt\Omega \times [0,T])$. Similarly, with the help of the Fubini's theorem, the definition of the maps $\Lambda^2_n$, $\Lambda^2$ and convergences \eqref{eqn:convg_non}--\eqref{eqn:convg_press} we can deduce the convergence \eqref{eqn:cor_4}.\end{proof}

Since $\left(\vr_n,\vu_n\right)$ is a solution to \eqref{eqn-1.1}, it satisfies
\begin{align*}
    \intO{\vr_n(t,x)\cdot\phi(x)}& = \Lambda^1_n(\vr_n,\vu_n,\phi)(t) \quad \Prob\mbox{-a.s.}\\
    \int_\dom\vr_n(t,x)\vu_n(t,x)\cdot\bpsi(x)\,\dx & = \Lambda^2_n\left(\vr_n,\vu_n,W_n,\bpsi\right)(t)\quad\Prob\mbox{-a.s.}
\end{align*}
for all $t \in [0,T]$ and $\phi \in W^{1,\infty}(\dom; \R)$, $\bpsi \in W^{2,\infty}(\dom; \R^3)$ and in particular, we have
\begin{align*}
    &\int_0^T \E\left|\left(\intO{\vr_n(t,x)\cdot\phi(x)}\right) - \Lambda^1_n(\vr_n,\vu_n,\phi)(t)\right|\dt = 0, \\
    &\int_0^T\E\left|\int_\dom\vr_n(t,x)\vu_n(t,x)\cdot\bpsi(x)\,\dx  - \Lambda^2_n\left(\vr_n,\vu_n,W_n,\bpsi\right)(t)\right|\dt = 0.
\end{align*}
Since the laws $\mathcal{L}\left(\widetilde{\vr}_n, \widetilde{\vt}_n, \widetilde{\vm}_n, \widetilde{\vq}_n,\widetilde{\vrr}_n, \wt W_n\right)$ and $\mathcal{L}\left({\vr}_n, {\sqrt{\vr}}_n, {\vr}_n\vu_n, {\sqrt{\vr}_n\vu_n},\vr_n^{\frac{1}{2+\delta}}\vu_n, W_n\right)$ coincide and by the identifications from Lemma~\ref{lem-identification} and Lemma~\ref{lem-tilde u_n}, we find that
\begin{align*}
    &\int_0^T \wt\E\left|\left(\intO{\wt\vr_n(t,x)\cdot\phi(x)}\right) - \Lambda^1_n(\wt\vr_n,\wt\vu_n,\phi)(t)\right|\dt = 0, \\
    &\int_0^T \wt\E\left|\int_\dom \wt\vr_n(t,x) \wt\vu_n(t,x)\cdot\bpsi(x)\,\dx  - \Lambda^2_n\left(\wt\vr_n,\wt\vu_n,\wt W_n,\bpsi\right)(t)\right|\dt = 0.
\end{align*}
By Corollary~\ref{cor_lambda}, the limit $n \to \infty$ in these equations yield
\begin{align*}
    &\int_0^T \wt\E\left|\left(\intO{\wt\vr(t,x)\cdot\phi(x)}\right) - \Lambda^1(\wt\vr,\wt\vu,\phi)(t)\right|\dt = 0, \\
    &\int_0^T \wt\E\left|\int_\dom \wt\vr(t,x) \wt\vu(t,x)\cdot\bpsi(x)\,\dx  - \Lambda^2\left(\wt\vr,\wt\vu,\wt W,\bpsi\right)(t)\right|\dt = 0.
\end{align*}
Hence, for Lebesgue-a.e. $t \in [0,T]$ and $\wt\Prob$-a.e. $\omega \in \wt\Omega$, we deduce that
\begin{align*}
    &\intO{\wt\vr(t,x)\cdot\phi(x)} - \Lambda^1(\wt\vr,\wt\vu,\phi)(t) = 0, \\
    &\int_\dom \wt\vr(t,x) \wt\vu(t,x)\cdot\bpsi(x)\,\dx - \Lambda^2\left(\wt\vr,\wt\vu,\wt W,\bpsi\right)(t) = 0.
\end{align*}
By definition of $\Lambda^1$ and $\Lambda^2$, this means that for Lebesgue-a.e. $t \in [0,T]$ and $\wt\Prob$-a.e. $\omega \in \wt\Omega$,
\begin{align*}
    \intO{\wt\vr(t,x)\cdot\phi(x)} &= \intO{\wt\vr(0,x)\cdot\phi(x)} +  \int_0^t\int_\dom \wt\vr(s)\wt\vu(s)\cdot \Grad \phi\,\dx\ds, \\
    \int_\dom \wt\vr(t,x) \wt\vu(t,x)\cdot\bpsi\,\dx &= \int_\dom \wt\vm(0)\cdot \bpsi \,\dx + \int_0^t \int_\dom \widetilde{\vr}(s)\widetilde{\vu}(s) \otimes \widetilde{\vu}(s)\colon \nabla \bpsi \,\dx \ds \nonumber\\
    &\quad - \int_0^t \int_\dom \widetilde{\vr}(s)\Grad\widetilde{\vu}(s)\colon \nabla \bpsi \,\dx \ds
     + \int_0^t \int_\dom \widetilde{\vr}^\gamma(s) \Div \bpsi \,\dx \ds \nonumber\\
      &\quad + \int_0^t \int_\dom\widetilde{\vr}(s) \vf \cdot\bpsi \,\dx \rm{d}\wt W(s).
\end{align*}
Setting $\wt U := \left(\wt\Omega, \wt{\mathcal{F}},\wt\Prob,\wt{\mathbb{F}}\right)$, we infer the system $\left(\wt U, \wt W, \wt\vr,\wt\vu\right)$ is a martingale solution to \eqref{eqn-1.1}.


\bigskip

{\bf Acknowledgments.} EZ received the support of the EPSRC Early Career Fellowship no. EP/V000586/1. ZB's work was partially supported by the grant 346300 for IMPAN from the Simons Foundation and the matching 2015-2019 Polish MNiSW fund and  by the Australian Research Council Project DP160101755. GD acknowledges partial support from the Austrian Science Fund (FWF), grants I3401, P30000,
 P33010, W1245, and F65 and from the Australian Research Council Project DP160101755.

\newpage
\appendix

\section{Some compactness and tightness results}
\label{sec-compactness}

We now move to the tightness of law for the momentum.
For the compact embeddings in this case we will be using Lemma~\ref{lem-compactness Simon-2} \cite[Theorem 6]{Sim87}, Lemma~\ref{lem-BFH weak cont compact} and the Aldous condition $\textbf{[A]}$, which we recall in Appendix~\ref{sec-stoch_pre}.

The following lemma (see \cite[Lemma~A.7 and Lemma~A.8]{Mot12}) gives us a useful conclusion of the Aldous condition $\textbf{[A]}$.
\begin{lemma}
\label{lem-Aldous condition}
Let $\left(X_n\right)_{n \in \N}$ be a sequence of continuous stochastic processes indexed by $t \in [0,T]$ and  taking values  in a Banach space $E$. Assume that this sequence   satisfies the Aldous condition $\textbf{[A]}$. Then, for every $\epsilon > 0$ there exists a measurable subset $A_\epsilon \subset C([0,T]; E)$ such that
\[
\mathbb{P}^{X_n}(A_\epsilon) \ge 1 - \epsilon, \qquad\quad \lim_{h \to 0}\sup_{u \in A_\epsilon}\sup_{|t-s| \le h}\|u(t) - u(s)\|_E = 0.
\]
{Here, by $\mathbb{P}^{X_n}$ we denote the law of $X_n$, which is a Borel probability measure on the Banach space $C([0,T]; E)$.}
\end{lemma}

We will use the following  modification of the above lemma.

\begin{lemma}
\label{lem-Aldous condition-2}
Let $\left(X_n\right)_{n \in \N}$ be a sequence of  stochastic processes with integrable trajectories indexed by $t \in [0,T]$ and  taking values  in a Banach space $E$.
Assume that this sequence  satisfies the Aldous condition $\textbf{[A]}$. Then, for every $\eps > 0$ there exists a {Borel} measurable subset $A_\eps \subset L^1([0,T]; E)$ such that
\[
\mathbb{P}^{X_n}(A_\eps) \ge 1 - \eps, \qquad\quad \lim_{h \to 0}\sup_{u \in A_\eps}\|\tau_h u-u\|_{L^1(0,T;E)} = 0.
\]
\end{lemma}

\begin{remark} \label{rem-Motyl-Aldous}
Lemma~\ref{lem-Aldous condition-2} is a modification of \cite[Lemma~A.7]{Mot12} where we have replaced the space of c\`adl\`ag functions taking values in a separable and complete metric space, $\mathbb{D}(0,T; \mathbb{S})$, with the space of integrable functions taking values in a Banach space, $L^1(0,T; E)$.

By closely observing the proof of \cite[Lemma~A.7]{Mot12} and \cite[Lemma~A.8]{Mot12} (see also \cite[Theorem~2.2.2]{JoMe86}) it is clear that the space $\mathbb{D}(0,T; \mathbb{S})$ can be replaced by $L^1(0,T; E)$.
\end{remark}

We have used the following compactness result to establish the tightness of $\{\Law(\vr_n\vu_n):n\in\N\}$ on $C_{w}([0,T]; L^{3/2}(\dom))$. The exact statement and the proof can be found in \cite[Theorem~1.8.5]{Breit+Feir+Ho18}.

\begin{lemma}
\label{lem-BFH weak cont compact}
Let $\kappa \ge 0$, $1 < p < \infty$, and $\ell \in \R$. Then
\begin{equation}\label{eqn:BFH weak cont}
L^\infty(0,T; L^p(\dom)) \cap C^\kappa([0,T]; W^{\ell,2}(\dom)) \embed  C_w([0,T]; L^p(\dom))
\end{equation}
continuously. Moreover, if $\kappa > 0$, then the  embedding \eqref{eqn:BFH weak cont} is sequentially compact.
\end{lemma}

\section{Measurability of some functions}
\label{sec-Appendix}

\begin{lemma}
\label{lem-Gamma}
Let us assume that $\delta > 0$ and
{ \begin{equation}\label{eqn:q-alpha}
  q \ge 3\frac{1+\delta}{1+2\delta}.
                \end{equation}
}
Then the  map
\begin{align}\label{eqn:Gamma-7.1}
\Gamma \colon C([0,T]; L^q(\dom)) &\times \bigl(L^\infty(0,T; L^{2+\delta}(\dom)),w^\ast \bigr) \\
&\ni (\rho,\vw) \mapsto \rho^{\frac{1+\delta}{2+\delta}} \vw \in \big(L^\infty(0,T; L^{3/2}(\dom)), w^\ast\big)
\nonumber
\end{align}
is (well-defined and) sequentially continuous.
\end{lemma}
\begin{remark}\label{rem-q} Since $\frac{1+\delta}{1+2\delta}<1$ for $\delta>0$, we can always find $q$ such that the above and $q<3$, i.e.
\[ 3 \frac{1+\delta}{1+2\delta} \leq q < 3.\]
So the assumption \eqref{eqn:q-alpha} is not void.
\end{remark}

\begin{lemma}
\label{lem-lemma-Holder} Assume that $\alpha \in (0,1)$. Consider  maps $g_2:\mathbb{R} \ni x \mapsto \operatorname{sgn}{(x)}\vert x\vert^\alpha \in \mathbb{R}$ and
$g_1:\mathbb{R} \ni x \mapsto \vert x\vert^\alpha \in [0,\infty)$. Then $g_1$ and $g_2$  are  globally $\alpha$-H\"older continuous  with constants $1$ and resp.  $2$, i.e.
\begin{equation}\label{eqn:Holder-g}
  \vert g_i(x_2)-g_i(x_1) \vert \leq i \vert x_2-x_1\vert^\alpha, \;\; x_1,x_2 \in \mathbb{R}, \;\;i=1,2.
\end{equation}
Then, for $i=1,2$, the  Nemytski map $G_i$ associated with $g_i$,  is   globally $\alpha$-H\"older from the Lebesque space $L^p(\dom)$ to $L^{\frac{p}\alpha}(\dom)$ for any $p \in [1,\infty]$, with the same
constant as $g_i$.
\end{lemma}

\begin{proof}[Proof of Lemma   \ref{lem-lemma-Holder}]
It is a standard observation that $g_i$, $i=1,2$ restricted to the interval $[0,\infty)$ is globally $\alpha$-H\"older with constant $1$. Then the first part of the result for $i=2$ follows by applying another standard
argument, see e.g. the proof of Proposition 3.2 \cite{BrDa09}. Then the first part of the result for $i=2$ follows by observing that function $g_1$ is even.
The proof of the second part is also standard.
\end{proof}

\begin{proof}[Proof of Lemma \ref{lem-Gamma}] Let us fix $\delta > 0$.  Put $\alpha =\frac{1+\delta}{2+\delta}$.
Since by Lemma \ref{lem-lemma-Holder} the map
\[
C([0,T]; L^q(\dom)) \ni \rho \mapsto \rho^{\alpha} \in C([0,T]; L^{q/\alpha}(\dom))
\]
is globally $\alpha$-H\"older continuous, it is sufficient to consider the following map (denoted, in order to avoid introducing additional symbols) by $\Gamma$ as well,

\begin{align}\label{eqn:Gamma new-7.3}
\Gamma \colon C([0,T]; L^{q/\alpha}(\dom)) &\times \bigl(L^\infty(0,T; L^{2+\delta}(\dom)),w^\ast \bigr) \ni (\mathrm{v},\vw)\\
&\mapsto \mathrm{v} \vw \in \big(L^\infty(0,T; L^{3/2}(\dom)), w^\ast\big)
\nonumber
\end{align}
First of all, let us observe that since by \eqref{eqn:q-alpha}
  \begin{equation}\label{ineq-Holder-assumptions}
  \frac{1}{{q/\alpha}}+\frac{1}{2+\delta} \leq \frac{1}{3/2}.
  \end{equation}
in view of Proposition \ref{prop-Holder} the auxiliary map $\Gamma$ defined in \eqref{eqn:Gamma-7.2} is well defined,  bilinear  and bounded with respect to the normed  spaces involved. However, we need to prove the continuity in the weak$^\ast$ topologies as shown.

 Since $L^\infty(0,T;L^{3/2}(\dom))$ is the dual of $L^1(0,T;L^3(\dom))$, it is sufficient to prove that for every $\bpsi \in L^1(0,T; L^3(\dom))$  the map
\begin{equation}\label{eqn:Gamma-7.2}
 C([0,T]; L^{q/\alpha}(\dom)) \times \bigl(L^\infty(0,T; L^{2+\delta}(\dom)),w^\ast \bigr) \ni (\mathrm{v},\vw) \mapsto
\langle \Gamma(\rm{v}, \vw), \bpsi \rangle \in \R,
\end{equation}
is sequentially continuous. Here $\langle \cdot, \cdot \rangle$ denotes the duality between $L^\infty(0,T; L^{3/2}(\dom))$ and $L^1(0,T; L^3(\dom))$.\\
 For this aim let us
choose and fix  an element   $\bpsi \in L^1(0,T; L^3(\dom))$,  a sequence $\{\mathrm{v}_n\}$ convergent to $\mathrm{v}$ in $C([0,T]; L^q(\dom))$ and a sequence $\{\vw_n\}$  convergent to $\vw$ in $\bigl(L^\infty(0,T; L^{2+\delta}(\dom)),w^\ast \bigr)$.
We will show that the sequence $\langle \Gamma(\rm{v_n}, \vw_n), \bpsi \rangle$ converges to $\langle \Gamma(\rm{v}, \vw), \bpsi \rangle$.

By Proposition \ref{prop-Holder} in view of \eqref{ineq-Holder-assumptions}  we have the following train of equalities or inequalities
\begin{align*}
    &\langle \Gamma(\rm{v_n}, \vw_n), \bpsi \rangle-\langle \Gamma(\rm{v}, \vw), \bpsi \rangle=\langle \Gamma(\rm{v_n}, \vw_n)-\Gamma(\rm{v}, \vw), \bpsi \rangle\\
    &=\langle \Gamma(\rm{v_n}, \vw_n)-\Gamma(\rm{v}, \vw_n), \bpsi \rangle+\langle \Gamma(\rm{v}, \vw_n)-\Gamma(\rm{v}, \vw), \bpsi \rangle\\
    &\; = \int_0^T \int_\dom \left( (\mathrm{v}_n(t,x) - \mathrm{v})\vw_n(t,x)\right)\cdot\bpsi(t,x)\,\dx\,\dt \\
    & \qquad + \int_0^T \int_\dom \left(\mathrm{v}(t,x)(\vw_n(t,x) - \vw(t,x))\right)\cdot\bpsi(t,x)\,\dx\,\dt \\
    &\;\le {\|\mathrm{v}_n - \mathrm{v}\|_{L^\infty(0,T;L^{q/\alpha}(\dom))}}\|\vw_n(t)\|_{L^\infty(0,T; L^{2+\delta}(\dom))}\|\bpsi(t)\|_{L^1(0,T; L^3(\dom))} \\
    & \qquad + \vert \int_0^T \int_\dom \left(\vw_n(t,x) - \vw(t,x)\right)\mathrm{v}(t,x)\cdot\bpsi(t,x)\,\dx\,\dt\vert \qquad := \mathrm{I}^n_1 + \mathrm{I}_2^n.
\end{align*}
Since $\mathrm{v}_n \to \mathrm{v}$ in $C([0,T];L^{q/\alpha}(\dom))$  we infer that $\mathrm{I}_1^n \to 0$ as $n \to \infty$. For the term  $\mathrm{I}_2^n$, let us  observe that  $\mathrm{v} \bpsi \in L^1(0,T; L^{\frac{2+\delta}{1+\delta}}(\dom))$ by Proposition \ref{prop-Holder}. Since  $L^\infty(0,T; L^{2+\delta}(\dom))$ is the dual of $L^1(0,T; L^{\frac{2+\delta}{1+\delta}}(\dom))$ and $\vw_n \to \vw$ weakly$^\ast$ in $L^\infty(0,T; L^{2+\delta}(\dom))$. Therefore, $\mathrm{I}_2^n \to 0$ as $n \to \infty$. Hence,  the result follows.
\end{proof}

\begin{lemma}
\label{lem:gamma2}
Let $\delta > 0$ and $q \geq 1$. The map
\begin{equation}\label{eqn:Gamma hat-7.3}
\widehat\Gamma \colon C([0,T]; L^q(\dom)) \times \bigl(L^\infty(0,T; L^{2+\delta}(\dom)),w^\ast \bigr) \ni (\mathrm{v},\vw) \mapsto \mathrm{v}^{\frac{\delta}{2(2+\delta)}} \vw \in \big(L^\infty(0,T; L^{2}(\dom)), w^\ast\big)
\end{equation}
is {sequentially} continuous.
\end{lemma}
\begin{proof}[Proof of Lemma \ref{lem:gamma2}] Let us choose and fix $\delta > 0$. {It is sufficient to consider the case  $q =1$}. Denote $\beta=\frac{\delta}{2(2+\delta)} \in (0,\frac12)$, i.e. $\frac12=\frac1{2+\delta}+\beta$. As in the proof of Lemma \ref{lem-Gamma} it is sufficient
to consider a modified function $\widehat\Gamma$ defined by
\begin{equation}\label{eqn:Gamma hat-7.3-new}
\widehat\Gamma \colon C([0,T]; L^{1/\beta}(\dom)) \times \bigl(L^\infty(0,T; L^{2+\delta}(\dom)),w^\ast \bigr) \ni (\mathrm{v},\vw) \mapsto \mathrm{v} \vw \in \big(L^\infty(0,T; L^{2}(\dom)), w^\ast\big)
\end{equation}
and prove that this new function $\widehat\Gamma$ is sequentially continuous. First of all let us observe that since $\frac12=\frac1{2+\delta}+\beta$, by Proposition \ref{prop-Holder} the maps $\widehat\Gamma$ is defined. In order to prove that it is {sequentially} continuous,  let consider a sequence $\{\mathrm{v}_n\}$  convergent  to $\mathrm{v}$ in $C([0,T]; L^{1/\beta}(\dom))$ and  a sequence  $\{\vw_n\}$  convergent to $\vw$ in $\bigl(L^\infty(0,T; L^{2+\delta}(\dom)),w^\ast \bigr)$.
Let us also choose $\bpsi \in L^1(0,T; L^2(\dom))$.
 We will show that the sequence $\langle  \widehat\Gamma(\mathrm{v}_n,\vw_n), \bpsi \rangle  $ converges to $\langle \widehat\Gamma(\mathrm{v},\vw),\bpsi \rangle $. We have
\begin{align*}
&\langle  \widehat\Gamma(\mathrm{v}_n,\vw_n), \bpsi \rangle  - \langle \widehat\Gamma(\mathrm{v},\vw),\bpsi \rangle =
\langle  \widehat\Gamma(\mathrm{v}_n,\vw_n)-\widehat\Gamma(\mathrm{v},\vw_n), \bpsi \rangle  + \langle \widehat\Gamma(\mathrm{v},\vw_n) -\widehat\Gamma(\mathrm{v},\vw),\bpsi \rangle
\\
    &\; = \int_0^T \int_\dom \left(\mathrm{v}_n(t,x)\vw_n(t,x) - \mathrm{v}\vw_n(t,x)\right)\cdot\bpsi(t,x)\,\dx\,\dt \\
    & \qquad + \int_0^T \int_\dom \left(\mathrm{v}(t,x)\vw_n(t,x) - \mathrm{v}\vw(t,x)\right)\cdot\bpsi(t,x)\,\dx\,\dt \\
    &\;\le {\|\mathrm{v}_n - \mathrm{v}\|_{L^\infty(0,T;L^{1/\beta}(\dom))}}\|\vw_n(t)\|_{L^\infty(0,T; L^{2+\delta}(\dom))}\|\bpsi(t)\|_{L^1(0,T; L^2(\dom))} \\
    & \qquad + \vert \int_0^T \int_\dom \left(\vw_n(t,x) - \vw(t,x)\right)\mathrm{v}(t,x)\cdot\bpsi(t,x)\,\dx\,\dt \vert \qquad := \mathrm{I}^n_1 + \mathrm{I}_2^n.
\end{align*}
Since $\mathrm{v}_n \to \mathrm{v}$ in $C([0,T];L^{1/\beta}(\dom))$,  $\mathrm{I}_1^n \to 0$ as $n \to \infty$. For $\mathrm{I}_2^n$, we observe that since $\beta+\frac12=\frac{1+\delta}{2+\delta}$ by Proposition \ref{prop-Holder} we infer that  $\mathrm{v} \bpsi \in L^1(0,T; L^{\frac{2+\delta}{1+\delta}}(\dom))$.
Since  $L^\infty(0,T; L^{2+\delta}(\dom))$ is the dual of $L^1(0,T; L^{\frac{2+\delta}{1+\delta}}(\dom))$ and $\vw_n \to \vw$ weakly$^\ast$ in $L^\infty(0,T; L^{2+\delta}(\dom))$, we infer that
 $\mathrm{I}_2^n \to 0$ as $n \to \infty$. Hence, the result follows.
\end{proof}


\section{About some simple consequences of the H\"older inequality}
\label{sec-Holder}

In many places we have used the following simple consequence of the H\"older inequality.
\begin{proposition}\label{prop-Holder}
Assume that real numbers $r\in [1,\infty]$ and $p,q\in [r,\infty]$ are such that
\begin{equation}\label{eqn:exponents}
  \frac1p+\frac1r=\frac1q.
\end{equation}
Assume that $X,Y$ and $Z$ are  Banach spaces and
\begin{equation}\label{eqn:bilinear}
  \beta: X\times Y \to Z \mbox{ is bilinear continuous map}.
\end{equation}
If $\mu$ is a non-negative measure on $(M,\mathcal{M})$, then
\begin{equation}\label{eqn:Holder}
\Vert \beta(f,g) \Vert_{L^q(M,Z)} \leq  \Vert f \Vert_{L^p(M,X)} \Vert g \Vert_{L^r(M,Y)}
\end{equation}
for all elements $f \in L^p(M,X)$ and $g \in L^r(M,Y)$.
\end{proposition}
We have the following two simple corollaries.
\begin{corollary}
\label{cor-Holder}
Assume that $q \in [1,2]$ and $p\in [q, \infty)$.
\[
\frac1p+\frac12=\frac1q.
\]
Then, for appropriate $f$ and $g$   we have
\begin{align}\label{eqn:Holder-2}
\Vert fg \Vert_{L^q(\Omega,L^\infty(0,T;L^{3/2}(\dom))} &\leq  \Vert f \Vert_{L^2(\Omega,L^\infty(0,T;L^6(\dom))} \Vert g \Vert_{L^p(\Omega,L^\infty(0,T;L^2(\dom))},
\\
\label{eqn:Holder-3}
\Vert fg \Vert_{L^q(\Omega,L^2(0,T;L^{3/2}(\dom))} &\leq  \Vert f \Vert_{L^2(\Omega,L^\infty(0,T;L^6(\dom))} \Vert g \Vert_{L^p(\Omega,L^2(0,T;L^2(\dom))},
\\
\label{eqn:Holder-4}
\Vert fg \Vert_{L^q(\Omega,L^\infty(0,T;L^1(\dom))} &\leq  \Vert f \Vert_{L^2(\Omega,L^\infty(0,T;L^2(\dom))} \Vert g \Vert_{L^p(\Omega,L^\infty(0,T;L^2(\dom))}.
\end{align}
\end{corollary}


\section{Stochastic preliminaries}
\label{sec-stoch_pre}

\begin{definition}
\label{def-weak convergence measures}
A sequence of measures $\{\mu_n\}_{n \in \N}$ on $\bigl(E, \mathcal{B}(E)\bigr)$ is said to be weakly convergent to a measure $\mu$ if for every $\phi \in C_b(E)$, we have
\[
\lim_{n \to \infty} \int_E \phi(x) \mu_n(dx) = \int_E \phi(x) \mu(dx).
\]
\end{definition}

\begin{definition}
\label{def-compactness weak measures}
The family $\Lambda$ is said to be compact (respectively relatively compact), if an arbitrary sequence $\{\mu_n\}_{n \in \N}$ of elements from $\Lambda$ contains a subsequence $\{\mu_{n_k}\}_{k \in \N}$ weakly convergent to a measure $\mu \in \Lambda$ (respectively to a measure $\mu$ on $\bigl(E, \mathcal{B}(E)\bigr)$).
\end{definition}

\begin{theorem}[Prokhorov Theorem]
 The family $\Lambda$ of probability measures on $\bigl(E, \mathcal{B}(E)\bigr)$ is relatively compact if and only if it is tight.
\end{theorem}

\begin{remark}
\label{rem-estimates}
Let $X$, $Y$ be two topological spaces and let $\bigl(\Omega, \FT, \mathbb{P}\bigr)$, $\bigl(\tOmega, \tF, \tP\bigr)$ be two probability spaces. Also assume that
\[
\begin{split}
    z_n: \Omega \to X, \qquad \mathrm{v}_n : \Omega \to Y; \\
    \tilde{z}_n: \Omega \to X, \qquad \tilde{\mathrm{v}}_n : \Omega \to Y
\end{split}
\]
are Borel-measurable maps. Let $F: X \times Y \ni (z, \mathrm{v}) \mapsto F(z,\mathrm{v}) \in \R$ is a Borel-measurable map. If $\mathcal{L}\left((z_n, \mathrm{v}_n)\right) = \mu = \tilde{\mathcal{L}}\left((\tilde{z}_n, \tilde{\mathrm{v}}_n)\right)$ on $\mathcal{B}(X\times Y)$, then
\eq{ \label{eqn:equal_estimate}
\int_\Omega F(z_n , \mathrm{v}_n)(\omega)\,d\Prob(\omega) = \int_{\tOmega} F(\tilde{z}_n , \tilde{\mathrm{v}}_n)(\omega)\,d\tP(\omega).
}
\begin{proof}{
Starting from the LHS of \eqref{eqn:equal_estimate}, using the definition of law of a random variable and equivalence of laws on $\mathcal{B}(X \times Y)$, we get
\begin{align*}
    & \int_\Omega F(z_n , \mathrm{v}_n)(\omega)\,d\Prob(\omega) = \int_\Omega F(z_n(\omega) , \mathrm{v}_n(\omega))\,d\Prob(\omega) = \int_{X \times Y} F(x,y)\,d\mathcal{L}(x,y)\\
    &\; = \int_{X \times Y} F(x,y)\,d\mu(x,y) = \int_{X \times Y} F(x,y)\, d\tilde{\mathcal{L}}(x,y)  = \int_{\widetilde{\Omega}} F(\tilde{z}_n(\omega) , \tilde{\mathrm{v}}_n(\omega))\,d\widetilde{\Prob}(\omega)
\end{align*}
}
\end{proof}
\end{remark}

\begin{theorem}[Kuratowski Theorem]
\label{thm-Kuratowski}
Assume that $X_1$, $X_2$ are two Polish spaces with their Borel $\sigma$-fields denoted respectively by $\mathcal{B}(X_1)$, $\mathcal{B}(X_2)$. If $\phi : X_1 \to X_2$ is an injective Borel measurable map, then for any $E_1 \in \mathcal{B}(X_1)$, $E_2 := \phi(E_1) \in \mathcal{B}(X_2)$.
\end{theorem}

We may also need the following corollary \cite[Proposition~C.2]{Brz+Mot+Ondr_2017} of the above result.
\begin{corollary}
\label{cor-Kuratowski}
Suppose that $X_1$, $X_2$ are two topological spaces with their Borel $\sigma$-fields denoted respectively by $\mathcal{B}(X_1)$, $\mathcal{B}(X_2)$. Suppose that $\phi : X_1 \to X_2$ is an injective Borel measurable map such that for any $E_1 \in \mathcal{B}(X_1)$, $E_2 := \phi(E_1) \in \mathcal{B}(X_2)$. Then, if $g : X_1 \to \R$ is a Borel measurable map then a function $f : X_2 \to \R$ defined by
\begin{equation}
    \label{eqn:Kuratowski-1}
f(x_2)=
    \begin{cases}
g(\phi^{-1}(x_2)), \quad \mbox{if } x_2 \in \phi(X_1),\\
\infty, \qquad \qquad \; \; \mbox{if } x_2 \in X_2\setminus \phi(X_1),
    \end{cases}
\end{equation}
is also Borel measurable.
\end{corollary}

The approach to establish H\"older continuity of a stochastic integral then relies on the Kolmogorov continuity theorem,  which statement and proof follow directly from  \cite[Theorem~3.3]{DaPrZa92}.

\begin{theorem}[Kolmogorov Continuity Theorem]
\label{thm-Kolmogorov cont}
Let $X$ be a stochastic process taking values in a separable Banach space $\mathbb{U}$. Assume that there exist  constant $K > 0$, $\nu \ge 1$, $\sigma > 0$ such that
$X(0)=0$ and
\begin{equation}\label{eqn:Kolmogorov-assumption}
\E\|X(t) - X(s)\|^\nu_{\mathbb{U}} \le K|t-s|^{1 + \sigma}, \mbox{ for all $s,t \in [0,T]$.}
\end{equation}
Then there exists $Y$, a modification of $X$, which has $\mathbb{P}$-a.s. H\"older continuous trajectories with exponent $\gamma$ for every $\kappa \in \left(0, \frac{\sigma}{\nu}\right)$. In addition, we have
\begin{equation}\label{eqn:Kolmogorov-asertion}
\E\|Y\|^\nu_{C^\kappa([0,T];\mathbb{U})} \lesssim K,
\end{equation}
where the proportional constant  depends only on  $T$.
\end{theorem}

We use the following consequence  of the Vitali Convergence Theorem  \cite[Theorem C.4]{Oksendal-2003}.

\begin{theorem}[Vitali]\label{thm-Vitali}
Let $(a_N)$ be a sequence of integrable functions on a probability space
$(\Omega,\mathcal{B}(\Omega),\mathbb{P})$ such that $a_N\to a$ a.e.\ as $N\to\infty$
(or $a_N\to a$ in measure) for some integrable function $a$ and there exists $r>1$ and
a constant $C> 0$ such that \[ \E|a_N|^r\le C\mbox{  for all }N\in\N.\] Then
$\E|a_N|\to\E|a|$ as $N\to\infty$.
\end{theorem}


\subsection{Tightness Preliminaries}
\label{sec-Prelim_tightness}

{Let $E$ be a Hausdorff topological  space with the Borel sigma-field denoted by $\mathcal{B}(E)$.}

\begin{definition}
\label{def-tightness}
The family $\Lambda$ of probability measures on $\bigl(E, \mathcal{B}(E)\bigr)$ is said to be tight if for arbitrary $\ep > 0$ there exists a compact set $K_\ep \subset E$ such that
\[
\nu(K_\ep) \ge 1 - \ep \quad \forall\; \nu \in \Lambda.
\]
\end{definition}

We use the following two lemmas, see \cite{Sim87}, to establish the tightness of laws for various processes.

\begin{lemma}\label{lem-compactness Simon-1}
Let $X$, $B$, and $Y$ be Banach spaces such that the embeddings
\[X\embed  B\embed  Y \mbox{  are dense and continuous }\]
 and the embedding \[ X\embed  B  \mbox{ is compact}. \]
 Assume that  $r>1$.
Let $F \subset L^\infty(0,T; X)$ be a bounded set such that the
set

\[ \bigl\{  \frac{\partial u}{\partial t}:  u \in F \}  \subset L^r(0,T; Y) \mbox{ and is bounded}.
\]
Then $F$ is relatively compact in $C([0,T]; B)$.
In other words, the embedding
\[
W^{1,r}(0,T;Y) \cap L^r(0,T;Y) \embed C([0,T]; B) \mbox{ is compact}.
\]
\end{lemma}

\begin{lemma} \label{lem-compactness Simon-2}
Let $X$, $B$, and $Y$ be Banach spaces such that the embeddings
\[X\embed  B\embed  Y \mbox{  are dense and continuous }\]
 and the embedding \[ X\embed  B  \mbox{ is compact}. \]
 Assume that $1 \leq  p< q\leq \infty$ and
\begin{enumerate}
    \item $F$ is uniformly bounded in $L^q(0,T; B)\cap L^1(0,T;X)$;
    \item $\forall\, 0 < t_1 < t_2 < T$,
    \[ \lim_{h\to 0}\; \sup_{f\in F}  \;\|\tau_h f-f\|_{L^1(t_1,t_2;Y)}= 0.
    \]
\end{enumerate}
Then $F$ is relatively compact in $L^p(0,T; B)$.
\end{lemma}

\begin{definition}
\label{def-Aldous criteria}
Let $\left(X_n\right)_{n \in \N}$ be a sequence of stochastic processes indexed by $t \in [0,T]$ and  taking values in a Banach space $E$.
We say that $\left(X_n\right)_{n\in\N}$ satisfies the Aldous condition $\textbf{[A]}$ iff  for every $\epsilon > 0$ and $\eta > 0$ there exists a $\theta > 0$ such that for every sequence $\left(\tau_n\right)_{n \in \N}$ of $[0,T]$-valued stopping times one has
\[
\sup_{n \in \N}\sup_{0\le h \le \theta}\Prob\left\{\|X_n((\tau_n+h)\wedge T) - X_n(\tau_n)\|_E \ge \eta\right\}\le \epsilon.
\]
\end{definition}


\end{document}